\documentclass[11pt, a4paper]{amsart}
\pdfoutput=1 %

\usepackage[T1]{fontenc}

\usepackage[indent]{parskip}

\usepackage[margin=2.5cm, headheight=13pt, headsep=13pt]{geometry}
\usepackage{fancyhdr}
\fancypagestyle{pagetoplr}{%
  \fancyhf{}%
  \fancyhead[RO,LE]{\footnotesize\fancyplain{}{\thepage}}%
  \fancyhead[CO,CE]{\footnotesize\fancyplain{}{\rightmark}}%
}
\usepackage{etoolbox}
\apptocmd{\maketitle}{\thispagestyle{empty}}{}{}

\makeatletter
\patchcmd{\@setaddresses}{\indent}{\noindent}{}{}
\patchcmd{\@setaddresses}{\indent}{\noindent}{}{}
\patchcmd{\@setaddresses}{\indent}{\noindent}{}{}
\patchcmd{\@setaddresses}{\indent}{\noindent}{}{}
\makeatother

\usepackage{microtype}
\usepackage{amsfonts,amssymb}
\usepackage{amsthm,amsmath}
\usepackage{mathtools}
\usepackage{textcomp}
\usepackage{xfrac}
\usepackage{xspace}
\usepackage{bm}
\usepackage[mathscr]{eucal}
\usepackage[inline]{enumitem}

\newlist{paraenum}{enumerate}{1}
\setlist[paraenum]{wide, label=(\arabic*)}
\newlist{inparaenum}{enumerate*}{1}
\setlist[inparaenum]{label=(\arabic*)}
\newlist{paradesc}{description}{1}
\setlist[paradesc]{wide}

\setlist[itemize]{leftmargin=4em, rightmargin=4.5em}
\setlist[enumerate]{leftmargin=4em, rightmargin=4.5em}
\setlist[description]{leftmargin=4em, labelindent=4em, rightmargin=4.5em}

\makeatletter
\g@addto@macro\bfseries{\boldmath}
\makeatother

\usepackage{graphicx}
\usepackage{url}
\usepackage[dvipsnames]{xcolor}

\definecolor{darkred}{RGB}{160,0,0}
\definecolor{darkblue}{RGB}{0,0,160}

\usepackage{hyperref}
\usepackage[sort]{cleveref}
\hypersetup{
  colorlinks,
  citecolor=darkred,
  filecolor=black,
  linkcolor=darkblue,
  urlcolor=darkblue
}

\usepackage[doipre={doi: }]{uri}

\usepackage{stmaryrd}

\newcommand{\CIperp}{\mathrel{\text{$\perp\mkern-10mu\perp$}}}

\usepackage{ifthen}
\usepackage{listofitems}
\newcommand{\CI}[1]{{%
  \setsepchar{{:}/{|}/{,}}
  \ignoreemptyitems
  \readlist*\mylist{#1}
  \ifthenelse{\listlen\mylist[] = 2}{\mylist[2]}{}%
  [\mylist[1,1,1] \CIperp \ifthenelse{\listlen\mylist[1,1] = 2}{\mylist[1,1,2]}{\mylist[1,1,1]}%
  \ifthenelse{\listlen\mylist[1] = 2}{{} \mid \mylist[1,2]}{}]%
}}

\newcommand{\Dsep}[1]{{%
  \setsepchar{{:}/{|}/{,}}
  \ignoreemptyitems
  \readlist*\mylist{#1}
  \ifthenelse{\listlen\mylist[] = 2}{\mylist[2]}{}%
  [\mylist[1,1,1] \perp_d \ifthenelse{\listlen\mylist[1,1] = 2}{\mylist[1,1,2]}{\mylist[1,1,1]}%
  \ifthenelse{\listlen\mylist[1] = 2}{{} \mid \mylist[1,2]}{}]%
}}
\newcommand{\Dconn}[1]{{%
  \setsepchar{{:}/{|}/{,}}
  \ignoreemptyitems
  \readlist*\mylist{#1}
  \ifthenelse{\listlen\mylist[] = 2}{\mylist[2]}{}%
  [\mylist[1,1,1] \not\perp_d \ifthenelse{\listlen\mylist[1,1] = 2}{\mylist[1,1,2]}{\mylist[1,1,1]}%
  \ifthenelse{\listlen\mylist[1] = 2}{{} \mid \mylist[1,2]}{}]%
}}

\usepackage{xstring}
\newcommand{\CIproc}[1]{\StrSubstitute[0]{#1}{|}{\mid}}

\usepackage{xparse}
\newcommand{\CIdOp}{\triangle}

\newcommand{\CIbOp}{\square}
\NewDocumentCommand{\CId}{mG{}}{\CIdOp #2({\CIproc{#1}})}
\NewDocumentCommand{\CIb}{mG{}}{\CIbOp #2({\CIproc{#1}})}

\usepackage{bbm}

\renewcommand{\hat}[1]{\widehat{#1}}
\renewcommand{\tilde}[1]{\widetilde{#1}}

\newcommand{\dual}{\rceil}

\theoremstyle{definition}
\newtheorem{theorem}{Theorem}[section]
\newtheorem*{theorem*}{Theorem}
\newtheorem{lemma}[theorem]{Lemma}

\newtheorem{cor}[theorem]{Corollary}
\newtheorem*{corollary*}{Corollary}

\newtheorem*{fact*}{Fact}
\newtheorem{example}[theorem]{Example}
\newtheorem{definition}[theorem]{Definition}
\newtheorem{observation}[theorem]{Observation}

\newtheorem*{convention*}{Convention}
\newtheorem{open}{Open problem}

\newtheorem*{namedgenthm}{\namedgenthmname}
\newcounter{namedgenthm}
\makeatletter
\newenvironment{genthm*}[1]{%
  \def\namedgenthmname{#1}%
  \refstepcounter{namedgenthm}%
  \namedgenthm\def\@currentlabel{#1}%
  \cref@constructprefix{namedgenthm}{\cref@result}%
  \protected@edef\@currentlabel{#1}%
  \protected@edef\@currentlabelname{genthm}%
  \protected@edef\cref@currentlabel{%
    [namedgenthm][][\cref@result]%
    #1%
  }%
}
{\endnamedgenthm}
\makeatother
\crefname{namedgenthm}{}{}
\Crefname{namedgenthm}{}{}

\newcommand{\Macaulay}[1]{%
  \ifthenelse{\equal{\detokenize{#1}}{\detokenize{2}}}{}%
  {\PackageError{tboege-preprint}{Always use \protect\Macaulay2, never \protect\Macaulay\space alone}{}}%
  \texttt{Macaulay#1}\xspace}

\newcommand{\TodoColor}[2]{\colorlet{TodoColor#1}{#2}}
\TodoColor{Default}{darkred}

\NewDocumentEnvironment{TodoList}{oO{TODO\IfNoValueF{#1}{~(#1)}}+b}{%
  \begingroup%
  \noindent\sffamily\color{TodoColor\IfNoValueTF{#1}{Default}{#1}}%
  \textbf{#2:}
  \begin{itemize}
  #3
  \end{itemize}%
  \endgroup%
}{}

\NewDocumentCommand{\todo}{oO{TODO\IfNoValueF{#1}{~(#1)}}+m}{%
  \begingroup%
  \noindent\sffamily\color{TodoColor\IfNoValueTF{#1}{Default}{#1}}%
  \textbf{#2:} #3%
  \endgroup%
  \xspace%
}

\usepackage{comment}
\usepackage{appendix}

\usepackage{multibib}
\newcites{soft}{Mathematical software}

\def\cal#1{\mathcal{#1}}
\def\dv{\mathbb}
\def\ci{\perp\!\!\!\perp} %
\def\sta{{\cal K}} %
\def\gsta{{\cal T}} %
\def\staext{{\cal K}^{\bullet}}  %
\def\gstaext{{\cal T}^{\bullet}} %
\def\bij{\varphi} %
\def\altbij{\psi} %
\def\bijelse{\rho} %
\def\mod{{\cal M}} %
\def\altmod{{\cal R}} %
\def\modext{{\cal E}} %
\def\dual{\star} %
\def\ffam{\mathfrak{f}} %
\def\profam{\mathfrak{pr}} %
\def\semgr{\mathfrak{sg}} %
\def\strum{\mathfrak{st}} %
\def\gra{\mathfrak{gr}} %
\def\cogr{\mathfrak{cgr}} %
\def\ungr{\mathfrak{ug}} %
\def\sa{\diamond} %
\def\lift{\nearrow} %

\def\Sat{\mathsf{SAT}}

\def\sfX{\mbox{\sf X}}
\def\isfX{\mbox{\scriptsize\sf X}}
\def\bxi{\mbox{\boldmath$\xi$}}
\def\boeta{\mbox{\boldmath$\eta$}}
\def\bzeta{\mbox{\boldmath$\zeta$}}
\def\ibxi{\mbox{\scriptsize\boldmath$\xi$}}
\def\iboeta{\mbox{\scriptsize\boldmath$\eta$}}
\def\ibzeta{\mbox{\scriptsize\boldmath$\zeta$}}
\def\latt{{\cal L}}
\def\irre{{\cal I}}
\def\cover{\prec\!\!\!\cdot\,\,\,}
\def\cl{\mbox{\rm cl}}
\def\calF{{\cal F}}
\def\id{\mbox{\rm id}}
\def\rel{\odot}

\title%
[Self-adhesivity in lattices of abstract conditional independence models]%
{Self-adhesivity in lattices of \\%
abstract conditional independence models}

\author{Tobias Boege}
\address[Tobias Boege]{Department of Mathematics, KTH Royal Institute of Technology, Stockholm, Sweden}
\email{post@taboege.de}

\author{Janneke H.\ Bolt}
\address[Janneke H.\ Bolt]{Department of Information and Computing Sciences, Utrecht University, The Netherlands \& Faculty of Science, Open University of the Netherlands, Heerlen, The Netherlands}
\email{j.h.bolt@uu.nl}

\author{Milan Studen\'{y}}
\address[Milan Studen\'{y}]{Institute of Information Theory and Automation of the CAS, Prague, Czech Republic}
\email{studeny@utia.cas.cz}


\date{\today}

\begin{document}

\begin{abstract}
We introduce an algebraic concept of the frame for abstract
{\em conditional independence\/} (CI) models, together with basic
operations with respect to which such a frame should be closed: copying
and marginalization. Three standard examples of such frames are
(discrete) probabilistic CI structures, semi-graphoids and structural
semi-graphoids. We concentrate on those frames which are closed under the operation of set-theoretical intersection because, for these, the respective families of CI models are lattices. This allows one to apply the results from lattice theory and formal concept analysis to describe such families in terms of implications
among CI statements.

The central concept of this paper is that of {\em self-adhesivity\/} defined in algebraic terms, which is a combinatorial reflection of the self-adhesivity concept studied earlier in context of polymatroids and
information theory. The generalization also leads to a self-adhesivity operator defined on the meta-level of CI frames. We answer some of the questions related to this approach and raise other open questions.

The core of the paper is in computations. The combinatorial approach to computation might overcome some memory and space limitation of software packages based on polyhedral geometry, in particular, if SAT solvers are utilized.
We characterize some basic CI families over 4 variables in terms of canonical implications among CI statements. We apply our method in information-theoretical context to the task of entropic region demarcation over 5 variables.
\end{abstract}

\maketitle

\section{Introduction}
This paper is devoted to the concept of {\em self-adhesivity\/} in the context of abstract
{\em conditional independence\/} (CI) structures, which
generalize/approximate classic probabilistic CI structures \cite{Stu05}. The concepts of adhesivity and self-adhesivity have been introduced by Mat\'{u}\v{s} \cite{Mat07DM} in the context of {\em polymatroids}, which are special set functions playing a crucial role in combinatorial optimization \cite{Fuj05}. Mat\'{u}\v{s}'s motivation
came from information theory  \cite{Yeu08}, more specifically, from the area of information-theoretical inequalities.
The self-adhesivity concept for polymatroids is a kind of abstraction of the method used to derive the first
non-Shannon information inequality \cite{ZY98}, which method later became known in information-theoretical community
as the so-called {\em copy lemma} \cite{DFZ11}. The method led to the application of tools of polyhedral geometry
(and linear programming solvers) to derive new inequalities \cite{Csi20Kyb}.
\smallskip

In this paper we introduce the concept of self-adhesivity for abstract CI structures, which
is a combinatorial reflection of the self-adhesivity concept for polymatroids. The idea comes from \cite{Boe22} and has already been applied in the context of Gaussian CI structures \cite{Boe23}. Nonetheless, our paper puts
even more emphasis on the abstract view on (the families of) CI models and utilizes the results of {\em lattice theory} \cite{Bir95} for this purpose. More specifically, every finite lattice can alternatively be described  in terms of a closure operation, or in terms of (abstract) functional dependence relation \cite{Mat91TCS}. Thus, one can utilize the results from the theory of {\em formal concept analysis} \cite{GW99}. This leads to an analogy with mathematical logic and to the effort to describe families of CI models in an axiomatic way, using implications among CI statements.
\smallskip

We introduce the notion of an abstract frame for CI structures and give several specific examples of such CI frames. We also introduce relevant algebraic operations with respect to which these abstract CI frames should be closed: these are copying and marginalization.
From the point of view of our paper, the substantial algebraic operation is the set-theoretical
{\em intersection}, which ensures that, for a fixed (finite) set of variables, the family of respective CI models
forms a finite lattice. Then we introduce the concept of self-adhesivity, which is relative to a CI frame.
The self-adhesivity concept induces special {\em self-adhesivity operator} of the meta-level of CI frames. This operator
assigns a new CI frame to a starting CI frame and this new frame is a shrinkage of the original frame.
In case the starting CI frame is an outer approximation of the classic probabilistic CI frame, the application of the
self-adhesivity operator yields an even tighter outer approximation of the probabilistic CI frame.
\smallskip

This leads to a method to derive stronger valid probabilistic implications among CI statements on basis of previously known valid implications. The point of the story is that newly derived implications over $n$ variables are obtained on basis of known implications over a higher number of variables than $n$ through self-adhesivity.
\smallskip

The core of the paper are computations. Their goal was the description of several lattices of abstract CI models over a low number of variables. There are two methodologically different options: the lattices can
\begin{itemize}
\item either be described by listing meet-irreducible CI models,
\item or in terms of implications among CI statements.
\end{itemize}
Our approach also allows one to derive such CI implications directly, by means of performing the computations
on the combinatorial level of CI frames. To~contrast with Matúš's original concept of self-adhesivity for polymatroids, note that the dimension of the polymatroidal cone grows exponentially in $n$ and the number of CI statements grows even faster. Still, our computational methods
based on SAT solvers can deal with this increase much better than the linear programming solvers used in the polymatroid setup.
\smallskip

In the paper, a specific example of the application of our combinatorial
method is given. We~applied the method to the question of identification of those extreme rays of the polymatroidal cone over $5$ variables which are {\em not entropic},
which means that they do not have probabilistic representation. It~appears that most of the 1319 permutational types of these extreme rays lie outside the entropic region. More specifically, it follows from our computations that at most 154 of types of such extreme rays may be entropic. The proof is entirely computational and takes about 330 seconds to execute on the author's laptop.
Throughout the paper we also formulate remaining open mathematical questions related to our approach.
\medskip

The structure of the paper is as follows. In Section~\ref{sec.basic}, we define our concept of an abstract CI frame together with basic algebraic operations applied to CI models.
We also recall those particular concepts and results from lattice theory on which our algorithms are based
and which allow one to introduce frame-based closures of CI models.
In Section~\ref{sec.examples-frames} we present examples of such abstract CI frames.
The concept of self-adhesion relative to such an abstract CI frame is
introduced in Section~\ref{sec.adhesivity}, together with related basic observations and two other
relevant algebraic operations applied to CI models.
Section~\ref{sec.iteration} is devoted to the self-adhesivity operator and the idea of its iterating.
In Section~\ref{sec.transition} we review our tools for computations.
The results of our computational experiments are presented in Section~\ref{sec.catalogues} and their
application in information-theoretical context is explained in Section~\ref{sec.entro-region}.
Concluding remarks are presented in Section~\ref{sec.additional}. More detailed information about our computations, source code and resulting data files can be found on the website:
\begin{center}
  \url{https://cinet.link/data/selfadhe-lattices}\,\,.
\end{center}
\medskip
\pagebreak %

\section{Basic concepts}\label{sec.basic}
In this section we introduce both our notational conventions and terminology.
Recall that {\bf CI} is an abbreviation for {\em conditional independence}.
\smallskip

Given a %
set $X$, the symbol ${\cal P}(X):=\{S:S\subseteq X\}$
will denote its {\em power set\/}, that is, the set %
of all subsets of $X$. The upper case letters like $K,L,M,N,O$ will denote
finite sets of {\em variables\/} of interest.
In our intended probabilistic interpretation, their elements
correspond to random variables.
Additionally, two special notational conventions will be used.
Given a variable $i\in N$, the symbol $i$ will also denote the singleton set $\{i\}$
and, given two variable sets $K$ and $L$, the juxtaposition $KL$ will be a shortened notation
for their union $K\cup L$.
Given a bijection $\bij:N\to M$ and $L\subseteq N$ the restriction of $\bij$ to $L$
will be denoted by $\bij_{|L}$, the identical bijection of $L$ onto itself by $\id_{L}$.

\subsection{Abstract families of CI models and the probabilistic CI frame}\label{ssec.CI-families}

An %
elementary {\em CI statement\/} over a finite variable set
$N$ is the statement claiming that, for some $K\subseteq N$ and distinct $i,j\in N\setminus K$,
the variable $i$ is conditionally independent of the variable $j$ given %
$K$, traditionally denoted  by $i\ci j\,|\,K$. A converse conditional {\em dependence
statement\/} (= negation) is the claim that $i$ and $j$ are conditionally dependent given $K$,
traditionally denoted by $i\not\ci j\,|\,K$. Thus, a triplet $\langle i,j|K\rangle$
of pairwise disjoint subsets of $N$ can be interpreted either as an independence statement
or as a dependence statement.

Because, in this paper, we do not deal with ``mixed'' CI structures, involving both dependence and independence statements,
we accept a convention that our sketchy notation $ij|K$ for a triplet
$\langle i,j|K\rangle$ will be interpreted as the CI statement $i\ci j\,|\,K$.
Moreover, in our paper we consider only {\em symmetric\/} %
CI concepts in the sense that $i\ci j\,|\,K$ is equivalent to $j\ci i\,|\,K$.
Every such a ``symmetric'' CI statement can mathematically be described by an ordered pair $\,L\,|\,K\,$ of sets $L,K\subseteq N$, where $L\cap K=\emptyset$ and $|L|=2$; specifically we have $L=\{i,j\}$.
So, our second convention is that $ij|K$ means the same as $ji|K$.

Note that this way of introducing CI statements implicitly means that any CI statement over
a variable set $N$ is automatically also a CI statement over any variable set $M$ with $N\subseteq M$.

\begin{definition}[abstract CI statements, models, families, frames]\label{def.elem-models}~\rm \\
Given a finite variable set $N$, the symbol $\sta(N)$ will denote the
collection of all ordered pairs $\,L\,|\,K\,$, where $L,K\subseteq N$, $L\cap K=\emptyset$, and $|L|=2$.
\begin{itemize}[leftmargin=2em, rightmargin=1em]
\item The elements of\/ $\sta(N)$ will be interpreted as (CI) {\em statements\/} over $N$.
They will usually be denoted
by  explicit writing: $ij|K\in\sta(N)$.
\item A subset of\/ $\sta(N)$ will be interpreted as a (CI) {\em model\/} over $N$. Models will
be denoted using the calligraphic font: $\mod\subseteq\sta(N)$ or $\mod\in {\cal P}(\sta(N))$.
\item A class of (CI) models over $N$ involving the (largest possible) model $\sta(N)$ will
be a (CI) {\em family\/} over~$N$.
Thus, a family is a subset of the power set of\/ $\sta(N)$.
Families will be denoted using the fractal font and will have the variable sets as their arguments: $\sta(N)\in\ffam(N)\subseteq{\cal P}(\sta(N))$.
\end{itemize}
By a (CI) {\em frame\/} we will understand a mapping %
which assigns a (CI) family over $N$ to each finite (variable) set $N$.  Frames will be denoted using the fractal
font without an argument: mapping
$\ffam:\,N\mapsto \ffam(N)\in {\cal P}({\cal P}(\sta(N)))$. %
Given two CI frames $\ffam_{1}$ and $\ffam_{2}$, we write
$\ffam_{1}\subseteq\ffam_{2}$ to mean $\ffam_{1}(N)\subseteq\ffam_{2}(N)$ for any (finite) variable set $N$.
\end{definition}

The reader can recognize a clear hierarchy: a (CI) statement is an element of\/ $\sta(N)$, a model is an element of the power set of\/ $\sta(N)$ and a family is an element of the power set of the power set of\/ $\sta(N)$.\footnote{A frame corresponds to what was called a ``property'' in the PhD thesis \cite[\S\,2.2]{Boe22}.}

Note that $\sta(N)=\emptyset$ if $N$ is empty or a singleton. In particular,
only one (CI) model over $N$ exists if $|N|\leq 1$, namely the empty set/model.
Moreover, only one (CI) family over $N$ exists if $|N|\leq 1$ because of our explicit requirement that $\sta(N)\in\ffam(N)$ in the definition of a family.

The definition of a CI~model is relative to a variable set~$N$ and, strictly speaking, the model is a pair $(N, \mod)$ with $\mod \subseteq \sta(N)$. For example, the empty set might be considered as a CI~model over $\{a,b\}$ or over $\{a,b,c\}$. To simplify notation, we will usually rely on context to identify the variable set or mention it explicitly by means of the ``frame notation'' $\mod \in \ffam(N)$. If two variable sets are in inclusion $N \subseteq M$, then any CI~model $\mod$ over $N$ may be regarded as a model over $M$ as well. In the context of a frame, this ability to transition to larger variable sets is formalized by the operation of \emph{ascetic extension} defined below.
\smallskip

\begin{example}\rm\label{exa.least-base-set}
Put $N:=\{a,b,c,d\}$ and $\mod:=\{\, ab|\emptyset, bc|\emptyset\,\}\subseteq\sta(N)$.
Then $N^{\mod}=\{a,b,c\}$ and $\mod$ is a CI model over this strict subset of $N$.
Nonetheless, $\mod$ is not a CI model over $\{a,b\}$ nor over $\{b,c,e\}$, where $e$ is a variable outside $N$. On the other hand, $\mod$ is a model over any variable set $M$ with $\{a,b,c\}\subseteq M$.
\end{example}

We now give a prototypical example of an abstract CI frame, which is the frame of  discrete probabilistic CI structures.
A {\em discrete random vector\/} over a variable set $N$ is an indexed collection $\bxi=[\xi_{i}]_{i\in N}$ of random variables, each $\xi_{i}$ taking its values in an individual finite sample space $\sfX_{i}$,
which is necessarily non-empty then.
The joint sample space for $\bxi$ is $\sfX_{N}:=\prod_{i\in N} \sfX_{i}$.
Given $\bxi$ and $I\subseteq N$, the respective {\em marginal density\/} of\/ $\bxi$
can be defined as function $p_{I}:\sfX_{N}\to [0,1]$ such that the value $p_{I}(x)$ at
$x=[x_{i}]_{i\in N}\in\sfX_{N}$ is the probability of having $\xi_{i}=x_{i}$
for each $i\in I$.
Given $ij|K\in\sta(N)$, we have $i\ci j\,|\,K\,\,[\bxi]$ if
$$
\forall\, x\in\sfX_{N}\qquad
p_{ijK}(x)\cdot p_{K}(x) = p_{iK}(x)\cdot p_{jK}(x)\,.
$$
The %
CI model induced by $\bxi$ is the set
$\mod_{\ibxi} := \{\, ij|K\in\sta(N) \,:\ i\ci j\,|\,K\,\,[\bxi]\,\}$.
The family of {\em discrete probabilistic models\/} over $N$ is then
$$
\profam\, (N) \,:=\, \{\, \mod_{\ibxi} \,:\ \bxi ~~
\mbox{\rm is a discrete random vector over $N$}\,\}\,.
$$
Observe that $\sta(N)$ belongs to $\profam(N)$ because it is induced by a random vector with
completely independent components.
This defines the {\em discrete probabilistic\/} CI frame $\profam$.
\medskip

It will be convenient to deal with CI statements with more variables in place of the singletons $i$ and $j$. Although the treatment of (probabilistic) CI statements can be extended to
this more general case, in this paper the notation $I\ci J | K$ for disjoint subsets of $N$ will be considered only as an abbreviation.

\begin{definition}[global CI statement]\label{def.global-CI}~\rm \\[0.3ex]
Let $N$ be a variable set and $I,J,K\subseteq N$ pairwise disjoint subsets of it.
Given a CI model $\mod\subseteq\sta(N)$ over $N$, we will write $I\ci J\,|\,K\,\,[\mod]$
if
\begin{equation}
 \forall\, i\in I\quad \forall\, j\in J\quad\forall\, L : K\subseteq L\subseteq IJK\setminus ij
\qquad ij|L\in\mod\,.
\label{eq.glob-CI}
\end{equation}
We interpret this as a global CI statement encoded within the CI model $\mod$. In our examples and catalogues we also use a shorthand $[I,J|K]$ to denote the list of (elementary) CI statements described in \eqref{eq.glob-CI}.
\end{definition}

Note that, in the context of the probabilistic frame $\profam$, the abbreviating condition $I\ci J\,|\,K\,\,[\mod_{\ibxi}]$ is equivalent to the requirement that
$p_{IJK}(x)\cdot p_{K}(x) = p_{IK}(x)\cdot p_{JK}(x)$
holds for every $x\in\sfX_{N}$, which requirement is
interpreted as the (global) probabilistic CI statement $I\ci J\,|\,K\,\,[\bxi]$; see \cite[Lemma\,2.2]{Stu05}.
\medskip

The next step is to introduce five algebraic operations on abstract
CI models. The goal is to study various CI frames which are closed under such
operations.
\pagebreak %

\begin{definition}[copying, marginalization, intersection, duality, ascetic extension]\label{def.operations}~\rm
\begin{itemize}[leftmargin=2em, rightmargin=1em]
\item Any bijection %
$\bij : N\to M$ between
two variable sets $N$ and $M$ can naturally be extended to a bijection
between their power sets and to a bijection between $\sta(N)$ and $\sta(M)$:
put $\bij(ij|K)\,:=\, \bij(ij)|\bij (K)$ for every $ij|K\in \sta (N)$.

\item Any model $\mod\subseteq\sta(N)$ over $N$ is then transformed by
$\bij$ to a model over $M$: put
$$
\bij(\mod) :=\{\, \bij(ij|K)\,:\ ij|K\in\mod\,\}\subseteq\sta (M)\,.
$$
This is the {\em $\bij$-copy\/} of $\mod$.
We say that a CI frame $\ffam$ is closed under {\em copying\/} if, for every model $\mod\in\ffam(N)$
and every bijection $\bij : N\to M$, one has %
$\bij(\mod)\in\ffam(M)$.

\item Given a model $\mod\subseteq\sta(N)$ over $N$ and\, $M\subseteq N$, its {\em marginal\/}
for $M$, denoted by $\mod^{\downarrow M}$, is simply its restriction
to $\sta(M)$: put $\mod^{\downarrow M}:= \mod\cap \sta(M)$.
A CI frame $\ffam$ is closed under {\em marginalization\/} if, for every model $\mod\in\ffam(N)$
and every subset $M\subseteq N$, one has $\mod^{\downarrow M}\in\ffam(M)$.

\item Given two models $\mod,\altmod\subseteq\sta (N)$ over the same set $N$, their
(set-theoretical) {\em intersection\/} is the model $\mod\cap\altmod$, again a model over $N$.
A CI frame $\ffam$ will be called closed under {\em intersection\/} if, for
every two models $\mod,\altmod\in\ffam(N)$, one has $\mod\cap\altmod\in\ffam(N)$.

\item Given a variable set $N$, the duality map from $\sta(N)$ to $\sta(N)$
is as follows: $ij|K\in\sta(N)$ is assigned the pair $(ij|K)^{\dual}:=ij|L$,
where $L\,:=\,N\setminus (K\cup ij)$.
Given a model $\mod\subseteq\sta(N)$, the {\em dual model\/} $\mod^{\dual}$
is its image by this mapping.
A CI frame $\ffam$ is closed under {\em duality\/} if $\mod\in\ffam(N)$
implies $\mod^{\dual}\in\ffam(N)$.

\item Given a model $\mod\subseteq\sta(N)$ over $N$ and $N\subseteq M$, an {\em extension\/} of $\mod$ to $M$ is any model $\altmod\subseteq\sta (M)$ over $M$ such that $\altmod^{\downarrow N}=\mod$. A CI frame $\ffam$ is closed under {\em extension\/} if, for every model $\mod\in\ffam(N)$ and every $M$ with $M\supseteq N$, there exists at least one extension $\altmod$ of $\mod$ to $M$ with $\altmod\in\ffam(M)$.

\item Given $\mod\subseteq\sta(N)$ and $N\subseteq M$, an extension
$\altmod\subseteq\sta(M)$ of $\mod$ to $M$ will be called {\em ascetic\/} if
$\altmod\setminus\mod=\emptyset$. A CI frame $\ffam$ is closed under {\em ascetic extension\/} if, for every model $\mod\in\ffam(N)$ and every variable set $M$ with  $M\supseteq N$, one has
$\mod\in\ffam(M)$, that is, the ascetic extension of $\mod$ to $M$ belongs to $\ffam(M)$.

\end{itemize}
\end{definition}

Since $\sta(N)=\emptyset$ if $|N|\leq 1$, the concepts of copying and duality
make real sense only in case $|N|\geq 2$. On the other hand, to introduce
self-adhesion in Section~\ref{sec.adhesivity} one needs marginalization to sets $M$ with $|M|\leq 1$.

Note that if $N\subseteq M$ then every model $\mod$ over $N$ has an ascetic extension to $M$, namely $\mod$ itself, viewed as a CI model over $M$. Clearly, this is the least possible extension of $\mod$ to $M$ and, thus, unique.
The word ``ascetic" was chosen because, in comparison with other extending
algebraic operations, it adds nothing to $\mod$.
Clearly, a CI frame $\ffam$ is closed under ascetic extension iff $N\subseteq M$
implies $\ffam(N)\subseteq\ffam(M)$.
We leave it to the reader as an easy exercise to show that, if a CI frame $\ffam$ is closed under extension and intersection, then it is closed under ascetic extension iff $N\subseteq M$ implies $\sta(N)\in\ffam(M)$.

Two basic algebraic operations with abstract CI models are copying and marginalization.
Most of the considered CI frames in this paper are also closed under intersection, which implies that the respective families of models are lattices and one can introduce frame-based closures for CI models (see Section~\ref{ssec.frame-closure}).
\medskip

Our prototypical example of a CI frame is closed under all but one algebraic operations considered.

\begin{lemma}\label{lem.pr-frame}\rm
The frame $\profam$ is closed under copying, marginalization, intersection, and ascetic extension.
\end{lemma}

\begin{proof}
Given a discrete random vector $\bxi=[\xi_{i}]_{i\in N}$ and a bijection $\bij : N\to M$, the random vector
$\boeta:=[\xi_{\,\bij_{-1}(j)}]_{j\in M}$ induces
the $\bij$-copy of\/ $\mod_{\ibxi}$: this is because $i\ci j\,|\,K\,\,[\bxi]$ if and only if
$\bij(i)\ci \bij(j)\,|\,\bij (K)\,\,[\boeta]$.
Given $\bxi=[\xi_{i}]_{i\in N}$ and $M\subseteq N$, the random sub-vector
$\bxi_{M}:=[\xi_{i}]_{i\in M}$ induces the marginal model $(\mod_{\ibxi})^{\downarrow M}$.

Let $\bxi^{1}=[\xi^{1}_{i}]_{i\in N}$ and $\bxi^{2}=[\xi^{2}_{i}]_{i\in N}$ be discrete random vectors
over $N$. Take a bijection $\bij : N\to M$ with $N\cap M=\emptyset$ and construct an extended
random vector $\bxi=[\xi_{i}]_{i\in NM}$
with stochastically independent sub-vectors $\bxi_{N}:=\bxi^{1}=[\xi^{1}_{i}]_{i\in N}$ and $\bxi_{M}:=[\xi^{2}_{\,\bij_{-1}(j)}]_{j\in M}$. Define
a random vector $\boeta=[\eta_{i}]_{i\in N}$ over $N$ where each
$\eta_{i}:=[\xi_{i},\xi_{\bij(i)}]$ is a composite random variable.
The construction allows one to observe that, for any $ij|K\in\sta(N)$, one has
$i\ci j\,|\,K\,\,[\boeta]$ if and only if both $i\ci j\,|\,K\,\,[\bxi^{1}]$ and $i\ci j\,|\,K\,\,[\bxi^{2}]$.
This gives $\mod_{\iboeta}=\mod_{\ibxi^{1}}\cap \mod_{\ibxi^{2}}$.

To show that $\profam$ is closed under extension consider $N\subseteq M$
and a discrete random vector $\bxi=[\xi_{i}]_{i\in N}$ and
extend it to a random vector $\boeta=[\eta_{i}]_{i\in M}$
with stochastically independent sub-vectors $\boeta_{N}:=\bxi$ and
$\eta_{i}$, $i\in M\setminus N$. Because we already know that $\profam$ is closed under intersection, to show that it is closed under ascetic extension
it remains to show that, if $N\subseteq M$, then $\sta(N)\in\profam(M)$.
This leads to a construction of a random vector $\boeta=[\eta_{i}]_{i\in M}$
with stochastically independent components $\eta_{i}$, $i\in N$,
while there is no other CI relation within $\boeta$. A specific construction of such a random vector is left to the reader.
\end{proof}

On the other hand, the CI frame $\profam$ is not closed under duality. Recall a counter-example by Mat\'{u}\v{s} from  \cite[end of \S\,5]{Mat95CPC}.

\begin{example}\rm\label{exa.non-dual}
Given $N:=\{a,b,c,d\}$ and $\mod:=\{\, ab|cd,\, ab|d,\, ab|c,\, cd|\emptyset\,\}$, one has
$\mod\in \profam\, (N)$ and $\mod^{\dual}=\{\, ab|\emptyset,\, ab|c,\, ab|d,\, cd|ab\,\}\not\in \profam\, (N)$. Indeed, a binary random vector uniformly distributed on
four special configurations
$$
(0,0,0,0)\quad (0,1,0,1)\quad (0,1,1,0)\quad (1,1,1,1)
$$
of values for $(x_{a},x_{b},x_{c},x_{d})$ implies the former fact.
The latter fact follows from the property (I:13) in \cite[\S\,V.B]{Stu21} saying that the implication
$$
[\,\, a\ci b\,|\,\emptyset ~~\&~~ a\ci b\,|\,c ~~\&~~
a\ci b\,|\,d ~~\&~~ c\ci d\,|\,ab\,\,] \quad\Rightarrow\quad a\ci b\,|\,cd\
$$
holds in the context of discrete probabilistic CI structures.
\end{example}

\noindent {\em Remark.} This is to elucidate possible alternative ways to formalize CI models. The classic way of understanding probabilistic CI structures from \cite{Pea88}, followed in \cite{Stu05}, was to define an abstract CI model over $N$ as the collection of global CI statements over $N$ valid with respect to a random vector $\bxi$ over $N$.
This approach leads to an alternative global mode, where
$\sta(N)$ is replaced by the collection $\gsta(N)$ of all triplets
$\langle I,J|K\rangle$ of pairwise disjoint subsets $I,J,K$ of $N$
and (global) CI models are subsets of $\gsta(N)$. Such an alternative approach is equivalent to the one discussed in our paper
through the relation from Definition~\ref{def.global-CI}.

Non-equivalent descriptions of probabilistic CI structures are those that additionally deal with {\em functional dependencies among random variables}, as the one from \cite{MatStu95CPC}. Given a discrete random vector $\bxi$ over $N$, $K\subseteq N$, and $i\in N\setminus K$, the respective functional dependence statement is the claim that $\xi_{i}$ is distribution-equivalent to a function
of $\bxi_{K}:=[\xi_{j}]_{j\in K}$. The point is that any such a functional dependence statement within $\bxi$ can be interpreted as a generalized CI statement $i\ci i\,|\,K\,\,[\bxi]$, namely that $i$ is conditionally independent of itself given $K$.
The way to formalize such wider CI structures is to replace\/ $\sta(N)$ by the collection $\staext(N)$ of ordered pairs $\,L\,|\,K\,$, where $L,K\subseteq N$, $L\cap K=\emptyset$, and $1\leq |L|\leq 2$. Any pair $i\,|\,K$ then corresponds to a functional dependence statement claiming that $i\in N\setminus K$ is ``determined'' by $K\subseteq N$. Subsets of $\staext(N)$ were called {\em augmented\/} CI models over $N$ \cite[\S\,II.B]{Stu21}, %
represented in the local/elementary mode. An equivalent version of this would be to replace $\staext(N)$ by the collection $\gstaext(N)$ of all triplets $\langle I,J|K\rangle$ of subsets $I,J,K$ of $N$. The subsets of\/ $\gstaext(N)$ could be viewed as  augmented CI models, represented in the global mode. However, we do not consider augmented CI~models with functional dependence statements in this paper.

\subsection{Tutorial on Galois connections and formal concept analysis}\label{ssec.tutorial}
This part contains an overview of particular results from lattice theory
and formal concept analysis we later apply in our context of abstract CI frames (see Section~\ref{ssec.frame-closure}) to establish frame-based implications of
CI models. %
We have included such a summary in our paper because a common reader need not
be familiar with these finer results, perhaps well-known in lattice-theoretical
community, but not in general. The reason is that the consistency of our
algorithms (\Cref{app.canon-basis}) follows just from these facts.
\smallskip

The main message of this tutorial to be passed towards the reader is that
any finite lattice, possibly with a huge number of elements, can (efficiently) be characterized in two methodologically different ways:
\begin{itemize}
\item in terms of irreducible elements (= an inner description in the sense that the characterization only depends on the lattice itself), %
\item in terms of implicational generators (= an outer description defined in case
the lattice is specifically embedded into a Boolean lattice).
\end{itemize}
Below we describe both approaches and recall relevant concepts from lattice theory \cite{Bir95} for this purpose.
\smallskip

Throughout this tutorial, $X$ will be a general set.
Facts from lattice theory hold for any such a set $X$. Finer concepts, which concern the implications between subsets of $X$, and results from formal concept analysis implicitly assume that $X$ is finite. We  are going to apply those observations in Section~\ref{ssec.frame-closure} in the case $X:=\sta(N)$ for some finite variable set $N$.
\medskip

Given a set $X$, a {\em Moore family\/} of subsets of $X$, alternatively called a {\em closure system\/} (for $X$), is a collection $\calF\subseteq {\cal P}(X)$ of subsets of $X$ which is closed under (arbitrary) set intersection and includes the set $X$ itself\/:
$$
{\cal D}\subseteq\calF ~\Rightarrow ~ \bigcap{\cal D}\in\calF\quad
\mbox{with a convention that\, $\bigcap{\cal D}=X$ for ${\cal D}=\emptyset$}.
$$
An elementary observation is that any Moore family $(\calF ,\subseteq)$,
equipped with the set inclusion relation, %
is a complete lattice \cite[Theorem\,2 in \S\,V.1]{Bir95} in which
the meet operation is intersection.\footnote{This is a universal construction of a complete lattice in the sense that any such a lattice is isomorphic to a Moore family and any finite lattice is isomorphic to a finite Moore family lattice. These conclusions can be derived from observations in \cite[\S\,V.9]{Bir95}.}
Hence, the following simple facts can be derived.

\begin{observation}\label{obs.model-lattice.1}\rm
If a CI frame $\ffam$ is closed under intersection then, for each finite $N$,
the family $(\ffam(N),\subseteq)$ is a finite lattice in which the meet operation
is intersection. %
If $\ffam$ is also closed under copying and $|N|=|M|$ then $(\ffam(N),\subseteq)$ and $(\ffam(M),\subseteq)$
are order-isomorphic and permutations of $N$ yield automorphisms of $(\ffam(N),\subseteq)$.
\end{observation}

\begin{proof}
By Definition~\ref{def.operations}, as $\ffam(N)$ is a Moore family, this is evident from what was said before.
\end{proof}

Let $(\latt,\preceq)$ be a (finite) lattice with the greatest element denoted by $\mbox{\bf 1}$ and the meet/infimum operator denoted  by $\inf$. Let $\prec$\, denote the strict order relation.
Given $\lambda,\tau\in\latt$, one says that $\tau$ {\em covers\/} $\lambda$ and writes
$\lambda\cover\tau$ if $\lambda\prec\tau$ and there is no $\kappa\in\latt$ with $\lambda\prec\kappa\prec\tau\/$.
A {\em coatom\/} of $\latt$ is a sub-maximal element in $\latt$, that is, $\tau\in\latt$ satisfying
$\tau\cover \mbox{\bf 1}$.
An element $\lambda\in\latt$ is called
(meet) {\em irreducible\/} if $\lambda\neq \inf\, \{\, \tau\in\latt\, :\ \lambda\prec\tau\,\}$.\footnote{Note that $\mbox{\bf 1}$ is not irreducible by this definition: it is the infimum of the empty set.}
A well-known fact is that the set $\irre$ of (meet) irreducible elements in a finite lattice $(\latt , \preceq)$ is the least set determining all elements of $\latt$ by means of the infimum operation in the following sense:
$$
\forall\, \kappa\in\latt
\quad \exists\, {\cal S}\subseteq\irre\qquad
\mbox{such that}~~ \kappa = \inf\, \{\, \tau\, :\ \tau\in {\cal S}\,\}\,.
$$
The coatoms are standard examples of meet irreducible elements in $\latt$.
In general, a lattice can have other irreducible elements.
A lattice $(\latt ,\preceq)$ is called {\em coatomistic\/} if every element of $\latt$ is
the infimum of a set of coatoms, which is another way of saying that there is no irreducible element in $\latt$ except the coatoms.

\begin{example}\rm\label{exa.irr-co}
Put $X:=\{x,y,z,w\}$ and consider a Moore family
$$
\calF:=\{\,\, \{x,y,z,w\},\{x,y,z\},\{x,y,w\},\{x,y\},\{x\} \,\,\}\,,
$$
of subsets of $X$, depicted in Figure \ref{lattice-irr-co} in the form of its Hasse diagram. The elements $\{x,y,z\}$, $\{x,y,w\}$ and $\{x\}$ of the lattice $(\calF,\subseteq)$ are then irreducible, but only the elements $\{x,y,z\}$ and $\{x,y,w\}$ are its coatoms. Thus, $(\calF,\subseteq)$ is a finite lattice which is not coatomistic.
\end{example}

\begin{figure}[t]
\centering
\includegraphics[scale=0.6]{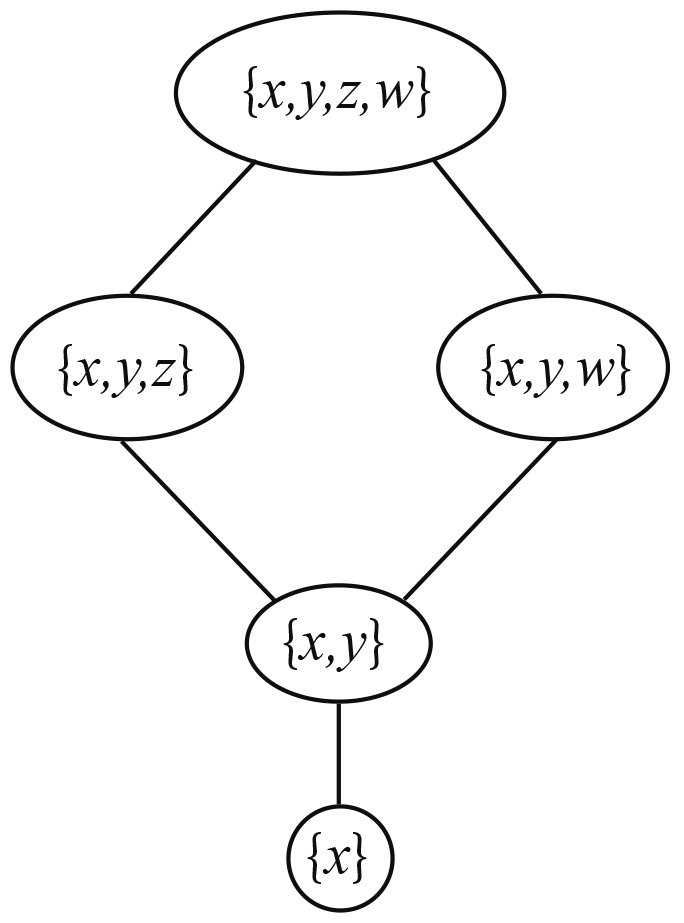}
\caption{A Moore family $\{\, \{x,y,z,w\},\{x,y,z\},\{x,y,w\},\{x,y\},\{x\} \,\}$.
}\label{lattice-irr-co}
\end{figure}

The first way to characterize a finite lattice $(\calF,\subseteq)$, where $\calF$ is a Moore family of subsets of $X$, is to give the complete list of irreducible elements in $\calF$. In our considered case they fall into equivalence classes:

\begin{observation}\label{obs.model-lattice.2}\rm
If a CI frame $\ffam$ is closed under copying and intersection then, for each finite variable set $N$, the meet irreducible sets in $(\ffam(N),\subseteq)$ (= irreducible models) fall into permutational equivalence classes
(= relative to permutations of $N$).
\end{observation}

\begin{proof}
By Observation~\ref{obs.model-lattice.1}, any permutation $\pi:N\to N$ defines an order automorphism of $(\ffam(N),\subseteq)$, the natural extension of $\pi$
to\/ $\sta(N)$. Such a transformation preserves the irreducibility of an element of\/ $\ffam(N)$.
\end{proof}

\begin{example}\rm\label{exa.pr-3-var}(lattice of probabilistic models over 3 variables)\\
Let us take $N:=\{a,b,c\}$, $X:=\sta(N)$, and put
$\calF:=\profam\, (N)$.
Either a human-made analysis or simple computations allows one to observe that the lattice $(\profam\, (N),\subseteq)$ has 22 elements in this case and they break into 10 permutational equivalence classes; for the result see \cite[Figure~5.6]{Stu05}.
There are 5 (meet) irreducible elements in
$\profam\, (\{a,b,c\})$  which fall down into 3 permutational~types:
\begin{itemize}
\item $\{\, ab|\emptyset,\, ac|\emptyset,\, bc|\emptyset\,\}$,
\item $\{\, ab|c,\, ac|b,\, bc|a\,\}$,
\item $\{\, ab|\emptyset,\, ab|c,\, ac|\emptyset,\, ac|b\,\}$ (with 2 additional permutational copies).
\end{itemize}
All these elements are coatoms; in particular, the lattice
is coatomistic.
\end{example}
\medskip

A kind of dual description of a complete lattice is based on the following concept.
A {\em closure operation\/} on subsets of $X$ is
a mapping $\cl :{\cal P}(X)\to {\cal P}(X)$ which is
\begin{itemize}
\item {\em extensive\/}: $Y\subseteq \cl\,(Y)$ for $Y\subseteq X$,
\item {\em isotone\/}: $Y\subseteq Z\subseteq X \Rightarrow \cl\,(Y)\subseteq \cl\,(Z)$, and
\item {\em idempotent\/}: $\cl\, (\cl\,(Y))=\cl\,(Y)$ for $Y\subseteq X$.
\end{itemize}
A set $Y\subseteq X$ is then called {\em closed with respect to $\cl$} if $Y=\cl\, (Y)$.
A basic facts are that the collection of subsets of $X$ closed
with respect to $\cl$ is a Moore family and any Moore family can be defined in this way \cite[Theorem\,1 in \S\,V.1]{Bir95}. Specifically, given a Moore family $\calF$ of subsets of $X$, the formula
$$
\cl_{\calF}(Y) \,:=\, \bigcap\, \{\,Z\subseteq X\,:\, Y\subseteq Z\in\calF\,\} \quad \mbox{\rm for} ~ Y\subseteq X\,,
$$
defines a closure operation on subsets of $X$ having $\calF$ as the collection of closed sets with respect to $\cl_{\calF}$; see \cite[Theorem\,1 in \S\,0.3]{GW99} for the proof of this simple observation.
Thus, for any $Y\subseteq X$, the closure $\cl_{\calF}(Y)$ of $Y$ relative to $\calF$ is the smallest element of $\calF$ containing $Y$.

\begin{example}\rm
Consider the Moore family $\calF$ from Example \ref{exa.irr-co}. Then, for instance, $\cl_{\calF}(\emptyset)=\{x\}=\cl_{\calF}(\{x\})$, $\cl_{\calF}(\{y\})=\{x,y\}$, and $\cl_{\calF}(\{x,z\})=\{x,y,z\}$.
\end{example}

The closure operation can also be interpreted as an {\em abstract functional dependence\/} relation on subsets of $X$, as explained in \cite{Mat91TCS}. This is a binary relation $\to$
on ${\cal P}(X)$ satisfying three axioms (for $Y,Z,W\subseteq X$):
\begin{itemize}
\item {\em reflexivity\/}: $Z\subseteq Y ~\Rightarrow~ Y\to Z$,
\item {\em transitivity\/}: $[\,Y\to Z ~\&~ Z\to W\,]  ~\Rightarrow~ Y\to W$,
\item {\em union\/}: $[\,Y\to Z ~\&~ Y\to W\,] ~\Rightarrow~ Y\to Z\cup W$.
\end{itemize}
In case of finite $X$, these axioms are together equivalent to the well-known {\em Armstrong axioms} for functional dependence relations in database theory.
Given a closure operation $\cl$ on subsets of a finite set $X$, define $Y\to Z$ as $Z\subseteq \cl\,(Y)$ and observe it is an abstract functional dependence relation on subsets of $X$.
Given a functional dependence relation $\to$, put
$\cl\,(Y):= \bigcup\, \{\,Z\subseteq X \,:\ Y\to Z\,\}$ for $Y\subseteq X$ and observe this defines a closure operation on subsets of $X$ inducing the original functional dependence relation $\to$.

\begin{example}\rm
Let us consider again the Moore family $\calF$ from Example \ref{exa.irr-co} and its induced closure operation $Y\mapsto\cl_{\calF}(Y)$ for $Y\subseteq X$. The respective
functional dependence relation is quite huge. It certainly
contains 16 implications saying that subsets of $X$ imply their closures:
\begin{eqnarray*}
\lefteqn{\hspace*{-14mm}\emptyset\to \{x\}, ~~ \{x\}\to \{x\}, ~~ \{y\}\to \{x,y\}, ~~ \{z\}\to \{x,y,z\}, ~~
 \{w\}\to \{x,y,w\},}\\
 ~\ldots && ~~\{y,z,w\}\to \{x,y,z,w\}, ~~ \{x,y,z,w\}\to \{x,y,z,w\}\,.
\end{eqnarray*}
Besides these, it also contains their natural consequences because of inclusion: for instance, it contains the implication $\{y,z,w\}\to \{x,w\}$ which follows from the inclusion $\{x,w\}\subseteq \cl_{\calF}(\{y,z,w\})$ (use reflexivity and transitivity for this).
\end{example}

\noindent {\em Remark.} To avoid possible terminological misunderstanding on the side of the reader note that functional dependencies within a random vector $\bxi$ mentioned in the remark concluding Section~\ref{ssec.CI-families} can be viewed as a part of an abstract {\em probabilistic functional dependence relation}.
Specifically, given a random vector $\bxi$ over a finite variable set $N$, the collection of pairs $K\to I$, where $K,I\subseteq N$, $K\cap I=\emptyset$,
and $I\ci I\,|\,K\,\,[\bxi]$ (augmented global CI statements)
is an abstract functional dependence relation on subsets of $X:=N$.
One of the results in \cite[Remark\,\,4]{Mat91TCS} is that every abstract functional dependence relation on subsets of a finite $N$ is probabilistic in the above sense. So, these notions are indeed related.

Nonetheless, in our paper we ignore probabilistic functional dependencies within $N$.
Instead, we deal with functional dependence on a meta-level of CI statements/models
over $N$: we consider abstract functional dependence relations on subsets of $X:=\sta(N)$,
where $N$ is a finite variable set. To emphasize the different context, we also talk about ``implications" between subsets of $X$, which are CI models over $N$.
\medskip

The second way to characterize a finite lattice $(\calF,\subseteq)$, where $\calF$ is a Moore family of subsets of $X$, is to give a suitable generator of the respective functional dependence relation $\to$.
Before~explaining the details we wish the reader to realize/observe that there is a straightforward relationship (=~correspondence) between the functional dependence relations $\to$ and Moore families $\calF$.\footnote{A one-to-one
assignment between Moore families and functional dependence relations has already been established through their shared closure operations. We now re-interpret this as a kind of duality relationship and avoid the closure operations.}
\smallskip

Specifically, any element $(Y,Z)$ of ${\cal P}(X)\times {\cal P}(X)$, where $X$ is a finite set, can be interpreted as a formal implication $Y\to Z$, meaning that the set $Y$ of {\em antecedents\/} implies (all elements of) the set $Z$ of {\em consequents}. One can introduce a
particular binary relation $\rel$ between such pairs $(Y,Z)$ and subsets $F$ of $X$:
we say that {\em $Y\to Z$ is valid for $F\subseteq X$} if
$Y\subseteq F$ implies $Z\subseteq F$. More technically:
\begin{eqnarray*}
\lefteqn{\hspace*{-12mm}\forall\, (Y\to Z)\in {\cal P}(X)\times {\cal P}(X)~~ \forall\, F\in {\cal P}(X)}\\[0.2ex]
&&(Y\to Z)\rel F \quad:=\quad
[\,\, Y\setminus F\neq\emptyset ~~\mbox{or}~~ Z\subseteq F\,\,].
\end{eqnarray*}
Next we introduce the so-called {\em Galois connections} between subsets of ${\cal P}(X)\times {\cal P}(X)$ and subsets of ${\cal P}(X)$:
\begin{eqnarray*}
&\!\!{\dv G}\subseteq {\cal P}(X)\times {\cal P}(X) \,\mapsto\,
{\dv G}^{\triangleright}:=\{\,F\in {\cal P}(X):\,
\forall\, (Y\to Z)\in {\dv G}~\, (Y\to Z)\rel F\,\},& \\
&\calF\subseteq {\cal P}(X) \,\mapsto\,
{\calF}^{\triangleleft}:=\{\,(Y\to Z)\in {\cal P}(X)\times {\cal P}(X):\,
\forall\, F\in\calF ~\, (Y\to Z)\rel F\,\}\,.&
\end{eqnarray*}

Well-known classic results in lattice theory on Galois connections from
\cite[\S\,V.7]{Bir95}, completed by particular observations from
\cite[(unnumbered) Theorem]{Mat91TCS}, say the following in case of a finite set~$X$:
\begin{itemize}[leftmargin=2em, rightmargin=1em]
\item the mapping ${\dv G}\mapsto {\dv G}^{\triangleright\triangleleft}$ is a closure operation on subsets of the set ${\cal P}(X)\times {\cal P}(X)$ and ${\dv G}^{\triangleright\triangleleft}$ is the smallest functional dependence relation (on subsets of $X$) containing~${\dv G}$,

\item the mapping $\calF\mapsto {\calF}^{\triangleleft\triangleright}$ is a closure operation on subsets of ${\cal P}(X)$ and ${\calF}^{\triangleleft\triangleright}$ is the
smallest Moore family (of subsets of $X$) containing $\calF$,

\item a system  $\calF\subseteq {\cal P}(X)$ is a Moore family if and only if
$[\,\exists\,{\dv G}\subseteq {\cal P}(X)\times {\cal P}(X) \,:\ \calF= {\dv G}^{\triangleright}\, ]$.
\end{itemize}
We say that ${\dv G}\subseteq {\cal P}(X)\times {\cal P}(X)$ is an {\em implicational generator\/}
of a Moore family $\calF$ if $\calF= {\dv G}^{\triangleright}$. Elementary facts on Galois
connections allow one to observe that an equivalent condition is ${\calF}^{\triangleleft}={\dv G}^{\triangleright\triangleleft}$. Hence, ${\dv G}$ is an implicational generator
of a Moore family $\calF$ if and only if
both ${\dv G}\subseteq {\calF}^{\triangleleft}$ (soundness) and $|{\dv G}^{\triangleright}|=|{\calF}|$ (completeness).

\begin{example}\rm\label{exa.impli-gen}
Consider $X:=\{x,y,z,w\}$ and put
$$
{\dv G}:=
\{~ \emptyset\to\{x\}, ~\{z\}\to\{y\}, ~\{x,w\}\to\{x,y,w\} ~\}\,.
$$
Then ${\dv G}^{\triangleright}$ equals to the family $\calF$ from Example \ref{exa.irr-co}. Hence, ${\dv G}$ is one of several possible implicational generators of $\calF$, a minimal one in sense of inclusion (= a non-redundant one).
\end{example}

To introduce a prominent implicational generator one needs specific concepts
from the theory of formal concept analysis; see \cite[Chapter\, 3]{GO16} or \cite[\S\,1.3.2]{Vil21}.\footnote{In general, there is no unique smallest implicational generator is sense of inclusion, that is, a non-redundant generator.}

\begin{definition}[pseudo-closed sets, canonical basis]\label{def.pseudo-closed}~\rm \\[0.3ex]
Given a finite set $X$, a subset $Q\subseteq X$ is {\em pseudo-closed\/} relative to a Moore family $\calF\subseteq {\cal P}(X)$ if
\begin{itemize}
\item $Q\not\in\calF$ (= $Q$ is not closed),
\item $\forall\, R\subset Q ~\mbox{\rm pseudo-closed}~~ \cl_{\calF}(R)\subseteq Q$\quad (then $\cl_{\calF}(R)\subset Q$ because
$Q\not\in\calF$).
\end{itemize}
The class of pseudo-closed subsets relative to $\calF$ will be denoted by $\wp(\calF)$.
The {\em canonical implicational basis\/} for a Moore family $\calF$ is as follows:
$$
{\dv G}_{\calF} ~:=~ \{\, (\,Q\to \cl_{\calF}(Q)\setminus Q\,)\, :\ Q\in \wp(\calF)\,\}\,\subseteq\, {\cal P}(X)\times {\cal P}(X)\,.
$$
\end{definition}
\smallskip

The recursive definition of pseudo-closedness of a set $Q$ is well-formed as it depends only on the previously defined pseudo-closedness of strict subsets of~$Q$. The canonical implicational basis, alternatively
named Guigues--Duquenne
basis, was introduced already in the 1980s \cite{GD86}. We have also implemented an algorithm to compute the canonical basis ${\dv G}_{\cal F}$ for a given Moore family ${\cal F}$; its idea is described in \Cref{app.canon-basis}.
\smallskip

Every inclusion-minimal non-closed subset, that is, $Q\in ({\cal P}(X)\setminus\calF)_{\min}$, is always pseudo-closed
since all proper subsets of $Q$ are closed. The converse implication does not hold as the following example shows.

\begin{example}\rm\label{exa.impli-basis}
Consider again the Moore family $\calF$ from Example \ref{exa.irr-co} and its closure operation $\cl_{\calF}$.
Then one has $\wp(\calF)=\{\,\, \emptyset,\, \{x,z\},\, \{x,w\} \,\,\}$. Note in this context that none of the singleton subsets of $X$ is pseudo-closed because of $\cl_{\calF}(\emptyset)=\{x\}$. The empty set $\emptyset$ is the only inclusion-minimal non-closed set while $\{x,z\}$ and $\{x,w\}$ are additional pseudo-closed sets.
Specifically, one has
$$
{\dv G}_{\calF} ~=~\{~ \emptyset\to \{x\},\,
\{x,z\}\to \{y\},\, \{x,w\}\to \{y\} ~\}\,.
$$
Nevertheless, further simplification is possible.
Because of\/ $\emptyset\to \{x\}$, the implication
$\{x,w\}\to \{y\}$ is equivalent to $\{w\}\to \{y\}$ (use
axioms of functional dependence relations to observe that).
To summarize: the simplest implicational generator for the
family $\calF$ seems to be the following list:
$$
\{~ \emptyset\to \{x\},\,
\{z\}\to \{y\},\, \{w\}\to \{y\} ~\}\,.
$$
\end{example}

The widely-known result in formal concept analysis says that $({\dv G}_{\calF})^{\triangleright}=\calF$ for any Moore family $\calF\subseteq {\cal P}(X)$; see \cite[Theorem\,7 in \S\,3.2]{GO16} for a relatively
short proof. In fact, ${\dv G}_{\calF}$ is
a non-redundant implicational generator of $\calF$ in sense that the removal of any of its elements leads to a non-generator.
\smallskip

Let us introduce further specific terminology concerning formal implications.
Given a Moore family $\calF\subseteq {\cal P}(X)$, a {\em valid implication for $\calF$} is $(Y\to Z)\in {\cal F}^{\triangleleft}$.
An implication $Y\to Z$ is {\em perfect\/} (for $\calF$)
if $Y\not\in\calF$ but every strict subset of $Y$ is closed, that is, $Y\in ({\cal P}(X)\setminus\calF)_{\min}$,
and $Z=\cl_{\calF}(Y)\setminus Y$. The implication is {\em canonical\/} if it belongs to the canonical implicational basis~${\dv G}_{\calF}$.

It is immediate that every perfect implication is canonical but the converse does not hold as shown in Example~\ref{exa.impli-basis}, where $\{x,z\}\to \{y\}$ is a canonical implication which is not perfect.

\begin{definition}[implicatively perfect Moore family]\label{def.perfect}~\rm \\[0.3ex]
We say that a Moore family $\calF\subseteq {\cal P}(X)$ is {\em implicatively perfect\/} if the set of perfect implications for $\calF$ is an implicational generator of $\calF$.
\end{definition}

It makes no problem to deduce from non-redundancy of the canonical basis
that an equivalent condition is that every pseudo-closed set is an inclusion-minimal non-closed set, that is, an element of set collection $({\cal P}(X)\setminus\calF)_{\min}$. Let us illustrate the above defined notions by an example.

\begin{example}\rm\label{exa.pr-var-b} (probabilistic models over 3 variables revisited)\\
Consider again the situation from Example~\ref{exa.pr-3-var}: $N:=\{a,b,c\}$, $X:=\sta(N)$, and $\calF:=\profam\, (N)$.
There are six pseudo-closed sets relative to $\calF$ belonging to one permutational type: the set $\{\, ab|\emptyset,\, ac|b \,\}$
and five of its permutational copies. Each of these sets is minimal in ${\cal P}(X)\setminus\calF$; thus, $\calF$ is implicatively
perfect. The canonical basis ${\dv G}_{\calF}$ consists of permutational copies of the implication
$$
\{\, ab|\emptyset,\, ac|b \,\} \to \{\, ac|\emptyset,\, ab|c \,\}\,.
$$
The characterization of this particular Moore family $\calF$ in terms of its canonical implicational basis is equivalent to a later definition of a semi-graphoid over $\{a,b,c\}$ in Section~\ref{sec.sem-frame}.
\end{example}

\noindent {\em Remark.} The concept from Definition~\ref{def.perfect} can be viewed as a counter-part of the one of a coatomistic lattice in the context of implicational description of the lattice. Realize that the Galois connection $\calF\mapsto\calF^{\triangleleft}$, for Moore families $\calF$ of subsets of a finite set $X$, can be viewed as an anti-isomorphism between such Moore families and abstract functional dependence relations ${\dv G}$,
whose inverse is the Galois connection ${\dv G}\mapsto{\dv G}^{\triangleright}$. We can interpret this map as a kind of self-inverse duality transformation between these mathematical structures, which both can be viewed as
alternative (ways of) descriptions of finite lattices.
If one prefers the perspective of a Moore family then the simplest lattice is the
coatomistic one but if one prefers the perspective of implicational description then
the simplest lattice is the one described by perfect implications.

\subsection{Frame-based closure of a CI model and implications between CI models}\label{ssec.frame-closure}
Let us come back to the situation in Section~\ref{ssec.CI-families} and consider an abstract CI frame $\ffam$ closed under intersection. In this case, for every finite variable set $N$, the family of CI models $(\ffam(N),\subseteq)$ is a Moore family of subsets of $X:=\sta(N)$. Therefore, the closure operation $\cl_{\ffam(N)}$ on subsets of\/ $\sta(N)$, respectively the collection of
{\em valid implications for\/ $\ffam(N)$}, characterizes this family.
Observe that, for $\mod,\altmod\subseteq\sta(N)$, the validity of the implication $\mod\to\altmod$ for $\ffam(N)$ reduces to the condition
$$
\forall\, \modext\in\ffam(N)\quad \mod\subseteq\modext ~\mbox{implies}~  \altmod\subseteq\modext\,.
$$
In the sequel, we intend to switch freely between the description of a CI frame $\ffam$ in terms of the families $\ffam(N)$ and in terms of the corresponding closure operations.
\smallskip

The fact that the elements of the canonical basis in Example~\ref{exa.pr-var-b} fall down into permutational types  is not a surprise because we have the following.

\begin{observation}\label{obs.model-lattice.3}\rm
If a CI frame $\ffam$ is closed under copying and intersection then, for each set $N$, the valid implications $\mod\to\altmod$ for the Moore family $\calF:=\ffam(N)$ and the elements of the canonical implicational basis for $\ffam(N)$ fall into permutational equivalence classes.
\end{observation}

\begin{proof}
The arguments are analogous to those in Observation~\ref{obs.model-lattice.2}.
Any \mbox{permutation} $\pi:N\to N$ defines a simultaneous order automorphism of the posets $(\ffam(N),\subseteq)$ and $({\cal P}(\sta(N)),\subseteq)$. Because such an order automorphism of $(\calF,\subseteq)$ commutes with the closure operation $\cl_{\calF}$, it preserves the validity of implications.
\end{proof}

We typically wish to have the frame-based closure
stable relative to an enlargement of the basic set.

\begin{observation}\label{obs.model-lattice.4}\rm
Let $\ffam$ be a CI frame closed under marginalization, intersection, and ascetic extension.
Given $\mod\subseteq\sta(N)$ and $N\subseteq M$ one has $\cl_{\ffam(N)}(\mod)=\cl_{\ffam(M)}(\mod)$.
\end{observation}

\begin{proof}
The fact that $\ffam$ is closed under marginalization  implies $\cl_{\ffam(N)}(\mod)\subseteq\cl_{\ffam(M)}(\mod)$.
Since $\ffam(N)\subseteq\ffam(M)$ (ascetic extension), one has $\cl_{\ffam(N)}(\mod)\in\ffam(M)$, which yields the equality.
\end{proof}

Therefore, the assumption on a CI frame $\ffam$ being closed under ascetic extension makes the following definition of the (frame-based) $\ffam$-closure of a CI model $\mod$ consistent, because the closure does not depend on the ``context", by which we mean a variable set $M$
such that $\mod\subseteq\sta(M)$.

\begin{definition}[frame-based closure of a CI model]\label{def.short-closure}~\rm \\[0.3ex]
Consider a CI frame $\ffam$ closed under marginalization,
intersection, and ascetic extension.
Given a CI model $\mod\subseteq\sta(N)$ over a variable set $N$, its closure of relative to $\ffam(N)$ will be denoted by $\ffam(\mod):=\cl_{\ffam(N)}(\mod)$.
We will call it the {\em $\ffam$-closure\/} of the model $\mod$.
\end{definition}

Indeed, by Observation~\ref{obs.model-lattice.4}, if $\mod\subseteq\sta(N)$ then
both $\mod$ and $\ffam(\mod)$ are models over $N^{\mod}:=\bigcup_{ij|K\in\mod} ijK$. Also, if the assumptions of Definition~\ref{def.short-closure} are fulfilled and $\mod\subseteq\sta(N)$, $\altmod\subseteq\sta(M)$ are given then the condition $\mod\to\altmod$ for $\ffam(NM)$ enforces $N^{\altmod}\subseteq N^{\mod}$ and its verification
reduces to showing that $\altmod\subseteq \ffam(\mod)$.

\section{Examples of abstract CI frames}\label{sec.examples-frames}
This is an overview of CI frames; some of these were topics of our computational experiments.

\subsection{Semi-graphoids}\label{sec.sem-frame}
The concept of a semi-graphoid was introduced by Pearl \cite{Pea88}
with the aim to formalize fundamental properties of probabilistic CI
pinpointed earlier by Dawid \cite{Daw79JRSS}. Here we give an equivalent
definition from \cite[\S\,5]{Mat92conf} adapted to our situation.
A {\em semi-graphoid\/} over a variable set $N$ is a subset $\mod$ of\/ $\sta(N)$ satisfying
$$
 ij|K,\, i\ell |jK\in\mod  ~\Leftrightarrow~  i\ell|K,\, ij|\ell K\in\mod\,.
$$
Then we put
$$
\semgr\, (N) \,:=\, \{\, \mod\subseteq\sta(N) \,:\ \mbox{\rm $\mod$ is a semi-graphoid over $N$}\,\}\,,
$$
which defines the {\em semi-graphoidal\/} CI frame $\semgr$.

\begin{lemma}\label{lem.sg-frame}\rm
$\semgr$ is closed under copying, marginalization, intersection, duality, and ascetic extension.
\end{lemma}

\begin{proof}
This basically follows from the above definition of a semi-graphoid.
\end{proof}

Well-known facts are that $\profam\subseteq\semgr$ and $\semgr(N)=\profam(N)$ if $|N|\leq 3$.
The simplest example showing that $\semgr(N)\neq\profam(N)$ in case $|N|\geq 4$ is as follows.

\begin{example}\rm\label{exa.sg-not-pr}
Given $N:=\{a,b,c,d\}$ and $\mod:=\{\, ab|c,\, ac|d,\, ad|b\,\}\in\semgr(N)$, one has $\mod\not\in\profam(N)$ by property (E:1) in \cite[\S\,V.A]{Stu21} saying that the equivalence
$$
[\,\, a\ci b\,|\,c ~~\&~~ a\ci c\,|\,d ~~\&~~
a\ci d\,|\,b \,\,] ~\Leftrightarrow~
[\,\, a\ci b\,|\,d ~~\&~~ a\ci c\,|\,b ~~\&~~
a\ci d\,|\,c \,\,]
$$
holds in the context of discrete probabilistic CI structures.
\end{example}

\subsection{Structural semi-graphoids}\label{sec.struc-sem}
The concept of a structural semi-graphoid was introduced in \cite{Stu94IJGS}
as a semi-graphoid induced in a particular way by certain subsets of ${\dv Z}^{{\cal P}(N)}$.
The original complicated definition was shown in \cite[Corollary\,5.3]{Stu05} to be equivalent
to a much simpler condition that underlies our definition below.

\begin{definition}[supermodular function]\label{def.supermodular}~\rm \\[0.3ex]
Given a set function $m:{\cal P}(N)\to {\dv R}$ and arbitrary subsets $I,J,K\subseteq N$,
we are going to use the following special notation for a particular value:
$$
\Delta m\, \langle I,J|K\rangle ~:=~ m(IJK)+m(K)-m(IK)-m(JK)\,.
$$
The {\em supermodular difference expression\/} for $ij|K\in\sta(N)$  is a special case:
$$
\Delta m (ij|K) := \Delta m\, \langle i,j|K\rangle = m(ijK)+m(K)-m(iK)-m(jK)\,.
$$
The function $m$ is called {\em supermodular\/} if $\Delta m (ij|K)\geq 0$
for any $ij|K\in\sta(N)$, which requirement is equivalent to %
$\Delta m\, \langle I,J|K\rangle\geq 0$ for pairwise disjoint subsets $I,J,K\subseteq N$,
or to the well-known standard supermodularity condition
$$
m (A\cup B)+m(A\cap B)\geq m(A) + m(B)\qquad \mbox{for any $A,B\subseteq N$.}
$$
\end{definition}

The model $\mod^{m} := \{\, ij|K\in\sta(N) \,:\ \Delta m (ij|K)=0\,\}$,
where $m:{\cal P}(N)\to {\dv R}$
is supermodular,
is the structural semi-graphoid induced by $m$.
The class of {\em structural semi-graphoids\/} over $N$ is
$$
\strum\, (N) \,:=\, \{\, \mod^{m} \,:\ ~\mbox{\rm $m$ is a supermodular function on ${\cal P}(N)$}\,\}\,.
$$
This defines the structural CI frame $\strum$.

\begin{lemma}\label{lem.st-frame}\rm
$\strum$ is closed under copying, marginalization, intersection, duality, and ascetic extension.
\end{lemma}

\begin{proof}
Given a bijection $\bij : N\to M$ and $m:{\cal P}(N)\to {\dv R}$, one has
$\Delta m (ij|K) = \Delta (m\circ\bij_{-1})\, \bij(ij|K)$ for any $ij|K\in\sta(N)$,
where $\bij$ is naturally extended to the power set of $N$ and $\sta(N)$.
Hence, provided that $m$ is a supermodular function on ${\cal P}(N)$, the composition $m\circ \bij_{-1}$ is a supermodular function on ${\cal P}(M)$ and the $\bij$-copy of $\mod^{m}$ is $\mod^{m\circ\bij_{-1}}$.

If $M\subseteq N$ then the restriction of a supermodular function on ${\cal P}(N)$ to ${\cal P}(M)$ is supermodular, which implies the claim about the marginalization.

Given two set functions $m^{1}:{\cal P}(N)\to {\dv R}$ and $m^{2}:{\cal P}(N)\to {\dv R}$, one has $\Delta (m^{1}+m^{2}) (ij|K) = \Delta m^{1} (ij|K) + \Delta m^{2} (ij|K)$ for any $ij|K\in\sta(N)$, which implies that
if $m^{1}$ and $m^{2}$ are supermodular then
$m:=m^{1}+m^{2}$ is supermodular. The non-negativity of values $\Delta m^{1} (ij|K)$ and $\Delta m^{2} (ij|K)$ allows one to deduce that $\mod^{m}=\mod^{m^{1}}\cap \mod^{m^{2}}$.

Consider a self-inverse mapping $\iota:{\cal P}(N)\to {\cal P}(N)$ defined by $\iota(S):=N\setminus S$ for any $S\subseteq N$. The formula $\Delta m (ij|K) = \Delta (m\circ\iota) (ij|K)^{\dual}$ valid for
any $m:{\cal P}(N)\to {\dv R}$ and $ij|K\in\sta(N)$, implies that
$m$ is supermodular if and only if $m\circ\iota$ is supermodular and that
$(\mod^{m})^{\dual}=\mod^{m\circ\iota}$.

To show that $\strum$ is closed under extension, realize that, if $N\subseteq M$, then any supermodular
function $r:{\cal P}(N)\to {\dv R}$ can be extended to a supermodular function $m$ on ${\cal P}(M)$ by the formula
$m(T):=r(T\cap N)$ for any $T\subseteq M$. Since $\strum$ is closed under intersection, to show it is closed
under ascetic extension it remains to observe that $\sta(N)\in\strum(M)$. This follows from $\profam\subseteq\strum$;
a specific construction is left to the reader.
\end{proof}

Basic facts are that $\profam\subseteq\strum\subseteq\semgr$ and $\semgr(N)=\strum(N)=\profam(N)$ if $|N|\leq 3$;
for $\profam\subseteq\strum$ see \cite[\S\,2.3.4]{Stu05}. The strict inclusion $\strum(N)\subset\semgr(N)$
in case $|N|=4$ is evidenced by the model $\mod:=\{\, ab|c,\, ac|d,\, ad|b\,\}$ %
from Example~\ref{exa.sg-not-pr}, because the property (E:1) mentioned there
holds for structural semi-graphoids.

The strict inclusion $\profam(N)\subset\strum(N)$ for $|N|=4$ is evidenced by
$\mod^{\dual}:=\{\, ab|\emptyset,\, ab|c,\, ab|d,\, cd|ab\,\}$, which is the dual model
in Example~\ref{exa.non-dual}, induced by the next supermodular function $m$ on $N=\{a,b,c,d\}$:
$$
m(S) ~:=~
\left\{
\begin{array}{cl}
12 & \mbox{if $S=abcd$,} \\
7 & \mbox{if $S=acd$ or $S=bcd$,}\\
6 & \mbox{if $S=abc$ or $S=abd$,} \\
3 & \mbox{if $S\in\{\, ac,ad,bc,bd,cd\,\}$,}\\
0 & \mbox{otherwise,}
\end{array}
\right.
~~\qquad
\mbox{for $S\subseteq N$.}
$$

\noindent {\em Remarks.}~ This is to remind alternative views on structural semi-graphoids.
\begin{enumerate}[leftmargin=2em, rightmargin=1em]
\item \label{remarks:2.1} A set function $h:{\cal P}(N)\to {\dv R}$ is {\em submodular\/} if it fulfils the converse
inequalities to the supermodular ones\,:
$$
h (A\cup B)+ h(A\cap B)\leq h(A) + h(B)\qquad \mbox{for any $A,B\subseteq N$,}
$$
which is another way of saying that $-h$ is supermodular. The rank function of a {\em polymatroid\/} over $N$ is a submodular function $h$ on ${\cal P}(N)$ which satisfies $h(\emptyset)=0$ and
is non-decreasing: $h(A)\leq h(B)$ if $A\subseteq B\subseteq N$.

One can equivalently introduce structural semi-graphoids as CI models induced by submodular functions or by polymatroids in an analogous way.
Clearly, it is just a matter of taste and convention whether the supermodular view or the submodular view is chosen.
\item The concept of a structural semi-graphoid can be viewed as a reduced version %
of the concept of a {\em semi-matroid\/} from \cite[\S\,2]{MatStu95CPC}. The difference is that
semi-matroids also involve {\em functional dependencies\/} within $N$
(see the remark concluding Section~\ref{ssec.CI-families}).
But otherwise semi-matroids are defined analogously: they are induced by polymatroids
through vanishing of the respective submodular difference expressions.

The polymatroidal attitude allows one to emphasize the links to {\em matroid theory\/}
\cite{Oxl11} because polymatroids generalize matroids. Specifically, a polymatroidal rank function $h$
on ${\cal P}(N)$ corresponds to a {\em matroid\/} over $N$ if and only if it takes values in ${\dv Z}$ and
additionally satisfies $h(I)\leq |I|$ for any $I\subseteq N$.
\end{enumerate}

\subsection{Graphoids, compositional semi-graphoids, and compositional graphoids}\label{ssec.graphoids}
Researchers interested in graphical statistical models introduced various other abstract CI frames \cite{Pea88,SL14};
here we follow the line from \cite{LM07}.
A~{\em graphoid\/} over a variable set $N$
is a semi-graphoid $\mod\subseteq\sta(N)$ satisfying the implication
$$
ij|\ell K,\, i\ell|jK\in\mod ~~\Rightarrow~~ ij|K,\, i\ell|K\in\mod \,.
$$
A semi-graphoid is called {\em compositional\/} if it satisfies the reversed implication, and a {\em compositional graphoid\/} is a semi-graphoid satisfying both implications.
It is evident that $\mod$ is a graphoid if and only if its dual model
$\mod^{\dual}$ is a compositional semi-graphoid.

The above definitions, thus, delimit the frame $\gra$ of
graphoids, the frame $\gra^{\dual}$ of compositional
semi-graphoids, and the frame $\cogr$ of compositional
graphoids. By definition, $\cogr\subseteq \gra,\gra^{\dual}\subseteq\semgr$. Our notation reflects the fact that
$\gra^{\dual}$ can be interpreted as the dual frame to $\gra$
because $\mod\in\gra^{\dual}(N)\,\Leftrightarrow\,\mod^{\dual}\in\gra(N)$ for any variable set $N$ and $\mod\subseteq\sta(N)$.
It is known that $\gra\setminus \profam\neq\emptyset$ and $\profam\setminus\gra\neq\emptyset$.

\subsection{Graphical CI frames}\label{ssec.graphical}
A quite special CI frame is given by undirected graphical models,
which were called {\em separation graphoids} in \cite{LM07}. Given an undirected
graph $G$ over a variable set $N$ (= a graph having $N$ as the node set), we say that $ij|K\in\sta(N)$
is represented in $G$ (after the separation criterion) and
write $i\ci j\,|\,K\,\,[G]$ if every path in $G$ between $i$ and $j$ contains a node in $K$. Then $G$ induces a CI model
$$
\mod_{G} ~:=~ \{\, ij|K\, :\
i\ci j\,|\,K\,\,[G]\,\}\subseteq\sta(N)\,.
$$
The family of {\em undirected graphical models} is then
$$
\ungr(N) ~:=~ \{\,
\mod_{G}\,:\
\mbox{$G$ is an undirected graph over $N$}
\,\}\,.
$$
This defines a CI frame $\ungr$, which is closed under copying and marginalization, but not under intersection. Well-known facts are that $\ungr\subseteq\profam\subseteq\semgr$ \cite{GP93}. The $\ungr$-frame
has a finite ``axiomatization''~\cite{Pea88}.

A popular class of CI models studied in probabilistic reasoning are {\em directed acyclic graphical models}.
These are induced by directed graphs without directed cycles through the respective $d$-separation criterion  \cite{Pea88}.
The corresponding CI frame is closed under copying but not under marginalization \cite[Example\,3.1]{Stu05}.

\subsection{Other examples of abstract CI frames}\label{ssec.other-frames}
Another CI frame, which is not discussed in this paper, consists of {\em regular
Gaussian CI models\/}~\cite{Boe22}, that is, CI structures induced by
(continuous) Gaussian probability distributions. This frame is closed under copying, marginalization, and duality
but not under intersection.

The frame of {\em gaussoids} was introduced in \cite{LM07} and examined in \cite{BDKS19}. This frame approximates the frame
of (regular) Gaussian CI models. A recent paper \cite{Boe23} deals with the structural self-adhesivity concept for gaussoids and applies
these techniques to Gaussian random variables.

Further options were pinpointed by one of the reviewers. One can consider the {\em matroidal frame},
a subframe of $\strum$ consisting of CI models induced by rank functions of matroids.
The other option is the {\em linear frame}, consisting of CI models induced by linear
polymatroids. Yet another option is the {\em almost entropic frame}, discussed in more
details in Section~\ref{sec.entro-region}.

\section{Adhesion and self-adhesion}\label{sec.adhesivity}
The concept of an adhesion/self-adhesion was introduced by Mat\'{u}\v{s} \cite{Mat07DM}
in the context of polymatroids. As pinpointed by Csirmaz \cite{Csi20Kyb}, the
adhesion concept is a special case of the concept of an amalgam for polymatroids.
Note that amalgams were introduced and studied much earlier in matroid theory \cite{Oxl11}.

Mat\'{u}\v{s}'s concept of self-adhesion can be viewed as an abstraction of the procedure used to derive
non-Shannon information inequalities \cite{ZY98}, which is commonly named the ``{\em copy lemma}" method \cite{DFZ11}
in information-theoretical community. More details on the connection to this field
are discussed in Section~\ref{sec.entro-region}.

It has been observed recently in \cite{Boe22} that these concepts can also be defined
within a simpler discrete environment of CI structures. We follow this link here.
An auxiliary concept of consonant models is needed.

\begin{definition}[consonant models, amalgam, adhesion]\label{def.adhesion}~\rm
\begin{itemize}[leftmargin=2em, rightmargin=1em]
\item
Let $N,M$ be two variable sets, possibly intersecting.
Models $\mod\subseteq\sta(N)$ and $\altmod\subseteq\sta(M)$ will be called {\em consonant\/}
if $\mod^{\downarrow N\cap M}=\altmod^{\downarrow N\cap M}$.

\item
Given consonant models $\mod\subseteq\sta(N)$ and $\altmod\subseteq\sta(M)$, their {\em amalgam\/}
is any ``extended'' model $\modext\subseteq\sta(NM)$ over $NM$ such that $\modext^{\downarrow N}=\mod$ and
$\modext^{\downarrow M}=\altmod$.

\item
An {\em adhesion\/} of\/ $\mod\subseteq\sta(N)$ and $\altmod\subseteq\sta(M)$
is their amalgam $\modext\subseteq\sta(NM)$ that, moreover, satisfies
$(N\setminus M)\ci (M\setminus N)\,|\,(N\cap M)\,\,[\modext]$ (see Definition~\ref{def.global-CI}).
\end{itemize}
\end{definition}

Note that any pair of models $\mod\subseteq\sta(N)$ and
$\altmod\subseteq\sta(M)$ is consonant if $|N\cap M|\leq 1$.
The~consonancy is a necessary condition for the existence of an amalgam.
Observe that any pair of consonant models has an adhesion: take the union
of $\mod$, $\altmod$, and the subset of\/ $\sta(NM)$
delimiting the global CI statement $(N\setminus M)\ci (M\setminus N)\,|\,(N\cap M)$.

Unlike the algebraic operations from Definition~\ref{def.operations},
amalgams and adhesions are not uniquely determined by their input models $\mod$ and $\altmod$.
On the other hand, (set-theoretical) intersection of two amalgams/adhesions
of $\mod$ and $\altmod$ is their amalgam/adhesion; thus,
their least amalgam/adhesion exists.
\medskip

\begin{example}\label{exa.cons-am-ad}\rm
Take $N:=\{a,b,c,d\}$, $M:=\{c,d,e\}$, and put $\mod:=\{\,ab|\emptyset,\,cd|\emptyset\,\}$, $\altmod:=\{\,cd|\emptyset,\, cd|e\,\}$, and $\altmod^{\prime}:=\{\,cd|e\,\}$. Then $\mod$ and $\altmod$ are consonant models,
while $\mod$ and $\altmod^{\prime}$ are not.
The least amalgam of $\mod$ and $\altmod$ is $\{\, ab|\emptyset,\, cd|\emptyset,cd|e\,\}$ and
the least adhesion of $\mod$ and $\altmod$ is
$$
\{\, ab|\emptyset,\, cd|\emptyset,\, cd|e,\,\,\, ae|cd,\, ae|bcd,\, be|cd,\, be|acd\,\}\,,
$$
where the last 4 statements represent the global CI statement $ab\ci e\,|\,cd$.
\end{example}

\subsection{Self-adhesion}\label{ssec.self-adhesion}
The self-adhesion of a CI model $\mod$ is nothing but
the adhesion of $\mod$ with its own $\bij$-copy, where $\bij$ is an appropriate bijection.
The derived concept of self-adhesivity is relative to a %
CI frame.

\begin{definition}[self-adhesion]\label{def.frame-self-adhesion}~\rm
\begin{itemize}[leftmargin=2em, rightmargin=1em]
\item Let $\ffam$ be a CI frame closed under copying and marginalization.
Given $\mod\in\ffam(N)$ and $L\subseteq N$ we say that $\mod$ is {\em self-adhesive at $L$\/}
relative to $\ffam$ if, for every bijection $\bij:N\to M$ such that $L=N\cap M$ and $\bij_{\,|L}$ is the identity,
an adhesion of $\mod$ and $\bij(\mod)$ exists which is in the family $\ffam(NM)$:
\begin{eqnarray*}
\lefteqn{\hspace*{0mm}\forall\, \bij:N\to M ~\mbox{bijection with $L=N\cap M$ and $\bij_{\,|L}=\id_{L}$~}
~~ \exists\,\, \modext\in\ffam(NM)}\\ %
 && \mbox{such that~ $\modext^{\downarrow N}=\mod$,~ $\modext^{\downarrow M}=\bij(\mod)$ and
$(N\setminus M)\ci (M\setminus N)\,|\,L\,\,[\modext]$}\,.
\end{eqnarray*}
We say that $\mod\in\ffam(N)$ is {\em self-adhesive\/} relative to $\ffam$ if it is
self-adhesive at every $L\subseteq N$.

\item
The symbol $\ffam^{\sa}(N)$ will denote the family of {\em self-adhesive models\/} $\mod\in\ffam(N)$  relative to $\ffam$. If~$\ffam(N)=\ffam^{\sa}(N)$ then we say that $\ffam(N)$ is self-adhesive.
The frame $\ffam$ is {\em self-adhesive\/} if this holds for every finite $N$.

\item
Given $\mod\subseteq\sta(N)$ and $L\subseteq N$, if there exists the least model $\altmod\in\ffam(N)$
such that $\mod\subseteq\altmod$ and $\altmod$ is self-adhesive at $L$ relative to
$\ffam$ then we will call $\altmod$ the {\em self-adhesive closure\/} of $\mod$
at $L$ (relative to $\ffam$) and denote it by $\ffam^{\sa}(\mod|L)$.
\end{itemize}
\end{definition}

The assumptions on $\ffam$ (copying and marginalization) are needed to define the concept of self-adhesion. Clearly, a model $\mod\in\ffam(N)$ is self-adhesive at $L\subseteq N$ relative to $\ffam$ if and only if  $\ffam^{\sa}(\mod|L)=\mod$.
\medskip

\noindent {\em Remarks.}~ We have a few additional related remarks.
\begin{enumerate}[leftmargin=2em, rightmargin=1em]
\item \label{remarks:3.1} One can equivalently introduce a self-adhesive model $\mod$ at a set $L$ by a formally weaker
condition:
\begin{eqnarray*}
\lefteqn{\hspace*{-5mm}\exists\, \altbij:N\to O ~\mbox{bijection with $L=N\cap O ~\mbox{and}~ \altbij_{\,|L}=\id_{L}$}
~~~ \exists\,\, \widetilde{\modext}\in\ffam(NO) \,:}\\ %
 && \widetilde{\modext}^{\downarrow N}=\mod ~\&~~ \widetilde{\modext}^{\downarrow O}=\altbij(\mod) ~\&~~ (N\setminus O)\ci (O\setminus N)\,|\,L\,\,[\widetilde{\modext}]\,.
\end{eqnarray*}
This follows from the assumption that $\ffam$ is closed under copying. Having fixed $\altbij:N\to O$, for any considered bijection $\bij:N\to M$ we define $\bijelse:NO\to NM$ by $\bijelse(i):=i$ for $i\in N$ and $\bijelse(i):=\bij(\altbij_{-1}(i))$ for $i\in O\setminus N$.
Then $\modext:=\bijelse(\widetilde{\modext})$ is the desired adhesion for $\bij$.
Therefore, the choice of the bijection $\bij$ is irrelevant and the validity of the
self-adhesivity condition is, in fact, an intrinsic property of $\mod$ (relative to $\ffam$).
\item This is to explain why the self-adhesivity concept is relative to a CI frame. Having two
frames $\ffam_{1}\subseteq\ffam_{2}$, imagine one finds a model $\mod\in\ffam_{1}(N)$
and a bijection $\bij:N\to M$ with $L=N\cap M$ such that a self-adhesion $\modext\in\ffam_{2}(NM)$
of $\mod$ for $\bij$ exists while there is no such a self-adhesion within $\ffam_{1}(NM)$.
Then $\mod$ is self-adhesive at $L$ relative to $\ffam_{2}$ but not relative to $\ffam_{1}$
(see Example~\ref{exa.sa-is-relative} below).
\item It follows from the definition that, given $\mod\in\ffam(N)$, the question whether $\mod\in\ffam^{\sa}(N)$
can be answered on basis of the characterization/description of families
$\ffam(M)$, where $|N|\leq |M|\leq 2\cdot |N|$.
Nevertheless, further reduction is possible. For example, it is immediate that any $\mod\in\ffam(N)$ is self-adhesive at $L:=N$. Additional reductions, based on further assumptions on the frame,
are discussed below.
\end{enumerate}

\begin{example}\label{exa.sa-is-relative}\rm
(\,self-adhesivity is relative to a frame\,)\\
Consider frames $\ffam_{1}:=\ungr$, $\ffam_{2}:=\semgr$ and the
variable set $N:=\{a,b,c\}$. The model $\mod:=\{\, bc|a\,\}\subseteq\sta(N)$ belongs to $\ungr(N)$: it is induced
by the graph $G$ over $N$ with two edges, those connecting $a$ with the other two nodes.
Then $\mod\in\semgr^{\sa}(N)$ but  $\mod\not\in\ungr^{\sa}(N)$. To this end consider the set $L:= bc$ and a bijection
$\bij:N\to M:=\{ \bar{a},b,c\}$ with $\bij(a):=\bar{a}$,
$\bij(b):=b$, and $\bij(c):=c$ (see Figure~\ref{fig2}).
Then $\mod$ is not self-adhesive at $L$ relative to $\ungr$.
Indeed, a potential adhesion $\modext$ of $\mod$ and $\bij(\mod)$ in $\ungr(NM)$ has to satisfy
$$
\{\, bc|a,\, bc|\bar{a},\, a\bar{a}|bc\,\}\subseteq\modext~~
\mbox{and} ~~\{ ab|c,\, ac|b,\,\bar{a}b|c,\, \bar{a}c|b\,\}\cap\modext=\emptyset\,,
$$
and there is no undirected graph $H$ over $NM$ with $\mod_{H}=\modext$.
\end{example}

\begin{figure}[t]
\centering
\includegraphics[scale=0.6]{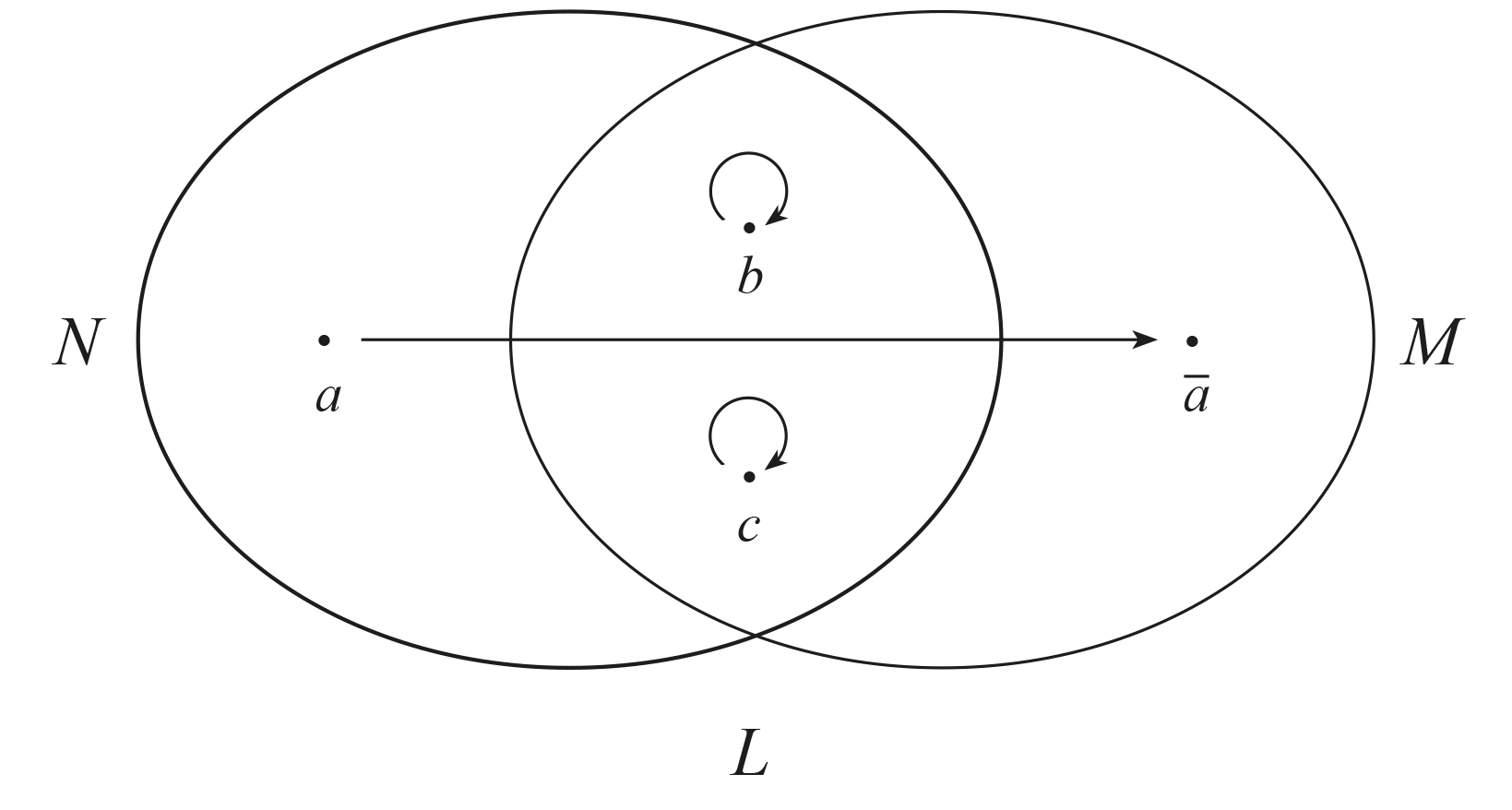}
\caption{The bijection $\bij:N\to M$ considered in Examples~\ref{exa.sa-is-relative}
and \ref{exa.self-adhe}.}
\label{fig2}
\end{figure}

If a CI frame $\ffam$ is closed under {\em intersection\/} then the self-adhesive closure of any model at any set is easy to compute by applying the $\ffam$-closure. In particular, the question of self-adhesivity testing can be simplified.

\begin{lemma}\label{lem.self-adhesion-at-set}\rm
Let $\ffam$ be a CI frame closed under copying, marginalization, intersection, and ascetic extension. Given $\mod\subseteq\sta(N)$ and $L\subseteq N$, the self-adhesive closure $\altmod=\ffam^{\sa}(\mod|L)$ exists and can be computed as follows. Having fixed a bijection
$\bij:N\to M$ with $L=N\cap M$ and $\bij_{\,|L}=\id_{L}$, apply the
next steps:
\begin{enumerate}
\item take the least self-adhesion of $\mod$:
$\modext:= \mod\cup \bij(\mod)\cup [N\setminus M,M\setminus N|L]$,
\item apply the $\ffam$-closure to it: $\widetilde{\modext}:=\ffam(\modext)$, \quad (see Definition~\ref{def.short-closure})
\item apply marginalization to it:
$\altmod:=\widetilde{\modext}^{\downarrow N}$.
\end{enumerate}
\end{lemma}

\begin{proof}
Extend the bijection $\bij:N\to M$ by
the restriction of its own inversion $\bij_{-1}$ to $M\setminus N$ to get a permutation $\altbij$ of $NM$. Observe that $\altbij(\modext)=\modext$.
Because $\ffam$ is closed under copying, by arguments in Observation~\ref{obs.model-lattice.1}, the $\ffam$-closure over $NM$ commutes with $\altbij$ and one even has $\altbij(\widetilde{\modext})=\widetilde{\modext}$.
It follows from the construction and the fact that $\ffam$ is closed under marginalization that $\mod\subseteq\altmod\in\ffam(N)$. Using that
$\widetilde{\modext}$ is fixed by~$\psi$ and the fact that copying commutes with marginalization, one has
$$
\widetilde{\modext}^{\downarrow M}=
\widetilde{\modext}^{\downarrow \altbij(N)} =
\altbij(\widetilde{\modext})^{\downarrow \altbij(N)} =
\altbij_{\,|N}(\widetilde{\modext}^{\downarrow N}) =
\bij(\widetilde{\modext}^{\downarrow N}) =
\bij(\altmod)\,,
$$
which implies, by Definition~\ref{def.frame-self-adhesion}, that $\altmod$ is self-adhesive at $L$.

Given $\altmod^{\prime}\in\ffam(N)$ which is self-adhesive at $L$, and satisfies $\mod\subseteq \altmod^{\prime}$, it is easy to observe that
the respective adhesion $\modext^{\prime}\in\ffam(NM)$
satisfies $\modext\subseteq\modext^{\prime}$, which
yields  $\widetilde{\modext}\subseteq\modext^{\prime}$ and, hence, $\altmod\subseteq \altmod^{\prime}$.
Therefore, one indeed has $\altmod=\ffam^{\sa}(\mod|L)$.
\end{proof}

Note that the $\ffam$-closure is applied only once and
it is over a variable set of cardinality $|N|+|L|$.
The next example illustrates the method from Lemma~\ref{lem.self-adhesion-at-set}.

\begin{example}\label{exa.self-adhe}\rm
Take $N:=\{a,b,c\}$ and put $\mod:=\{\,ab|c\,\}$; note that $\mod\in\semgr(N)$.  We are going to test
self-adhesivity of $\mod$ at the set $L:=\{b,c\}$ relative to the semi-graphoidal frame~$\semgr$. As~explained (in the remark n.~\ref{remarks:3.1} above) it suffices to fix a variable set $M:=\{\bar{a},b,c\}$ and a bijection $\bij:N\to M$ defined by $\bij(a):=\bar{a}$, $\bij(b):=b$, and $\bij(c):=c$, (see Figure \ref{fig2}) and to find an adhesion of $\mod$ and
$\bij(\mod)$ which belongs to $\semgr(NM)$. To this end we consider their least adhesion
$\modext = \{\,ab|c,\, \bar{a}b|c,\ a\bar{a}|bc  \,\}$
and take its semi-graphoidal closure
$$
\semgr (\modext) \,=\, \{\,ab|c,\, ab|\bar{a}c,\, a\bar{a}|c,\, a\bar{a}|bc,\,  \bar{a}b|c,\, \bar{a}b|ac  \,\}\,.
$$
Since $\semgr (\modext)^{\downarrow N}=\mod$, one has $\semgr^{\sa}(\mod|L)=\mod$, and $\mod$ is self-adhesive at $L$ relative to $\semgr$. Observe that one has $\semgr (\modext)^{\downarrow M}=\{\,\bar{a}b|c\,\}=\bij(\mod)$.
\end{example}

This was a positive example of self-adhesivity at a set. On the other hand, Example \ref{exa.vamosi} (in Section~\ref{ssec.sa-semgr-4}),
obtained as a result of our computations, shows that it may happen that $\semgr^{\sa}(\mod|L)\neq\mod$.

\begin{lemma}\label{lem.pr-self-adhesion}\rm
The discrete probabilistic CI frame $\profam$ is self-adhesive: $\profam^{\sa}=\profam$.
\end{lemma}

\begin{proof}
Given $\mod\in \profam(N)$, there exists a discrete random vector  $\bxi$
over $N$ such that $\mod$ is induced by $\bxi$\/: $\mod_{\ibxi}=\mod$. Given a bijection $\bij:N\to M$
with $\bij_{\,|N\cap\/ M}$ being the identity, define a random vector
 $\boeta:=[\xi_{\,\bij_{-1}(j)}]_{j\in M}$ over $M$ and have $\mod_{\iboeta}=\bij(\mod)$.
Then $\bxi$ and $\boeta$ completely coincide on $N\cap\/ M$
in the sense that $\bxi_{N\cap\/ M}=\boeta_{N\cap\/ M}$. Hence, one can construct
a random vector $\bzeta$ over $NM$ with $\bzeta_{N}=\bxi$, $\bzeta_{M}=\boeta$, and
$(N\setminus M)\ci (M\setminus N)\,|\,(N\cap\/ M)\,\,[\bzeta]$. This step corresponds to a
well-known construction of the {\em conditional product\/} of two consonant
multi-dimensional measures; see \cite[\S\,2.3.3]{Stu05}.
Thus, the model $\modext:=\mod_{\ibzeta}$ is the desired adhesion of $\mod$ and $\bij(\mod)$
belonging to the family $\profam(NM)$.
\end{proof}

\subsection{Self-adhesivity at small sets: lifting}\label{ssec.lifting}
In next two subsections we show that, if a CI frame is closed
under appropriate extending algebraic operations, then the involved models are automatically self-adhesive at sets of particular cardinality.
\smallskip

The first such operation, called {\em lifting\/}, ensures the self-adhesivity at the empty set and singletons. It formalizes the act of adding completely independent variables.
Specifically, in case of a CI model $\mod\subseteq\sta(N)$ and a single (outside) variable $o\not\in N$, the respective (one-variable) lift is the following model over $No$\,:
$$
\mod\,\cup\, \{\, ij|Ko\,:\ ij|K\in\mod\,\}
\,\cup\, \{\, io|K\,:\ i\in N,\, K\subseteq N\setminus i\,\}\,.
$$
Such one-element (independent) additions commute each other and allow one to introduce the addition
of a set of independent variables. Nonetheless, for technical reasons, we prefer the following
condensed definition, which is equivalent to the result of repeated one-variable additions.

\begin{definition}[lifting]\label{def.lifting}~\rm  \\[0.3ex]
Let $N\subseteq O$ be variable sets in inclusion and $\mod\subseteq\sta(N)$.
The {\em lift of $\mod$ to $O$} is the following model over $O$:
$$
\mod_{N\lift O} ~:=~ \{\, ij|K\in\sta(O) \,:\ ij\cap (O\setminus N)\neq\emptyset ~~\mbox{or}~~
ij|(K\cap N)\in\mod\,\}\,.
$$
A CI frame $\ffam$ is closed under {\em lifting}, if for any $\mod\in\ffam(N)$ and any superset $O\supseteq N$,
one has $\mod_{N\lift O}\in\ffam (O)$.
\end{definition}

Note that lifting complies with marginalization in the sense that,
for every $\mod\subseteq\sta(N)$ and $N\subseteq O$, one has $(\mod_{N\lift O})^{\downarrow N}=\mod$.

\begin{example}\label{exa.lift}\rm
Put $N:=\{a,b\}$, $O:=\{a,b,c\}$, $\mod:=\emptyset$, and $\altmod:=\{\,ab|\emptyset\,\}$ and
interpret both $\mod$ and $\altmod$ as CI models over $N$.
Then $\mod_{N\lift O}=\{\,ac|\emptyset,\, ac|b,\, bc|\emptyset,\, bc|a \,\}$.
Note that the variable $c$ became independent of the rest.
In the other case of $\altmod$, one gets
the full model: $\altmod_{N\lift O}=\sta(O)$.
\end{example}

\begin{lemma}\label{lem.lifting}\rm
The CI frames $\profam$, $\strum$, and $\semgr$ are closed under lifting.
\end{lemma}

\begin{proof}
Given a discrete random vector $\bxi=[\xi_{i}]_{i\in N}$ over $N$ and $N\subseteq O$, one
can extend it to a random vector $\boeta=[\eta_{i}]_{i\in O}$ over $O$
with stochastically independent sub-vectors $\boeta_{N}:=\bxi$ and
$\boeta_{i}$, $i\in O\setminus N$. Then $\mod_{\iboeta}=(\mod_{\ibxi})_{N\lift O}$.

Given a supermodular function $r:{\cal P}(N)\to {\dv R}$ and $N\subseteq O$, one
can extend it to a supermodular function $m$ on ${\cal P}(O)$ by the formula
$m(S):=r(S\cap N)$ for any $S\subseteq O$. It is easy to observe that $\mod^{m}=(\mod^{r})_{N\lift O}$.

Given a semi-graphoid $\mod\in\semgr(N)$ over $N$ and $N\subseteq O$, it is straightforward to observe that
$\mod_{N\lift O}$ is a semi-graphoid over $O$.
\end{proof}

The lifting operation is related to adhesions of CI models at small sets.

\begin{lemma}\label{lem.lift-adhes-01}\rm
Let $\ffam$ be a CI frame which is closed under marginalization, intersection, and lifting.
Then, for every $\mod\in\ffam(N)$ and $\altmod\in\ffam(M)$ such that $|N\cap\/ M|\leq 1$,
an adhesion $\modext\in\ffam(NM)$ of $\mod$ and $\altmod$ exists.
\end{lemma}

\begin{proof}
Put $\hat{\mod}:=\mod_{N\lift NM}$ and observe that
$(N\setminus M)\ci (M\setminus N)\,|\,(N\cap\/ M)\,\,[\hat{\mod}]$ and $\hat{\mod}^{\downarrow N}=\mod$.
Moreover, $|N\cap\/ M|\leq 1$ implies $\hat{\mod}^{\downarrow M}=\sta(M)$ and the fact $\hat{\mod}\in\ffam(NM)$
follows from the assumption that $\ffam$ is closed under lifting. Putting $\hat{\altmod}:=\altmod_{M\lift NM}\in\ffam(NM)$
gives
$(N\setminus M)\ci (M\setminus N)\,|\,(N\cap\/ M)\,\,[\hat{\altmod}]$
with $\hat{\altmod}^{\downarrow N}=\sta(N)$ and $\hat{\altmod}^{\downarrow M}=\altmod$.
Since $\ffam$ is closed under intersection, $\modext:=\hat{\mod}\cap\hat{\altmod}$
belongs to $\ffam(NM)$ and, moreover, by the above facts,
$\modext$ is an adhesion of $\mod$ and~$\altmod$.
\end{proof}

Of course, the above observation on adhesivity applies to self-adhesivity.

\begin{cor}\label{cor.lift-self-adhes-01}\rm
Let $\ffam$ be a CI frame which is closed under copying, marginalization, intersection, and lifting.
Then, for any $N$, any model
$\mod\in\ffam(N)$ is self-adhesive
at every set of cardinality $0$ or $1$.
\end{cor}

\begin{proof}
Apply Lemma~\ref{lem.lift-adhes-01} to $\mod$ and $\altmod:=\bij(\mod)$ for the appropriate bijection $\bij:N\to M$.
\end{proof}

\subsection{Self-adhesivity at co-singletons: tight replication}\label{ssec.tight-replica}
In this subsection we discuss a particular extending algebraic operation which ensures that,
if a CI frame $\ffam$ is closed under this operation, then every $\mod\in\ffam(N)$ is
self-adhesive at {\em co-singletons}, that is, at the complements of singletons within the variable set $N$.
This operation is called {\em tight replication\/} and formalizes the act of adding one
identical copy of a fixed variable in $N$. This (technically complex) operation was inspired by an analogous (poly-)matroidal
operation of {\em parallel extension}.
\smallskip

Let us first remind what is our goal and re-formulate the condition of self-adhesivity at a co-singleton
(we wish to ensure). Consider $\mod\in\ffam(N)$, fix a variable $u\in N$ and an outside variable $v\not\in N$. Consider
the co-singleton $L:=N\setminus u$, denote $M:=Lv$, and introduce $\altbij:N\to M$ which is identical on $L$
and $\altbij(u)=v$. By Definition~\ref{def.frame-self-adhesion}, our goal is to verify the
existence of
$$
\mbox{$\modext\in\ffam(NM)$ ~~which satisfies~~ $\modext^{\downarrow N}=\mod$, $\modext^{\downarrow M}=\altbij(\mod)$, and $u\ci v\,|\,L\,\,[\modext]$.}
$$
One of possible ways is to define $\modext$ as a particular extension of $\mod$ in which $v$ is the ``identical" copy of $u$. To this end an auxiliary mapping $\bijelse:NM\to N$ is used, which an extension of the inverse $\altbij_{-1}:M\to N$.

\begin{definition}[tight replication]\label{def.tight}~\rm \\[0.3ex]
Let $N,M$ be two variable sets of cardinality $n\geq 1$ such that $|N\cap\/ M|=n-1$.
Let us denote $L:=N\cap M$, $u:=N\setminus L$, $v:=M\setminus L$, and introduce
a mapping $\bijelse:NM\to N$ by $\bijelse(i):=i$ for $i\in N$ and $\bijelse(v):=u$.
Given $\mod\subseteq\sta(N)$, its {\em tight replication\/} (based on $\rho$) is the model $\mod_{v\|u}\subseteq\sta(NM)$
defined as follows:
\begin{eqnarray*}
\lefteqn{\hspace*{-5mm}\mbox{given $ij|K\in\sta(NM)$,}\quad ij|K\in\mod_{v\|u} ~~\mbox{if and only if}~~ \bijelse(i)\ci \bijelse(j)\,|\,\bijelse(K)\,\,[\mod],}\\
&& \mbox{under conventions that $a\ci b\,|\,C\,\,[\mod]$ whenever $ab\cap C\neq \emptyset$}\\
&& \mbox{and $u\ci u\,|\,C\,\,[\mod]$ for $C\subseteq L$ means $u\ci (L\setminus C)\,|\, C\,\,[\mod]$\,.}
\end{eqnarray*}
A CI frame $\ffam$ is closed under {\em tight replication}, if for any such a pair $N,M$ of variable sets
and $\mod\in\ffam(N)$, one has $\mod_{v\|u}\in\ffam (NM)$.
\end{definition}

Note that tight replication also complies with
marginalization in the sense that ${(\mod_{v\|u})^{\downarrow N}=\mod}$.

\begin{figure}[t]
\centering
\includegraphics[scale=0.6]{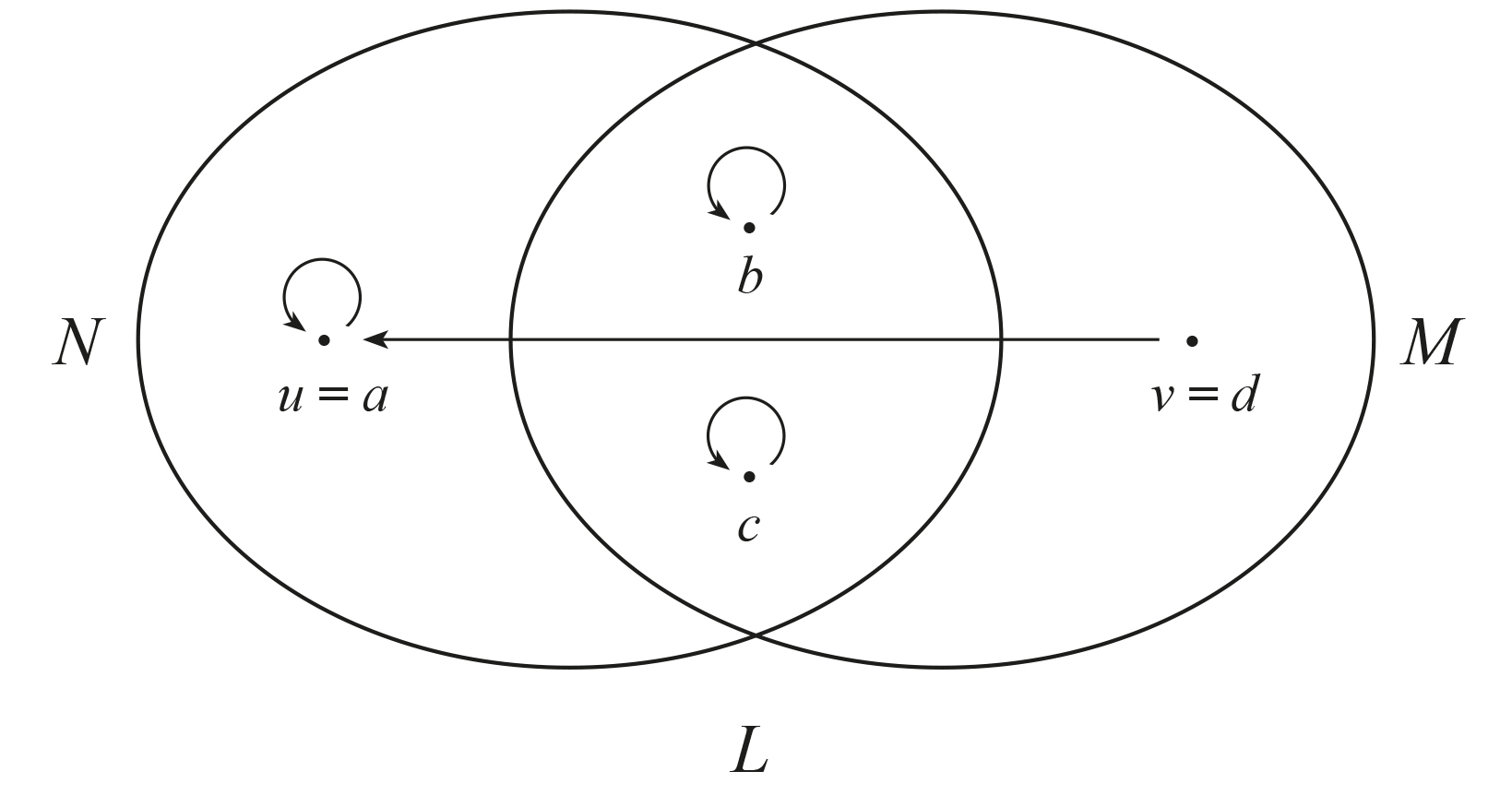}
\caption{The auxiliary mapping $\bijelse:NM\to N$ considered in Example~\ref{exa.tigh-replic}.}
\label{fig3}
\end{figure}

\begin{example}\label{exa.tigh-replic}\rm
Consider the situation from Definition~\ref{def.tight}
when one has $N:=\{a,b,c\}$, $M:=\{b,c,d\}$, $L:=\{b,c\}$, $u:=a$ and $v:=d$.
Define $\bijelse(a):=a$, $\bijelse(b):=b$, $\bijelse(c):=c$,
$\bijelse(d):=a$ (see Figure \ref{fig3} for illustration), put
$\mod:=\emptyset$ and $\altmod:=\{\,ab|c\,\}$. Then
$$
\mod_{v\|u}=\{\,\, ab|d,\, ab|cd,\, ac|d,\, ac|bd,\, bd|a,\, bd|ac,\, cd|a,\, cd|ab,\, ad|bc \,\,\}\,,
$$
where the first 8 statements are there because of the first convention saying that $a\ci b\,|\,C\,\,[\mod]$ whenever $ab\cap C\neq \emptyset$ and the last statement  $ad|bc$ is added
because of the second convention
$u\ci u\,|\,C\,\,[\mod]$ if and only if $u\ci (L\setminus C)\,|\, C\,\,[\mod]$.
In case of $\altmod$ one gets
$\altmod_{v\|u}=\mod_{v\|u}\cup\{\, ab|c,\, bd|c,\, ad|c \,\}$,
where $bd|c$ is mapped by $\bijelse$ to $ab|c$ and
$ad|c$ is added because of the second convention.
\end{example}

\noindent {\em Remark on terminology.}\,\,
Operations with supermodular set functions, which are analogous to those in Definitions~\ref{def.lifting} and \ref{def.tight}, were discussed in \cite[\S\,7.1 and \S\,10.2]{Stu16}. %
Moreover, the concept of tight replication of CI models is a counterpart
of %
an analogous notion of {\em parallel extension\/} of matroids \cite[\S\,7.2]{Oxl11} and of {\em polymatroids\/} \cite[\S\,2]{Mat07DM}.
Nevertheless, the parallel extension of polymatroids has
two slightly different reflections in the context of induced CI models.
In the case that the considered polymatroid $h$ is {\em tight\/}
\cite[\S\,1.3]{Csi20Kyb} at the variable $u\in N$, which means
that $h(N)=h(N\setminus u)$, it leads to our concept
from Definition~\ref{def.tight}. If, however, $h$ is not tight at $u$ then it leads to the concept of parallel extension of CI models introduced in 2004 by Mat\'{u}\v{s} \cite[end of \S\,4]{Mat04DM}.
The~extension we need for our purposes is the one
that corresponds to a tight polymatroid. Therefore,
to avoid a terminological clash, we have chosen a different
term than parallel extension. Our~terminology also reflects the
crucial role of tightness.

\begin{lemma}\label{lem.struc-par-extens}\rm
The CI frames $\strum$ and $\semgr$ are closed under tight replication.
\end{lemma}

\begin{proof}
Consider variable sets $N,M$ assumed in the definition of tight replication and keep the corresponding notational conventions from Definition~\ref{def.tight}.
Assume that $\mod\in\strum(N)$ is induced by a supermodular function $r:{\cal P}(N)\to {\dv R}$.
Slight modification $\hat{r}$ of $r$ defined by
$$
\hat{r}(S) ~:=~
\left\{
\begin{array}{cl}
r(S)-r(N)+r(L) & \mbox{if $u\in S$,}\\
r(S) & \mbox{if $u\not\in S$,}\\
\end{array}
\right.
\qquad \mbox{for $S\subseteq N$,}
$$
is another supermodular function on ${\cal P}(N)$ which also induces $\mod$ but, moreover, satisfies
$\hat{r}(N)=\hat{r}(L)$.
The next step is to define
$$
m(S) ~:=~ \hat{r}(\,(S\cap N)\cup \bijelse(S\cap M)\,)\qquad \mbox{for~ $S\subseteq NM$}\,.
$$
For any $ij|K\in\sta(NM)$ one has $\Delta m (ij|K)=\Delta \hat{r}\,\langle\bijelse(i),\bijelse(j)|\bijelse(K)\rangle$.
The latter expression is non-negative in case $\bijelse(ij)|\bijelse(K)\in\sta(N)$ by supermodularity of $\hat{r}$.
Moreover, the expression vanishes if $\bijelse(ij)\cap\bijelse(K)\neq\emptyset$. In the remaining case
$\bijelse(i)=u=\bijelse(j)$ and $C:=\bijelse(K)\subseteq L$ one has
$$
\Delta \hat{r}\, \langle\bijelse(i),\bijelse(j)|\bijelse(K)\rangle=
\Delta \hat{r}\, \langle u,u|C\rangle=
\hat{r}(C)-\hat{r}(uC) =
\Delta \hat{r}\, \langle u,L\setminus C|C\rangle\geq 0\,,
$$
by supermodularity of $\hat{r}$. Hence, $m$ is supermodular and it follows from the definition
of tight replication that $\mod^{m}=\mod_{v\|u}$.

Given $\mod\in\semgr(N)$, it is enough to show that $\modext:=\mod_{v\|u}$ is a semi-graphoid.
Consider pairwise disjoint sets $i,j,\ell,K\subseteq NM$ and the goal
is to verify
$$
ij|K,\, i\ell |jK\in\modext ~\Leftrightarrow~ i\ell|K,\, ij|\ell K\in\modext\,.
$$
One can distinguish several cases determined by the mutual position of variables $u$ and $v$ relative
to $i,j,\ell,K$. If $uv\setminus ij\ell K\neq\emptyset$ then the conclusion follows from the facts
that $\modext^{\downarrow N}=\mod$ and  $\modext^{\downarrow M}$ is a copy of $\mod$. The same principle applies if $u,v\in K$.
In~cases $|uv\cap K|=1=|uv\cap ij\ell|$ the first convention in the definition of $\mod_{v\|u}$
is applied. This also extends to the case $uv=j\ell$. In~the~cases $uv=ij$ and $uv=i\ell$
one additionally applies the second
convention in the definition of $\mod_{v\|u}$ together with the assumption that $\mod$ is a semi-graphoid.
\end{proof}

\noindent {\em Remark.} Note in connection with Lemma~\ref{lem.struc-par-extens} that
have not tried to answer the question whether $\profam$ is closed under tight replication.
If it is the case then this certainly requires a technically complicated proof, but the point
is that such a result is not needed for our purpose. Indeed, we know, by Lemma~\ref{lem.pr-self-adhesion},
that $\profam$ is self-adhesive, and, thus, self-adhesive at co-singletons.

\begin{lemma}\label{lem.par-ext-n-1}\rm
Let $\ffam$ be a CI frame which is closed under copying, marginalization, and tight replication.
Then, for any variable set $N$ of cardinality $n\geq 1$, any model $\mod\in\ffam(N)$ is self-adhesive
at every co-singleton.
\end{lemma}

\begin{proof}
Consider $\mod\in\ffam(N)$, a subset $L\subseteq N$ with $|L|=n-1$ (= a co-singleton), and
a bijection $\altbij:N\to M$ such that $N\cap\/ M=L$ and $\altbij_{\,|L}=\id_{L}$.
Put $u:=N\setminus L$ and $v:=M\setminus L$. By the assumption,
$\modext:=\mod_{v\|u}$ belongs to $\ffam(NM)$. The definition
of tight replication allows one to observe that $\modext^{\downarrow N}=\mod$,
$\modext^{\downarrow M}=\altbij(\mod)$, and $u\ci v\,|\,L\,\,[\modext]$.
Hence, $\modext$ is an adhesion of $\mod$ and $\altbij(\mod)$.
\end{proof}

We now apply the observations from Sections~\ref{ssec.lifting} and \ref{ssec.tight-replica}.

\begin{cor}\label{cor.case-3}\rm
If a CI frame $\ffam$ is closed under copying, marginalization, intersection, lifting,
and tight replication and $|N|\leq 3$ then $\ffam^{\sa}(N)=\ffam(N)$.
In particular, $\semgr^{\sa}(N)=\semgr(N)$ and $\strum^{\sa}(N)=\strum(N)$ if $|N|\leq 3$.
\end{cor}

\begin{proof}
Combine Corollary~\ref{cor.lift-self-adhes-01}, Lemma~\ref{lem.par-ext-n-1} and the fact
that any $\mod\in\ffam(N)$ is self-adhesive at $L:=N$.
Use Lemmas~\ref{lem.sg-frame}, \ref{lem.st-frame}, \ref{lem.lifting}, and \ref{lem.struc-par-extens} for the second part.
\end{proof}

The second part of Corollary~\ref{cor.case-3} is not surprising in light of Lemma~\ref{lem.pr-self-adhesion} since
$\profam(N)=\strum(N)=\semgr(N)$ if $|N|\leq 3$. On the other hand, Corollary~\ref{cor.lift-self-adhes-01} and
Lemma~\ref{lem.par-ext-n-1} allow one to simplify testing self-adhesive models in case of large variable sets $N$ with
$|N|\geq 4$.

\section{Self-adhesion operator and its iterating}\label{sec.iteration}
In this section we discuss a self-adhesion operator which assigns a subframe $\ffam^{\sa}$ to an abstract CI frame $\ffam$, closed under copying and marginalization.
We prove the consistency results for such a construction,
present an algorithm to compute the respective $\ffam^{\sa}$-closure, and discuss the iterations of the self-adhesion operator.
\smallskip

Observe that copying commutes with marginalization:
given $\mod\in\sta(N)$, a set $L\subseteq N$, and a bijection $\bij:N\to M$, one has
$\bij(\mod)^{\downarrow \bij(L)}=\bij_{|L}(\mod^{\downarrow L})$.
Analogously, intersection commutes both with copying and marginalization:
given $\mod,\altmod\in\sta(N)$, a bijection $\bij:N\to M$ and a set $L\subseteq N$,
one has $\bij(\mod\cap\altmod)=\bij(\mod)\cap\bij(\altmod)$ and $(\mod\cap\altmod)^{\downarrow L}=
\mod^{\downarrow L}\cap \altmod^{\downarrow L}$.
\smallskip

\begin{lemma}\label{lem.iter-copy}\rm
Let $\ffam$ be a CI frame closed under copying and marginalization.\\[0.4ex]
Then $N\mapsto\ffam^{\sa}(N)$ defines a CI frame $\ffam^{\sa}$ which is also closed under copying and marginalization.
\begin{itemize}[leftmargin=2.5em, rightmargin=1em]
\item[(i)] If~$\ffam$ is, additionally, closed under intersection, then the same holds for $\ffam^{\sa}$.
\item[(ii)] If~$\ffam$ is, additionally, closed under lifting, then the same holds for $\ffam^{\sa}$.
\item[(iii)] If~$\ffam$ is, additionally, closed under intersection, lifting, and ascetic extension, then so is $\ffam^{\sa}$.
\end{itemize}
\end{lemma}

\begin{proof}
It is easy to see that $\sta(N)\in\ffam^{\sa}(N)$ because $\sta(O)\in\ffam(O)$ for any $O\supseteq N$.
Hence, $N\mapsto\ffam^{\sa}(N)$ is an abstract CI frame in sense of Definition~\ref{def.elem-models}.

Consider $\mod\in\ffam^{\sa}(N)$ and a bijection $\bij:N\to M$. The goal is to show that
$\altmod:=\bij(\mod)\in\ffam^{\sa}(M)$. Take $L\subseteq M$ and choose a bijection $\altbij:M\to O$
such that $L=M\cap O$, $\altbij_{|L}=\id_{L}$, and ${(O\setminus L)\cap N=\emptyset}$.
By the remark n.~\ref{remarks:3.1} below Definition~\ref{def.frame-self-adhesion}, it is enough to show that
$\altmod$ and $\altbij(\altmod)$ have an adhesion in $\ffam(MO)$.
Put $K:=\bij_{-1}(L)\cup (O\setminus L)$ and define a bijection $\bijelse:N\to K$ by
$\bijelse(i):=i$ for $i\in\bij_{-1}(L)$ and $\bijelse(i):= \altbij(\bij(i))$ for $i\in N\setminus \bij_{-1}(L)$.
Since $\mod$ is self-adhesive at $\bij_{-1}(L)$, an extended model $\modext\in\ffam(NK)$ exists
with $\modext^{\downarrow N}=\mod$, $\modext^{\downarrow K}=\bijelse(\mod)$ and
$(N\setminus \bij_{-1}(L))\ci (K\setminus \bij_{-1}(L))\,|\, \bij_{-1}(L)\,\,[\modext]$.
Extend $\bij$ by putting $\bij(i):=i$ for $i\in K\setminus N=O\setminus L=O\setminus M$ and have $\bij:NK\to MO$.
Then $\bij(\modext)\in\ffam(MO)$ by the assumption on $\ffam$. The fact that copying commutes with marginalization
together with the identity $\bij_{|K}\circ\bijelse =\altbij\circ\bij_{|N}$ allows one to observe
$\bij(\modext)^{\downarrow M}=\altmod$, $\bij(\modext)^{\downarrow O}=\altbij(\altmod)$
and $(M\setminus L)\ci (O\setminus L)\,|\, L\,\,[\,\bij(\modext)\,]$.

Consider $\mod\in\ffam^{\sa}(N)$ and $M\subseteq N$. To show
$\mod^{\downarrow M}\in\ffam^{\sa}(M)$ take $L\subseteq M$ and choose a bijection $\altbij:M\to O$
with $L=N\cap O$ and $\altbij_{|L}=\id_{L}$.
Again, it is enough to show that $\mod^{\downarrow M}$ and $\altbij(\mod^{\downarrow M})$
have an adhesion in $\ffam(MO)$.
Choose a set $K$ with $K\cap NO=\emptyset$ and extend $\altbij$ to a bijection
from $N$ to $OK$. Since $\mod$ is self-adhesive at $L$, a model $\modext\in\ffam(NOK)$ exists
with $\modext^{\downarrow N}=\mod$, $\modext^{\downarrow OK}=\altbij(\mod)$ and
$(N\setminus L)\ci (OK\setminus L)\,|\, L\,\,[\modext]$. The assumption on $\ffam$ gives
$\modext^{\downarrow MO}\in\ffam(MO)$.
Clearly, $(\modext^{\downarrow MO})^{\downarrow M}=\modext^{\downarrow M}=(\modext^{\downarrow N})^{\downarrow M}=\mod^{\downarrow M}$,
and the fact that copying commutes with marginalization
allows one to observe $(\modext^{\downarrow MO})^{\downarrow O}=\altbij(\mod^{\downarrow M})$.
It is immediate that $(M\setminus L)\ci (O\setminus L)\,|\, L\,\,[\,\modext^{\downarrow MO}\,]$.

To verify (i) assume $\mod,\altmod\in\ffam^{\sa} (N)$ and $\bij:N\to M$ with $L:=N\cap\, M$ and
$\bij_{\,|L}=\id_{L}$. To show that $\mod\cap \altmod$ is self-adhesive at $L$
one finds the respective extensions $\mod\mapsto \modext_{\mod}$, $\altmod\mapsto \modext_{\altmod}$,
both in $\ffam(NM)$ and puts $\modext :=\modext_{\mod}\cap \modext_{\altmod}\in\ffam(NM)$, by the assumption on $\ffam$.
Since the intersection commutes both with marginalization and copying one gets
$\modext^{\downarrow N}=\mod\cap \altmod$ and $\modext^{\downarrow M}=\bij(\mod\cap \altmod)$.
The definition of $\modext$ also implies $(N\setminus M)\ci (M\setminus N)\,|\, L\,\,[\modext]$.

To verify (ii) assume $\mod\in\ffam^{\sa} (N)$ and, without loss of generality,
$O\supseteq N$ with $O\setminus N$ being a singleton~$o$. We distinguish two cases.
The first one is to show that $\mod_{N\lift O}$ is self-adhesive
at $L\subseteq N$ relative to $\ffam$. Consider a bijection
$\bij:O\to \tilde{M}$ with $L=O\cap\tilde{M}$ and $\bij_{\,|L}=\id_{L}$; put $M:=\tilde{M}\setminus \bij(o)$.
The assumption on $\mod$ implies the existence of $\modext\in\ffam(NM)$
with $\modext^{\downarrow N}=\mod$, $\modext^{\downarrow M}=\bij(\mod)$ and
$(N\setminus M)\ci (M\setminus N)\,|\, L\,\,[\modext]$. Then $O\tilde{M}$ is a disjoint union of
$NM$ and $o\bij(o)$ and we can put $\altmod:=\modext_{NM\lift O\tilde{M}}$, which
belongs to $\ffam(O\tilde{M})$, by the assumption on $\ffam$.
The definition of lifting and the fact that it commutes with copying allow one to observe
 $\altmod^{\downarrow O}=\mod_{N\lift O}$, $\altmod^{\downarrow \tilde{M}}=\bij(\mod_{N\lift O})$ and
$(O\setminus \tilde{M})\ci (\tilde{M}\setminus O)\,|\, L\,\,[\altmod]$.

The second case is to show that $\mod_{N\lift O}$ is self-adhesive at $Lo$ (for $L\subseteq N$) relative to $\ffam$. Consider a bijection $\bij:O\to \tilde{M}$ with $Lo:=O\cap\tilde{M}$ and $\bij$ identical on $Lo$\,; put $M:=\tilde{M}\setminus o$.
Choose an adhesion $\modext\in\ffam(NM)$ of $\mod$ and $\bij(\mod)$ (at $L$). As $O\tilde{M}=NMo$ put $\altmod:=\modext_{NM\lift O\tilde{M}}\in\ffam(O\tilde{M})$. Analogous arguments allow one to show that $\altmod$ is an adhesion of  $\mod_{N\lift O}$
and  $\bij(\mod_{N\lift O})$.

The verification of (iii) means, owing to (i) and (ii), to show
that, given $N\subseteq M$, one has $\ffam^{\sa}(N)\subseteq\ffam^{\sa}(M)$.
Since, however, we already know that $\ffam^{\sa}$ is closed under extension and intersection,
it reduces to showing that $\sta(N)\in\ffam^{\sa}(M)$ (see the comments below Definition~\ref{def.operations}).
Thus, we put $\mod:=\sta(N)$ and regard $\mod$ as a CI model over $M$. We know
that $\mod\in\ffam(M)$ by the assumption on $\ffam$. The goal is to verify that $\mod\in\ffam^{\sa}(M)$.
To this end consider $L\subseteq M$ and a bijection $\bij:M\to O$ with $L=M\cap O$ and $\bij_{\,|L}=\id_{L}$.
Let us put $\modext_{1}:=\mod_{M\lift MO}$ and note that $\modext_{1}\in\ffam(MO)$ by the assumption of that $\ffam$
is closed under lifting. We also have $\modext_{1}^{\downarrow M}=\mod$ and
$(M\setminus O)\ci (O\setminus M)\,|\, L\,\,[\modext_{1}]$. A special property of $\mod$, namely
$ij|K\in\mod,~ \tilde{K}\subseteq K ~\Rightarrow~ ij|\tilde{K}\in\mod$, together with the definition of
lifting, allows one to derive that $\bij(\mod)\subseteq\modext_{1}$. Analogously,
we put $\modext_{2}:=\bij(\mod)_{O\lift MO}$ and observe $\modext_{2}^{\downarrow O}=\bij(\mod)$,
$(M\setminus O)\ci (O\setminus M)\,|\, L\,\,[\modext_{2}]$ and $\mod\subseteq\modext_{2}$.
Since $\ffam$ is closed under intersection, $\modext:=\modext_{1}\cap\modext_{2}\in\ffam(MO)$ is an adhesion
of $\mod$ and $\bij(\mod)$.
\end{proof}

Note that the duality mapping commutes with copying: given $\mod\subseteq\sta(N)$
and a bijection $\bij:N\to M$, one has $\bij(\mod)^{\dual}=\bij(\mod^{\dual})$.
On the other hand, it does not commute with marginalization: take $N:=\{a,b,c\}$, $M:=\{a,b\}$, and
$\mod:=\{\,ab|\emptyset \,\}$, in which case $(\mod^{\downarrow M})^{\dual}=\mod$
while $(\mod^{\dual})^{\downarrow M}= \emptyset$. Therefore, a natural
question arised whether, given a CI frame $\ffam$ closed under duality, the frame $\ffam^{\sa}$ is
closed under duality too. The answer is no: both $\semgr^{\sa}$ (Example~\ref{exa.semg-non-dual}) and
$\strum^{\sa}$  (Example~\ref{exa.strum-non-dual}) are not closed
under duality. An analogous question
concerning the operation of tight replication remains open.

\begin{open}\label{open.sa-tight-repli}\rm
We would like to know whether either $\semgr$ or $\strum$ is an example of a CI frame $\ffam$
closed under tight replication whose self-adhesive shrinkage $\ffam^{\sa}$ is not closed under tight replication.
\end{open}

\smallskip
Thus, we know from the basic claim in Lemma~\ref{lem.iter-copy} that, if $\ffam$ is a CI frame
allowing to introduce self-adhesive models, then $\ffam^{\sa}$ is another
such a frame.

\begin{lemma}\label{lem.iteration}\rm
The self-adhesivity operator $\ffam\mapsto \ffam^{\sa}$, applied to
CI frames $\ffam$ closed under copying and marginalization, is
\begin{itemize}
\item {\em anti-extensive\/}: $\ffam^{\sa}\subseteq \ffam$, and
\item {\em isotone\/}: $\ffam_{1}\subseteq\ffam_{2}$ implies
$\ffam^{\sa}_{1}\subseteq\ffam^{\sa}_{2}$.
\end{itemize}
\end{lemma}

\begin{proof}
This is easy to observe on basis of Definition~\ref{def.frame-self-adhesion}.
\end{proof}

\begin{cor}\label{cor.iteration}\rm
Let $\ffam$ be a CI frame closed under copying and marginalization
such that $\profam\subseteq\ffam$. Then $\ffam^{\sa}$ has the same properties and $\profam\subseteq\ffam^{\sa}\subseteq\ffam$.
\end{cor}

\begin{proof}
Use the basic claim in Lemma~\ref{lem.iter-copy}. Then write
$\profam=\profam^{\sa}\subseteq \ffam^{\sa}\subseteq\ffam$ by
Lemmas~\ref{lem.pr-frame}, \ref{lem.pr-self-adhesion} and \ref{lem.iteration}.
\end{proof}

Thus, the new CI frame $\ffam^{\sa}$ yields a tighter approximation of $\profam$ than the
original frame $\ffam$.
If,~moreover, $\ffam$ is closed under intersection, then, by Lemma~\ref{lem.iter-copy}(i), $\ffam^{\sa}$ is also closed under intersection.
Therefore, by observations from Section~\ref{ssec.tutorial}, for any $N$, the family $\ffam^{\sa}(N)$ can be characterized in terms of
CI implications, namely the elements of the respective canonical implicational basis (see Definition~\ref{def.pseudo-closed}).
Elements of the canonical basis break down into permutational
equivalence classes and each of them is equivalent to a collections of so-called  Horn clauses \cite{Hor51}, which have singletons as consequents.
By Corollary~\ref{cor.iteration}, each of these valid implications for $\ffam^{\sa}(N)$ (see Section~\ref{ssec.frame-closure})
\begin{itemize}
\item is a valid probabilistic CI implication, that is, valid implication for $\profam(N)$, while
\item any valid implication for $\ffam(N)$ remains a valid implication for $\ffam^{\sa}(N)$.
\end{itemize}
If additional assumptions on $\ffam$ are accepted then, by utilizing the algorithm in Lemma~\ref{lem.self-adhesion-at-set} as a subroutine, we can compute the $\ffam^\sa$-closure of any model (see Definition~\ref{def.short-closure}).

\begin{lemma}\label{lem.self-adhesion-compute}\rm
Let $\ffam$ be a CI frame closed under copying, marginalization, intersection, lifting, and ascetic extension.
Then the $\ffam^\sa$-closure of any given $\mod\subseteq\sta(N)$ can be computed by the following procedure:
\begin{enumerate}
\item start with $\mod_0 := \mod$,
\item iteratively compute $\mod_{i+1} := \bigcup_{L \subseteq N} \ffam^\sa(\mod_i|L)$ for $i = 0, 1, \dots$,
\item if $\mod_{i+1} = \mod_i$, then output $\ffam^\sa(\mod) := \mod_i$ and stop.
\end{enumerate}
\end{lemma}

\begin{proof}
Note that, by Lemma~\ref{lem.iter-copy}(iii), $\ffam^{\sa}$ is also closed under intersection and ascetic extension and the concept of $\ffam^{\sa}$-closure makes sense (see Definition~\ref{def.short-closure}).
An ascending chain of CI~models $\mod_0 \subseteq \cdots \subseteq \mod_i \subseteq \cdots \subseteq\sta(N)$ is obtained by the iteration, which must eventually stabilize, say at~$\mod_k$.  Since $\ffam^{\sa}(\mod)\in\ffam(N)$ is self-adhesive at
any $L\subseteq N$ relative to $\ffam$, $\mod_{i}\subseteq\ffam^{\sa}(\mod)$ implies,
by Definition~\ref{def.frame-self-adhesion}, $\ffam^{\sa}(\mod_i|L) \subseteq \ffam^{\sa}(\mod)$.
Thus, one can show by induction that $\mod_i \subseteq \ffam^{\sa}(\mod)$ at any stage of the algorithm. Since $\mod_k$ is a fixed point of the map $\mod \mapsto \bigcup_{L \subseteq N} \ffam^\sa(\mod|L)$, it is self-adhesive relative to~$\ffam$ at any subset $L\subseteq N$ (simultaneously), and, thus, it equals to the $\ffam^\sa$-closure of~$\mod$.
\end{proof}

The computation of $\mod_{i+1}$ in step 2 proceeds via Lemma~\ref{lem.self-adhesion-at-set} and requires one $\ffam$-closure computation over a ground set of size $2|N|-|L|$ for each $L \subseteq N$. We cannot provide an upper bound on the number of iterations until $\mod_i$ stabilizes. Matúš~\cite{MatusLongSemgr} gave an example of a model for which the computation of its $\semgr(N)$-closure using a similar iterative procedure requires an exponential number of iterations in~$|N|$. There is no reason to expect the algorithm in Lemma~\ref{lem.self-adhesion-compute} to have better worst-case complexity. %

We wish to know whether or not the self-adhesivity operator $\ffam\mapsto \ffam^{\sa}$
is {\em idempotent}, which means $\ffam^{\sa\sa}=\ffam^{\sa}$.\footnote{By Lemma~\ref{lem.iteration},
$\ffam\mapsto \ffam^{\sa}$ would then be interpretable as the ``interior'' operation, which is a kind of counterpart of the ``closure'' operation discussed in Section~\ref{ssec.tutorial}.}
Since $\semgr$ is closed under intersection, we
were able to use a refined version of the procedure
from Lemma~\ref{lem.self-adhesion-compute}
to compute the iterated self-adhesive models in
$\semgr^{\sa\sa}(N)$ in case $|N|=4$. The basic idea of the algorithm is described in \Cref{app.double-sa}. The result of our long computations was the conclusion that $\semgr^{\sa\sa}(N)=\semgr^{\sa}(N)$ if $|N|=4$.
An analogous question concerning the $\strum$-frame remains open.

\begin{open}\label{conj.idempotent}\rm
We would like to know whether $\mod\in \strum^{\sa}(N)\setminus \strum^{\sa\sa}(N)$ exists in case $|N|=4$.
\end{open}

\section{Details on computations}\label{sec.transition}
This section is devoted to the computational aspects of our paper: here we explain how to use SAT solvers to compute self-adhesions of suitable CI~frames.

\subsection{Summary from the theoretical part}
We first remind the reader which of our theoretical results from
previous sections are used to make the computations possible. We certainly utilized the algorithm to compute the canonical basis (Definition~\ref{def.pseudo-closed}) described in \Cref{app.canon-basis}, which works for
arbitrary Moore family ${\calF}$ of subsets of a finite set $X$.
Our algorithms are, however, tailored to compute the frame-based closure
$\ffam(\mod)$ of a CI-model $\mod$, where $\ffam$ is an appropriate abstract CI frame, as introduced in Section~\ref{ssec.frame-closure}.

In connection with our central concept of self-adhesivity, we proposed an algorithm to compute the self-adhesive closure $\ffam^{\sa}(\mod)$ of a CI model $\mod$ relative to the self-adhesive frame $\ffam^{\sa}$ (Lemma~\ref{lem.iter-copy}), which needs the implementation
of the basic $\ffam$-closure. The algorithm is described
in \Cref{app.double-sa} and its consistency follows from
Lemma~\ref{lem.self-adhesion-compute}. This global
algorithm utilizes as a subroutine the procedure from Lemma~\ref{lem.self-adhesion-at-set} to compute (conditioned) self-adhesive closure $\ffam^{\sa}(\mod|L)$ of $\mod\subseteq\sta(N)$ at a given subset of variables $L\subseteq N$.

\subsection{The SAT problem}

The Boolean {\em satisfiability problem\/} (SAT) deals with Boolean formulas. A~Boolean formula $\phi$ is an expression using Boolean variables $V_i$, $i \in X$, indexed by a finite set~$X$, and the Boolean operations of negation $\neg$, conjunction $\land$, and disjunction~$\lor$. A~mapping $a: X \to \{\texttt{true}, \texttt{false}\}$ is called an \emph{assignment} and it allows one to evaluate a given formula $\phi$ by replacing each variable $V_i$ by the Boolean value $a(i)$ and computing negations, conjunctions, and disjunctions, resulting in a Boolean value denoted by~$\phi(a)$. If $\phi(a)$ is $\texttt{true}$, then $a$ is called a {\em satisfying assignment}. Denote the set of all satisfying assignments of $\phi$ by~$\Sat(\phi)$.

Two Boolean formulas $\phi$ and $\phi'$ are {\em equivalent} if $\Sat(\phi) = \Sat(\phi')$. A~fundamental result on Boolean formulas states that, for every formula $\phi$, there is an equivalent formula $\phi'$ in {\em conjunctive normal form} (CNF). In this normal form, $\phi'$ is written as a conjunction of subformulas, called {\em clauses}. Each clause, in turn, is a disjunction of {\em literals}; and each literal is either a variable or a negated variable.

The SAT problem asks to decide whether a Boolean formula given in CNF has a satisfying assignment. In this form, the problem is famously NP-complete; see \cite[Section~2.3]{AroraBarak}. Variants of this problem include counting all satisfying assignments (referred to as \#SAT) and enumerating all satisfying assignments (AllSAT).
Remarkably efficient software for SAT, \#SAT and AllSAT exists and is constantly improving thanks to regular international competitions indexed at \url{https://www.satcompetition.org/}.

\subsection{Using SAT solvers for self-adhesion computation}
The SAT-based modelling is a powerful tool in combinatorics. Let $X$ be the index set for variables in a Boolean formula. An assignment $a$
may be identified with the subset $\{\, i \in X :\ a(i) = \texttt{true} \,\} \subseteq X$. The set of all satisfying assignments is then an element of\/ ${\cal P}({\cal P}(X))$. Choosing $X:=\sta(N)$, assignments are CI~models over~$N$ and the set of satisfying assignments potentially forms a CI~family.

By recognizing that a CI implication $\bigwedge_{i=1}^k p_i \Rightarrow q$ (where $p_1, \dots, p_k, q \in \sta(N)$) is nothing but a CNF clause $q \vee \neg p_1 \vee \dots \vee \neg p_k$, it becomes apparent that the definition of a semi-graphoid in Section~\ref{sec.sem-frame} for fixed $N$ is, in fact, a Boolean formula $\phi_N$ in CNF with $\Sat(\phi_N) = \semgr(N)$. The frame of structural semi-graphoids is not defined by means of inference rules but, for each finite $N$, the Guigues--Duquenne procedure from \Cref{app.canon-basis} may be used to derive such a defining Boolean formula in conjunctive normal form.

Fix a CI~family\/ $\ffam(N)$ over~$N$ which is closed under intersection and assume that $\phi$ is a Boolean formula in CNF such that $\Sat(\phi) = \ffam(N)$. We now show how the closure operator of\/ $\ffam(N)$ can be computed using SAT solvers. Let $\mod \subseteq \sta(N)$ be any set of CI~statements and we fix $ab|C \not\in \mod$. To test whether $ab|C \in \ffam(\mod)$, it suffices to test whether the CNF formula
\begin{equation}
  \label{eq:CIsat}
  \phi \wedge \bigwedge_{ij|K \in \mod} V_{ij|K} \wedge \neg V_{ab|C}
\end{equation}
is satisfiable. To see this, note that the satisfying assignments are precisely those elements of\/ $\ffam(N)$ which contain $\mod$ and not $ab|C$. If one such model exists, then $\ffam(\mod)$, being the intersection of all models containing $\mod$, does not contain $ab|C$; otherwise the closure must contain~$ab|C$ as all supersets of $\mod$ in $\ffam(N)$ contain it.

It is known from \cite{MatusLongSemgr} that even in the simple frame of semi-graphoids, deciding whether a CI~statement belongs to the closure of a given model may require a number of steps exponential in the ground set size. Despite this, SAT solvers perform exceedingly well in practice as inference engines for CI~implication. We employed the solver CaDiCaL \citesoft{CaDiCaL}, winner of the SAT Race 2019, which uses conflict-driven clause learning. This strategy makes it especially powerful in our setting, where the same formula has to be checked many times under different specializations of the variables (only varying the statement $ab|C$ in \eqref{eq:CIsat}).
%
\smallskip

It is evident from the above description that the closure operator of\/ $\ffam$ is at the core of all our computations related to the self-adhesion operator. We employ a state-of-the-art incremental SAT solver to make this computation as fast as possible; all other code is less critical and mainly organizes the high-level computation, determining which closures are needed and combining the results.

\subsection{Using LP solvers for self-adhesion computation}
Self-adhesivity of polymatroids (see later Definition~\ref{def.sa-polym}) has been used to derive further information inequalities (cf.~\Cref{ssec.self-adhesion,sec.entro-region}). This process follows the same pattern as our algorithms, except that Boolean formulas are replaced by polyhedral cones (whose elements are polymatroids) and, consequently, SAT solvers must be replaced by {\em linear programming\/} (LP) solvers. Such computations have been carried out at a large scale in \cite{Boe23}. The accompanying webpage
\begin{quote}
\url{https://mathrepo.mis.mpg.de/SelfadhesiveGaussianCI/}
\end{quote}
goes into more details about the time and memory demands of SAT versus LP solvers. In the computation of self-adhesions of orientable gaussoids (using a SAT solver) versus self-adhesions of structural gaussoids (using an LP solver), the SAT solver was on average 5 times faster even though it had to deal with formulas with twice as many variables as the LP solver.

However, this comparison is unfair as the SAT solver was integrated into the process and designed to work incrementally, while the LP solver was an external program which had to be called repeatedly. About two thirds of the time it takes to compute the self-adhesive closure using the LP solver was reportedly spent on writing the linear program and starting the external LP solver. Compared to the abundance of excellent SAT solvers, it is more difficult to find a fast LP solver which is freely available software, easy to integrate into domain-specific large-scale computations and supports exact rational arithmetic which is necessary to guarantee the robustness of our results. Our computations also did not exploit recent advances in pre-processing of linear programs \cite{GuoYeungGao}.

\section{Computational results for small numbers of variables}\label{sec.catalogues}
In this section we present the results of our computations.
Our implementation relies on the following software packages: the SCIP optimization suite \citesoft{SCIP2018} and its LP solver SoPlex \citesoft{Soplex2015}; discrete geometry software normaliz \citesoft{Normaliz}; SAT solvers CaDiCaL \citesoft{CaDiCaL}, GANAK \citesoft{GANAK} and \texttt{nbc\_minisat\_all}~\citesoft{TodaSAT}; the CInet framework \citesoft{CInet}; as well as Mathematica \citesoft{Mathematica}. More details are available on our supplementary website \url{https://cinet.link/data/selfadhe-lattices}.
We also give a few instructive examples obtained as their results.
A summary of numerical data about the lattices we consider in case $|N|=4$ is available in \Cref{app.table}.

\subsection{Semi-graphoids over 4 variables}\label{ssec.semgr-4}
Let us start our overview by discussing the family of semi-graphoids $\semgr(N)$ in case $|N|=4$. One has $26\,424$ semi-graphoids over any 4-element variable set and they form 1512 permutational equivalence classes.
The lattice $(\semgr(N),\subseteq )$ has in this case 181 (meet) irreducible elements which fall into 20
permutational types.
On the other hand, it only has 37 coatoms falling down into 10 types; thus, the lattice is not coatomistic.
In particular, the description of $(\semgr(N),\subseteq )$ in terms (of types) of irreducible models is relatively
complicated.

\smallskip
On the other hand, the description of the lattice $(\semgr(N),\subseteq )$ in terms of implicational generators is very easy.
The canonical implicational basis ${\dv G}_{\calF}$ for $\calF=\semgr(N)$ in case $|N|=4$ consists of two
permutational types only, namely
$$
\{\, ij|\emptyset,\, i\ell |j \,\} \to \{\, i\ell|\emptyset,\, ij|\ell \,\} ~~\mbox{and}~~
\{\, ij|k,\, i\ell |jk \,\} \to \{\, i\ell|k,\, ij|k\ell \,\}\,.
$$
Thus, the lattice $(\semgr(N),\subseteq )$ is implicatively perfect. This is not a surprise because the semi-graphoidal frame has, by its definition, a finite ``axiomatic'' characterization. Informally, by such a characterization of a frame $\ffam$ we mean that there is
a finite number of schemes
which, for each $|N|$, yield
an implicational generator for the  $\ffam$-closure operation.

\subsection{Structural semi-graphoids over 4 variables}\label{ssec.strum-4}
The family $\strum(N)$ in case $|N|=4$ contains $22\,108$ models, which fall down into $1285$ permutational equivalence classes. The lattice  $(\strum(N),\subseteq)$ is known to be coatomistic for any $N$. Moreover, in the discussed case $|N|=4$, it has the same 37 coatoms as the semi-graphoidal lattice $(\semgr(N),\subseteq )$. Thus, the description of the family  $\strum(N)$ in terms of irreducible models is much easier in this case than the description of $\semgr(N)$ in these~terms.

As concerns the description of the lattice $(\strum(N),\subseteq )$ in terms of implicational generators, we
computed the canonical implicational basis ${\dv G}_{\calF}$ for $\calF=\strum(N)$ in case $|N|=4$.
What we found out is that it decomposes into 7 permutational types. More specifically, it contains two
semi-graphoidal types mentioned in Section~\ref{ssec.semgr-4} and five additional permutational types listed below:

\begin{itemize}
\item[(E:1)] $\{\, ij|k,\, ik|\ell,\, i\ell|j \,\}  \,\to\, \{\, ij|\ell,\, ik|j,\, i\ell|k \,\}$,
\item[(E:2)] $\{\, ij|k,\, i\ell|j,\, jk|\ell,\,  k\ell|i  \,\}  \,\to\, \{\, ij|\ell,\, i\ell|k,\, jk|i,\, k\ell|j  \,\}$,
\item[(E:3)] $\{\, ij|k\ell,\, ik|\emptyset,\, j\ell|\emptyset,\,  k\ell|ij  \,\}  \,\to\,
\{\, ij|\emptyset,\, ik|j\ell,\, j\ell|ik,\, k\ell|\emptyset  \,\}$,
\item[(E:4)] $\{\, ij|\emptyset,\, ij|k\ell,\, k\ell|i,\,  k\ell|j  \,\}  \,\to\,
\{\, ij|k,\, ij|\ell,\, k\ell|\emptyset,\, k\ell|ij \,\}$,
\item[(E:5)]    $\{\, ij|k,\, jk|i\ell,\, i\ell|j,\,  k\ell|\emptyset  \,\}  \,\to\,
\{\, ij|k\ell,\, jk|i,\, i\ell|\emptyset,\, k\ell|j \,\}$.
\end{itemize}
Note that our results agree with the analysis presented in \cite[\S\,V.A]{Stu21}, or with the rules A3--A7 from \cite{Stu94IJGS}. In particular, the lattice
$(\strum(N),\subseteq)$ is implicatively perfect in case $|N|=4$.
Moreover, both the semi-graphoidal rules and the rules (E:1)--(E:5) can be generalized to any higher cardinality than $|N|\geq 4$
\cite{Bolt21, Bol23}.

\subsection{Self-adhesive semi-graphoids over 4 variables}\label{ssec.sa-semgr-4}
The family $\semgr^{\sa}(N)$ in case $|N|=4$ contains $23\,190$ models, which fall down into $1352$ permutational equivalence classes. The number of (meet) irreducible elements in this lattice $(\semgr^{\sa}(N),\subseteq)$ is 385 and they break into 29 permutational types. Note that, in comparison with $\semgr(N)$, the lattice is smaller but the number of its irreducible elements is higher.

\smallskip
The lattice $(\semgr^{\sa}(N),\subseteq)$
has 31 coatoms in this case, which break into 9 types. Therefore, it is not coatomistic. There is only
one permutational type of semi-graphoidal
coatoms which is not self-adhesive.
It leads to an example of a CI model
in $\strum(N)\setminus\semgr^{\sa}(N)$, the largest one in case $|N|=4$.

\begin{example}\label{exa.vamosi}\rm (\,$\strum(N)\setminus \semgr^{\sa}(N)\neq\emptyset$\,)\\
Consider $N:=\{a,b,c,d\}$ and put
$$
\mod:=\{\,\, ab|c,\,  ab|d,\, ab|cd,\,\,\,  ac|b,\,  ad|b,\,  bc|a,\, bd|a,\,\,\,
cd|\emptyset,\,  cd|a,\, cd|b
\,\,\}\,.
$$
This model $\mod$ is a coatom both in $(\semgr(N),\subseteq)$ and $(\strum(N),\subseteq)$ (which elements coincide in case $|N|=4$). Nevertheless, $\mod\not\in\semgr^{\sa}(N)$. To~observe this we show below that $\mod$ is not self-adhesive at $L:=bd$ relative to $\semgr$. We apply Lemma~\ref{lem.self-adhesion-at-set} with $\ffam:=\semgr$ for this purpose. Fix a bijection $\bij:N\to M:=\{\bar{a},b,\bar{c},d\}$, where $\bij(a):=\bar{a}$, $\bij(b):=b$, $\bij(c):=\bar{c}$, and $\bij(d):=d$. To show that $bc|\emptyset\in \semgr^{\sa}(\mod|L)\setminus\mod$ it is enough to verify
that $bc|\emptyset\in\semgr^{\sa}(\altmod|L)$ for a subset $\altmod:=\{\, ab|d,\, ad|b,\,  bc|a,\, bd|a,\, cd|\emptyset \,\}$ of $\mod$.

By Lemma~\ref{lem.self-adhesion-at-set}, this means that $bc|\emptyset$ belongs to the semi-graphoidal closure
of the union of $\altmod$, $\bij(\altmod)$,
and $[ac,\bar{a}\bar{c}|bd]$.
The following is one of possible
semi-graphoidal derivation sequences:
\begin{enumerate}
\item  $ad|b\in\altmod, [a,\bar{a}|bd] ~\to~ ad|\bar{a}b ~\&~ bd|\bar{a}\in\bij(\altmod)
 ~\to~ ad|\bar{a}$,
\item $ab|d\in\altmod, [a,\bar{a}\bar{c}|bd] ~\to~
[a,b\bar{c}|\bar{a}d] ~\&~ ad|\bar{a}\,\mbox{(1.)} ~\to~ ab|\bar{a}\bar{c}$,
\item $\bar{a}d|b\in\bij(\altmod), [\bar{a},ac|bd] ~\to~
\bar{a}c|ab ~\&~ bc|a\in\altmod  ~\to~ bc|a\bar{a}$,
\item $\bar{a}b|d\in\bij(\altmod), [\bar{a},c|bd] ~\to~
\bar{a}c|d ~\&~ cd|\emptyset\in\altmod  ~\to~ \bar{a}c|\emptyset$,
\item $b\bar{c}|\bar{a}\in\bij(\altmod) ~\&~ ab|\bar{a}\bar{c}\,\mbox{(2.)} ~\to~
ab|\bar{a} ~\&~ bc|a\bar{a}\,\mbox{(3.)} ~\to~ bc|\bar{a} ~\&~ \bar{a}c|\emptyset\,\mbox{(4.)}
~\to~ bc|\emptyset$.
\end{enumerate}
Remark that the derivation like
$ab|d ~\&~ [a,\bar{a}\bar{c}|bd] ~\to~
[a,b\bar{c}|\bar{a}d]$ is based on the application of semi-graphoidal inference in its global mode (see Definition~\ref{def.global-CI}). In this case, this is a shorthand
for two-stage derivation:
$$
[\, a\ci b\,|\,d ~\&~ a\ci \bar{a}\bar{c}\,|\,bd\,]
~\rightarrow~ a\ci b\bar{a}\bar{c}\,|\,d
~\rightarrow~ a\ci b\bar{c}\,|\,\bar{a}d\,.
$$
Note that the only set $L$ at which
the model $\mod$ is self-adhesive is $cd$.
\end{example}

Of course, the derivation sequence from Example~\ref{exa.vamosi} also confirms that the considered semi-graphoid $\altmod$ is not self-adhesive at $L=bd$. This is a kind of a certificate (= a direct proof) of the validity of a later implicational rule (SI:3) for self-adhesive
semi-graphoids; take $i:=b$, $j:=a$, $k:=d$, and $\ell:=c$ in (SI:3).

The simplest example of a self-adhesive semi-graphoid that is not structural is given in Example~\ref{exa.sg-not-pr}. Its minor modification below shows that $\semgr^{\sa}$ is not closed under duality, analogously to the $\profam$-frame (see Example~\ref{exa.non-dual}).

\begin{example}\label{exa.semg-non-dual}\rm (\,$\semgr^{\sa}(N)\setminus\strum(N)\neq\emptyset$,\, $\semgr^{\sa}$ is not closed under duality\,)\\
Put $N:=\{a,b,c,d\}$ and consider the CI model $\mod:=\{\, ab|c,\, ac|d,\, ad|b,\, bc|d \,\}$.
It appears to be a self-adhesive semi-graphoid which is not structural: apply the rule (E:1) presented in Section~\ref{ssec.strum-4} with $i:=a$, $j:=b$, $k:=c$, and $\ell:=d$. Nevertheless,
its dual model $\mod^{\dual}=\{\, ab|d,\, ac|b,\, ad|c,\, bc|a \,\}$ is not self-adhesive. This follows from the below-mentioned property (SE:1) of self-adhesive semi-graphoids, where $i:=a$, $j:=c$, $k:=b$, and $\ell:=d$. Note that $\mod,\mod^{\dual}\not\in\strum(N)$ and, thus, they are not probabilistic CI models.
\end{example}
\smallskip

Note in this context that, if $|N|=4$ then 48 permutational types of models $\mod\in\semgr^{\sa}(N)$ with $\mod^{\dual}\not\in\semgr^{\sa}(N)$ exist.
\smallskip

\smallskip
We also computed the canonical implicational basis ${\dv G}_{\calF}$ for $\calF=\semgr^{\sa}(N)$ in case $|N|=4$.
It~decomposes into 9 permutational types: besides two semi-graphoidal types mentioned in Section~\ref{ssec.semgr-4} it has 7 additional types and we divided them into two subgroups. The implications in the first subgroup can be interpreted as weaker versions of probabilistic {\em equivalences} from  \cite[\S\,V.A]{Stu21}, recalled in Section~\ref{ssec.strum-4}, and, therefore, we use the label SE to denote them:
\begin{itemize}
\item[(SE:1)] $\{\, ij|k,\, ik|\ell,\, jk|i,\, i\ell|j \,\}  \,\to\, \{\, ij|\ell,\, ik|j,\, i\ell|k \,\}$,
\item[(SE:2)]    $\{\, ij|k,\, ij|k\ell,\, ik|j,\, jk|i,\, jk|i\ell,\,
i\ell|\emptyset,\, i\ell|j,\,  k\ell|\emptyset  \,\}  \,\to\,
\{\, k\ell|j \,\}$.
\end{itemize}
Indeed, the implication (SE:1) weakens the property (E:1) %
with an additional antecedent $jk|i$, while the rule (SE:2) weakens the rule (E:5): the 4 additional antecedents are $ij|k\ell$, $ik|j$, $jk|i$, and $i\ell|\emptyset$.

The implications in the second subgroup can be interpreted as weaker versions of probabilistic {\em implications} presented in \cite[\S\,V.B]{Stu21}, for which reason we use the label SI to denote them:
\begin{itemize}
\item[(SI:1)] $\{\,ij|k, ik|j, jk|i, jk|\ell, i\ell|j\,\} \to \{\,ij|\ell\,\}$,
\item[(SI:2)] $\{\,ij|k, ik|j, jk|i, jk|i\ell, i\ell|j\,\} \to \{\,ij|k\ell\,\}$,
\item[(SI:3)] $\{\,ij|k, ik|j, jk|i, i\ell|j, k\ell|\emptyset\,\} \to \{\,i\ell|\emptyset\,\}$,
\item[(SI:4)] $\{\,ij|k, ik|j, jk|i, i\ell|j, k\ell|i\,\} \to \{\,i\ell|k\,\}$,
\item[(SI:5)] $\{\,ij|k, ij|\ell, ik|j, jk|i, k\ell|\emptyset\,\} \to \{\,ij|\emptyset, ik|\emptyset, jk|\emptyset \,\}$.
\end{itemize}
This is a comparison with the implications from \cite[\S\,V.B]{Stu21}:
\begin{itemize}
\item remove the antecedent $jk|i$ in (SI:1)
to get (I:10) or $ik|j$ to get (I:11),
\item remove $jk|i$ in (SI:2)
to get (I:17) or $ik|j$ to get (I:18),
\item remove $ik|j$ in (SI:3)
to get (I:5) or $jk|i$ to get (I:6),
\item remove $ik|j$ in (SI:4)
to get (I:8) or $jk|i$ to get (I:12),
\item remove $jk|i$ and $jk|\emptyset$ in (SI:5)
to get (I:2) or $ij|k$ and $ij|\emptyset$ to get (I:4).
\end{itemize}
The result of our analysis was also the observation that the lattice
$(\semgr^{\sa}(N),\subseteq)$ is {\bf not} implicatively perfect in case $|N|=4$.
Indeed, the implication (SE:2) belongs to its canonical implicational basis but it is not perfect: its antecedent set is the union of CI statements from rules (SI:2) and (SI:3).

As illustrated by Example~\ref{exa.impli-basis}, the rule (SE:2) can further be simplified:
\begin{itemize}
\item[\mbox{$\overline{\mbox{(SE:2)}}$}]    $\{\, ij|k,\, ik|j,\, jk|i,\, jk|i\ell,\,  i\ell|j,\,  k\ell|\emptyset  \,\}  \,\to\,
\{\, k\ell|j \,\}$.
\end{itemize}

We also confirmed independently the validity of the rules (SE:1)--(SE:2) and (SI:1)--(SI:5) within the frame $\semgr^{\sa}$ by finding their certificates
in the form of semi-graphoidal derivation sequences, as illustrated in Example~\ref{exa.vamosi}. Specifically,
we found out that the antecedent set in (SE:1) is not self-adhesive at sets $ij$ and $jk$, while the antecedent sets in (SE:2) and in the rules labelled by SI are not self-adhesive at the sets $ij$, $ik$, and $jk$.

\subsection{Self-adhesive structural semi-graphoids: 4 variables}\label{ssec.sa-strum-4}
The family $\strum^{\sa}(N)$ in case $|N|=4$ contains $20\,968$ models, falling down into $1224$ permutational equivalence classes. The number of (meet) irreducible elements in the lattice $(\strum^{\sa}(N),\subseteq)$ is 85 and they break into 13 permutational types, presented in the form of a concise catalogue
in \Cref{app.sa-strum-irred}.

The same phenomenon was observed as in Section~\ref{ssec.sa-semgr-4}: in comparison with $(\strum(N),\subseteq)$, the number of elements of the lattice decreases while the number of its irreducible elements increases.
The set of coatoms of $(\strum^{\sa}(N),\subseteq)$ is the same as in case of $(\semgr^{\sa}(N),\subseteq)$, namely 31 of them falling into 9 permutational types. In particular, the lattice  $(\strum^{\sa}(N),\subseteq)$  is not coatomistic for $|N|=4$.

We also computed the canonical implicational basis ${\dv G}_{\calF}$ for $\calF=\strum^{\sa}(N)$ in case $|N|=4$.
It~decomposes into 12 permutational types: it consists of two semi-graphoidal types mentioned in Section~\ref{ssec.semgr-4}, five types (E:1)--(E:5) from Section~\ref{ssec.strum-4} and five types
(SI:1)--(SI:5) presented in Section~\ref{ssec.sa-semgr-4}.
It makes no problem to observe that all these implications are perfect, implying that the lattice $(\strum^{\sa}(N),\subseteq)$ is implicatively perfect in case $|N|=4$.

Our computation also revealed 30 types of models
$\mod\in\strum^{\sa}(N)$ such that
$\mod^{\dual}\not\in\strum^{\sa}(N)$. One of these examples
is as follows.

\begin{example}\label{exa.strum-non-dual}\rm (\,$\strum^{\sa}$ is not closed under duality\,)\\
Put  $N:=\{a,b,c,d\}$. The model $\mod:=\{\, ab|d, ac|d, ad|b, bc|\emptyset, bc|d  \,\}$ belongs to $\strum^{\sa}(N)$ while its dual model not: $\mod^{\dual}:=\{\, ab|c, ac|b, ad|c, bc|a, bc|ad \,\}\not\in\strum^{\sa}(N)$. The latter
fact follows from the rule (SI:2) above where $i:=a$, $j:=c$, $k:=b$, and $\ell:=d$.
\end{example}

\begin{figure}[t]
\centering
\includegraphics[scale=0.6]{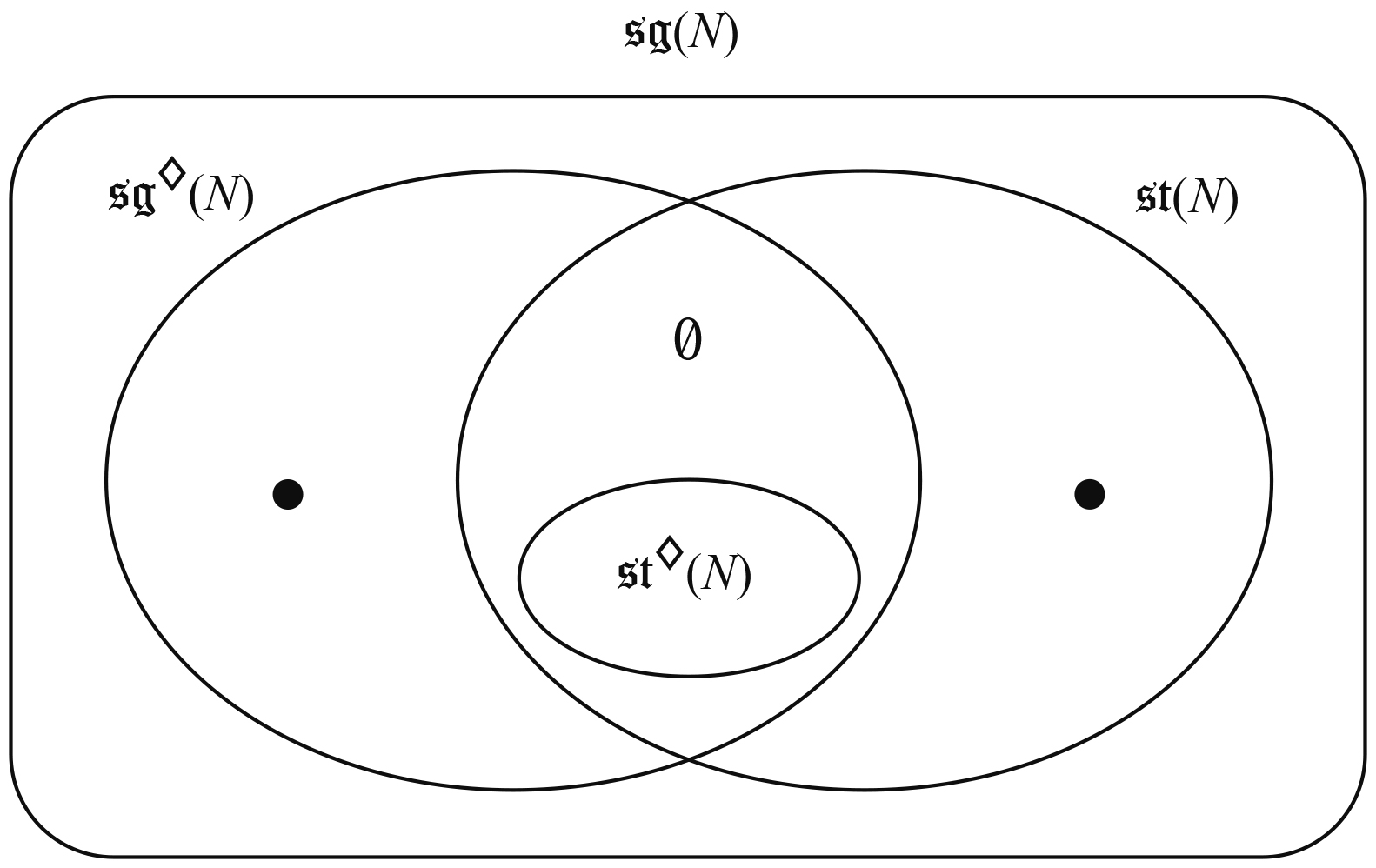}
\caption{Relation between considered families of CI models in case $|N|=4$.}
\label{fig4}
\end{figure}

A kind of surprise for us was the observation that $\semgr^{\sa}(N)\cap \strum(N)=\strum^{\sa}(N)$ if $|N|=4$
(see Figure \ref{fig4} for illustration); thus, self-adhesivity relative to $\semgr$ and relative to $\strum$
yield the same restriction. Therefore, we have a natural question.

\begin{open}\label{que.enough-sg-4}\rm
We would like to know whether the relation $\semgr^{\sa}(N)\cap \strum(N)=\strum^{\sa}(N)$ holds in general.
\end{open}

\subsection{Graphoids, compositional semi-graphoids, and compositional graphoids: 4 variables}\label{ssec.grap-4}

A result of our computations was the finding
that, in case $|N|=4$, the family $\gra(N)$ contains 6482 elements, which fall down into 421 permutational types. The number of (meet) irreducible models in $\gra(N)$ is 408 and they decompose into 32 types. The lattice  $(\gra(N),\subseteq)$ has 14 coatoms of 4 types; thus, it is not coatomistic. The family $\gra^\sa(N)$ in case $|N| = 4$ coincides with $\gra(N)$.
Note that we also found that, in case $|N| = 4$, precisely $369$ permutational types of graphoids are probabilistically representable.

Since, for any $N$, $(\gra^{\dual}(N),\subseteq)$ is order-isomorphic to $(\gra(N),\subseteq)$, the numbers of elements, irreducible elements, and coatoms of $(\gra^{\dual}(N),\subseteq)$ are the same as in case of graphoids. What is, however, different is that $\gra^\dual(N) \not= (\gra^\dual)^\sa(N)$ in case $|N| = 4$. The lattice $(\gra^\dual)^\sa(N)$ has only 6182 (instead of 6482) elements in 404 permutational types. It has 14 coatoms in 4 types and 428 irreducible elements falling to 33 types. This is the first concrete example of a frame $\ffam$
such that both $\ffam$ and $\ffam^{\dual}$ admit self-adhesion and
$(\ffam^\dual)^\sa \not= (\ffam^\sa)^\dual$, i.e., the operations of duality and self-adhesion on a CI~frame do not commute in general.

The family $\cogr(N)=\gra(N)\cap\gra^{\dual}(N)$ has 2084
elements in case $|N|=4$ which fall into 157 types. It has 470 irreducible models of 30 types and 6 coatoms of 1 type. In particular, the lattice is not coatomistic. Of course,
$\cogr(N)\subseteq\gra(N)\subseteq \semgr^{\sa}(N)$ in this case. Like $\gra(N)$ in case $|N| = 4$, $\cogr(N) = \cogr^\sa(N)$.

\subsection{The case of 5 variables}\label{ssec.5-variab}
We had the following special goal in this case. There are $1319$ permutational types of coatoms of the lattice $(\strum(N),\subseteq)$ of structural semi-graphoids over 5 variables.
Note that the overall number of coatoms of this lattice is $117\,978$.

In general, the coatoms of this lattice correspond to the extreme rays of the cone of standardized supermodular functions, that is, to supermodular functions $m:{\cal P}(N)\to {\dv R}$ satisfying $m(K)=0$ if $K\subseteq N$, $|K|\leq 1$, see
\cite[Corollary\,30]{SK16}. Thus, the coatoms
for $|N|=5$ were obtained by computing the extreme rays of that cone \cite{SBK00}.
We raised a couple of questions:
\begin{itemize}
\item Which of the permutational types belong to $\semgr^{\sa}(N)$?
\item Which of them belong to $\strum^{\sa}(N)$?
\end{itemize}
Surprising answers were that only 154 of those types belong to $\semgr^{\sa}(N)$ and these coincide with the types belonging to $\strum^{\sa}(N)$. This observation seems to support the conjecture indicated in Open problem~\ref{que.enough-sg-4}. Consequences of those results are discussed in detail
in Section~\ref{sec.entro-region}.

\section{Application to entropic region delimitation}\label{sec.entro-region}
Given a discrete random vector
$\bxi=[\xi_{i}]_{i\in N}$ over a variable set $N$
and $I\subseteq N$,
the (Shannon's) {\em entropy\/} (of its sub-vector) for $I$ is the number
$$
h_{\ibxi}(I) \,:=\,
\sum_{x\,\in \,\isfX_{N}:\, p_{N}(x)>0}\,
p_{N}(x)\cdot \log \left(\frac{1}{p_{I}(x)}\right)\qquad
\mbox{\rm (see Section~\ref{ssec.CI-families})}.
$$
Well-known facts are that
the {\em entropy function}
$h_{\ibxi}:{\cal P}(N)\to {\dv R}$ is submodular, non-decreasing with respect to inclusion, and
satisfies $h_{\ibxi}(\emptyset)=0$.
Thus, in terms of
Section~\ref{sec.struc-sem}, it is (the rank function of) a polymatroid over $N$.

A polymatroid $h\in {\dv R}^{{\cal P}(N)}$ is then called {\em entropic} if  $h=h_{\ibxi}$ for some discrete random
vector $\bxi$ over $N$. It is called {\em almost entropic} if it is a limit of entropic polymatroids (in the Euclidean
topology on ${\dv R}^{{\cal P}(N)}$). The {\em entropic region}, also named the {\em entropy region}, is the set of all
entropic polymatroids and the {\em almost entropic region}
is the set of all almost entropic polymatroids; the latter set is known to be a closed cone in ${\dv R}^{{\cal P}(N)}$ \cite[\S\,15.1]{Yeu08}.

{\em Information inequalities}, that is, linear inequalities
valid for vectors in the (almost) entropic region, are one of the research topics in information theory. In this paper we discuss {\em unconstrained\/} information inequalities, valid for almost entropic vectors.
Besides classic Shannon's inequalities, defining
the cone of polymatroids, a number of independent such inequalities
has been found; the first of them in 1998 \cite{ZY98}. The techniques to
derive these inequalities follow the scheme of the so-called
{\em copy lemma} \cite{DFZ11} and
inspired the concept of a self-adhesive polymatroid by Mat\'{u}\v{s} \cite{Mat07DM}, recalled below in analogy with Definition~\ref{def.frame-self-adhesion}.

\begin{definition}[self-adhesive polymatroid]\label{def.sa-polym}~\rm \\[0.3ex]
Given a bijection $\bij: N\to M$,
the {\em $\bij$-copy} of a polymatroid $h\in {\dv R}^{{\cal P}(N)}$ is
the polymatroid $h\circ\bij_{-1}\in {\dv R}^{{\cal P}(M)}$.
A pair of sets $N,M$ is called {\em modular} in a  polymatroid $\hat{h}\in {\dv R}^{{\cal P}(NM)}$ if\/ $\hat{h}(N)+\hat{h}(M)=\hat{h}(NM)+\hat{h}(N\cap M)$.
A polymatroid $\hat{h}\in {\dv R}^{{\cal P}(NM)}$ extends
$h\in {\dv R}^{{\cal P}(N)}$ if $h$ is the restriction of $\hat{h}$ to ${\cal P}(N)$. %
We say that a polymatroid $h\in {\dv R}^{{\cal P}(N)}$ is {\em self-adhesive at a set $L\subseteq N$\/} if,

\vspace*{-6mm}
\begin{eqnarray*}
\lefteqn{\hspace*{0mm}\forall\, \bij:N\to M ~\mbox{bijection with $L=N\cap M$ and $\bij_{\,|L}=\id_{L}$~}}\\
 && \exists\,\,
 \mbox{a polymatroid $\hat{h}\in {\dv R}^{{\cal P}(NM)}$  extending both $h$ and its $\bij$-copy}\\
 && ~~ \mbox{such that $N,M$ is a modular pair in $\hat{h}$}\,.
\end{eqnarray*}

\vspace*{-2mm}
\noindent A polymatroid over $N$ is {\em self-adhesive} if it is
self-adhesive at every $L\subseteq N$.
\end{definition}

Recall from
Section~\ref{sec.struc-sem} that a polymatroid $h\in {\dv R}^{{\cal P}(N)}$
induces a CI model
$\mod\subseteq\sta(N)$ as follows:
$ij|K\in\mod$ if and only if $\Delta h (ij|K)=0$.

\begin{lemma}\label{lem.sa-polyma}\rm
The following statements are true.
\begin{itemize}
\item[(a)] The set of self-adhesive polymatroids is a polyhedral
cone in ${\dv R}^{{\cal P}(N)}$.
\item[(b)] Every almost entropic polymatroid is self-adhesive.
\item[(c)]
Every self-adhesive polymatroid $h$ over $N$
induces a CI model in $\strum^{\sa}(N)$.
\end{itemize}
\end{lemma}

\begin{proof}
To show (a) is true, consider a set $L\subseteq N$ and the respective bijection $\bij:N\to M$. Realize that the set of polymatroids $\hat{h}\in {\dv R}^{{\cal P}(NM)}$
in which $N,M$ is a modular pair and that satisfy $\hat{h}(A)=\hat{h}(\bij(A))$ for $A\subseteq N$ is a polyhedral cone. A polymatroid $h\in {\dv R}^{{\cal P}(N)}$ is self-adhesive at $L$ if and only if it belongs the projection of that set of vectors $\hat{h}$ from ${\dv R}^{{\cal P}(NM)}$  to ${\dv R}^{{\cal P}(N)}$. %
Basic facts from polyhedral geometry are that
the class of polyhedral cones is closed under projections and
finite intersections. The application of these facts yields (a).

For (b) we first realize that any entropic polymatroid is self-adhesive. The arguments for this from \cite[Corollary\,1]{Mat07DM} are analogous to those in our proof of Lemma~\ref{lem.pr-self-adhesion}. Then (a) implies that set of self-adhesive
polymatroids over $N$ is closed, which forces that it contains the closure of the entropic region.

As concerns (c), recall from the remark n.~\ref{remarks:2.1} in the end of Section~\ref{sec.struc-sem} that  $\mod\in\strum(N)$ if and only if it is induced by a polymatroid \mbox{over $N$}.
Given $\mod\subseteq\sta(N)$ induced by a self-adhesive
polymatroid $h\in {\dv R}^{{\cal P}(N)}$, one needs to observe that $\mod$ is self-adhesive at (any) $L\subseteq N$ relative to $\strum$.
Having fixed the corresponding bijection $\bij:N\to M$, take the
polymatroid $\hat{h}\in {\dv R}^{{\cal P}(NM)}$ from
Definition~\ref{def.sa-polym} and it is straightforward to
verify that it induces $\hat{\mod}\in\strum(NM)$ which is
an adhesive extension of $\mod$ and $\bij(\mod)$.
\end{proof}

Lemma~\ref{lem.sa-polyma} has the following consequence.

\begin{cor}\label{cor.ray-separ}\rm
If a polymatroid $h$ over $N$ induces a CI model $\mod$ outside $\strum^{\sa}(N)$ then there is a valid (unconstrained) information inequality which is not true for $h$ and none of positive multiples of $h$ is almost entropic.
\end{cor}

\begin{proof}
Realize that any facet-defining inequality for the cone of self-adhesive polymatroids is a valid information inequality.
\end{proof}

To interpret the results from Section~\ref{ssec.5-variab} we recall some geometric facts. A polymatroid $h$ over $N$, $|N|\geq 1$, is called {\em tight} if $h(N)=h(N\setminus i)$ for any $i\in N$
and {\em modular} if $\Delta h (ij|K)=0$ for any $ij|K\in\sta(N)$.
Every polymatroid $h$ then decomposes into its modular part and its tight part
$\tilde{h}$ inducing the same CI model as $h$.
Moreover, there exists a linear one-to-one mapping between tight polymatroids
and standardized supermodular functions which preserves the induced CI model, see \cite[formulas (32) in \S\,7.2]{SK16}.
This allows one to observe, using Corollary~30 in
\cite[Appendix\,B]{SK16}, that a polymatroid $h$ over $N$
generates a non-modular (= tight) extreme ray of the polymatroidal cone if and only if
it induces a coatom in the lattice $(\strum(N),\subseteq)$.

Thus, in light of Corollary~\ref{cor.ray-separ}, the computational results from Section~\ref{ssec.5-variab} imply that most of the $117\,978$ extreme rays of the cone of tight polymatroids over a 5-element set are outside the almost entropic region.
\smallskip

Being motivated by this perspective, we put (for every variable set $N$):
\begin{eqnarray*}
\lefteqn{\overline{\profam}(N) ~:=~ \{\,
\mod\subseteq\sta (N)\, :\ \mbox{$\mod$ is induced by}}\\
&&\hspace*{36mm}\mbox{an almost entropic polymatroid over $N$}\,\}\,,
\end{eqnarray*}
which defines the {\em almost entropic\/} CI frame $\overline{\profam}$. It seems to coincide with
the ``approximate probabilistic frame" suggested by one of the reviewers: one has $\mod\in\overline{\profam}(N)$ if
$\exists\,\delta>0~~\forall\, 0<\varepsilon<\delta$ there exists a discrete random vector $\bxi(\varepsilon)$ over $N$ such that $|\Delta h_{\ibxi(\varepsilon)} (ij|K)|<\varepsilon$ for $ij|K\in\mod$ and
$|\Delta h_{\ibxi(\varepsilon)} (ij|K)|>\delta$ for $ij|K\in\sta(N)\setminus\mod$.

\begin{open}\label{task.almost-etropic}\rm
We would like to know whether one has $\profam=\overline{\profam}$ or not.
\end{open}

It is easy to see that the frame
$\overline{\profam}$ is closed under copying, marginalization, intersection, and lifting.
Since
$\sta(N)\in\profam(M)\subseteq\overline{\profam}(M)$ if $N\subseteq M$, it is also closed under ascetic extension. Therefore, in light of Definition~\ref{def.short-closure},
Open problem~\ref{task.almost-etropic} is equivalent to the question whether the probabilistic CI implication on $\sta(N)$ coincides with the (almost entropic) $\overline{\profam}$-closure operation (for any $N$). This was basically a question raised in a private communication by our colleague
Rostislav Matveev. We incline to believe that there are valid
CI implications which are not true in the frame $\overline{\profam}$.

The suspected probabilistic CI rules from \cite[\S\,V.B]{Stu21} which may appear to be examples of invalid implications in the $\overline{\profam}$-frame are the rules (I:1), (I:7), (I:13), and (I:15). The reason for this guess is that to derive these implications one inevitably needs the first conditional Ingleton inequality  \cite{ZY97} and this particular conditional inequality need not hold for almost entropic vectors \cite[\S\,V.1]{KR13}.
\medskip

\noindent {\em Remarks.}
Let us add three comments.
\begin{enumerate}[leftmargin=2em, rightmargin=1em]
\item It follows from Corollary~\ref{cor.ray-separ} that any CI implication
valid in $\strum^{\sa}$-frame can be viewed as a consequence of an unconstrained information inequality.
Note, however, that {\em conditional information inequalities} are also studied in information theory, whose validity is restricted to and guaranteed for entropic vectors only \cite{KR13}. All the so-far known probabilistic CI implications were derived from such conditional inequalities \cite{Stu21}.
\item One of the reviewers raised the question of how many of the remaining 154 types are known to be entropic for sure. The best lower bound we are aware of is provided by {\em linearly representable} polymatroids which are known to be entropic. The cone generated by linear $5$-polymatroids was characterized in \cite{DFZ10} both in terms of linear inequalities and in terms of its extreme rays: they computed 7943 of such extreme rays falling into 162 permutational types. Of these 162 types exactly 138 coincide with extreme rays of the polymatroidal cone. Thus there are only 16 extreme rays of the $5$-polymatroid cone for which it is unknown whether they are entropic or not. We provide the corresponding coatoms of $\strum^\sa(5)$ on our supplementary materials website.
\item
Progress on the remaining 16 types can be made using stronger probabilistic rules from \cite[\S\,V.B]{Stu21}. This is indeed the case: we found further 11 types which cannot be probabilistic
because their induced CI models $\mod$ do not satisfy the (extended version of the) probabilistic rule (I:2) saying that
$$
ij|M,\, ik|\ell M,\, k\ell|iM,\, k\ell|jM\,\in\mod
~\Rightarrow ~ ik|M,\, k\ell|M\,\in\mod\,.
$$
But only 5 of these 11 types can reliably be excluded from being almost entropic as they do not satisfy
the above implication with $M=\emptyset$; these finer conclusions follow from the (method of) rule derivation in \cite[\S\,V.B]{Stu21}. We leave this observation for future work.
\end{enumerate}
\section{Conclusions}\label{sec.additional}
In this paper, we adapted the self-adhesivity concept, introduced originally in the context of polymatroids, to abstract CI models and explored its algebraic and order-theoretic properties. Our intention was to use existing axioms for probabilistic CI to derive stronger axioms through self-adhesion. Unlike a former paper \cite{Boe23}, in which this strategy was applied to gaussoids, this paper emphasized the CI frames which are closed under set-theoretical intersection. Consequent lattice structure of families of models is the key to our computational advances. In particular, using state-of-the-art SAT solvers \citesoft{CaDiCaL}, we were easily able to compute $\semgr^\sa(N), \strum^\sa(N), \gra^\sa(N)$ and $\cogr^\sa(N)$ for $|N| \le 4$ and to derive axiomatizations for each of these families using the techniques from the formal concept analysis. On the basis of these axioms, we were able to compare the power of our approach with earlier work on CI axioms. In the context of information theory, we were able to exclude a large proportion (almost 90\%) of permutational types of extreme rays of the polymatroidal cone for $|N|=5$ from being almost entropic. These facts can in turn motivate new unconditional information inequalities.
\smallskip

It was also demonstrated in \cite{Boe23} that computations with gaussoids over 5 variables are still possible, for the number of gaussoids $60\,212\,776 \approx 60 \times 10^6$ \cite{BDKS19} is still manageable.
As concerns the semi-graphoids, based on a few runs of a probabilistic model counter \citesoft{GANAK} with error probability $\delta = 0.05$ we conjecture $|\semgr(5)| = 12\,144\,387\,289\,848 \approx 12 \times 10^{12}$ which rules out exhaustive computations on the set of $5$-semigraphoids.\footnote{The exact number of 5-semigraphoids remains unknown because the exact model counters quickly run out of memory on the axioms of 5-semigraphoids.}
Despite that, we believe that the coatoms and irreducible elements of this lattice can be obtained by a clever arrangement of the computations and their self-adhesivity status would be valuable information in the delimitation of~$\profam(5)$.
Thus, the following remains our goal.

\begin{open}\label{task.five-semigraphoids}\rm
We would like to find all coatoms and irreducible elements of $(\semgr(N),\subseteq)$ in case  $|N|=5$.
\end{open}

Note in this context that, according to a counter-example from \cite{HMSSW08CPC}, in case $|N|=5$, coatoms
of $(\semgr(N),\subseteq)$ exist which are outside $\strum(N)$.
\medskip

We also verified computationally that $\semgr^\sa(N) = \semgr^{\sa\sa}(N)$ in case \mbox{$|N|=4$}. Computing this ``second-order self-adhesion'' required a significant  implementation effort and took much longer than any of our other computations: while the results on single self-adhesions documented in \Cref{app.table} took roughly 18 minutes in total to compute, the double self-adhesion $\semgr^{\sa\sa}(4)$ required 6 hours. To explain this sudden increase, recall from \Cref{app.double-sa} that computing the closure operator with respect to $\ffam^{\sa\sa}(N)$ requires the closure for $\ffam^{\sa}(M)$ for $|N| \le |M| \le 2|N|-2$, which in turn is computed using the closure operators $\ffam(O)$ for $|M| \le |O| \le 2|M|-2$. In the case of semi-graphoids over $|N|=4$, this requires semi-graphoid closures over ground sets of size up to~$10$. At the maximum $|O|=10$, the corresponding Boolean formula used to compute closures has 184\,320 clauses in 11\,520 variables. %
We expect that $\semgr^\sa(N) = \semgr^{\sa\sa}(N)$ does not hold for $|N|=5$, but an exhaustive computation is out of our reach.

\smallskip

In this context, it may be worthwhile to study in future some more advanced notions of self-adhesivity, which allow for a different growth in the number of ground set elements over which the closures have to be computed than iterated self-adhesion (which doubles the number of elements each time). Such notions were proposed, again in the context of polymatroids, by Csirmaz \cite{Csi14IEEE}, who called the linear information inequalities they implied \emph{book inequalities}. Formally, one can fix any $k \ge 1$ and take a CI~model $\mod \in \ffam(N)$ and $L \subseteq N$. One asks whether there exists a sequence $\mod_0:=\mod, \mod_1, \dots, \mod_k$ where $\mod_i$ on ground set $N_i$ is an $L$-copy of $\mod_0$ and $N_i \cap N_j = L$ for all pairs $i\not= j$ and such that there exists $\overline{\mod} \in \ffam\left(\bigcup_{i=0}^k N_i \right)$ which restricts to $\mod_i$ on $N_i$ and satisfies $N_{i_1}\dots N_{i_p} \ci N_{j_1}\dots N_{j_q} \,|\, L \,\, [\overline{\mod}]$ for any pair of disjoint subsets $\{\,i_1, \dots, i_p\,\}$ and $\{\,j_1, \dots, j_q\,\}$ of~$\{\,0, \dots, k\,\}$. For $k=1$, one recovers our notion of self-adhesivity.
The computational advantages of CI frames whose CI families are lattices can also be exploited in this setting, and the analogue to Matúš's fundamental result \cite{Mat07DM} on the closedness of $\profam$ under such generalized self-adhesivity follows from the same ideas as presented in our proof of Lemma~\ref{lem.pr-self-adhesion}.
\smallskip

We end with an open question about finite axiomatizability. We have seen in Section~\ref{ssec.semgr-4} that the semi-graphoidal frame $\semgr$ is special in the sense that the canonical implicational basis for $(\semgr(N),\subseteq )$ for any $|N|$ follows one scheme only. One can interpret this by saying that semi-graphoids have a finite ``axiomatization'' (with a single axiom).

\begin{open}\label{que.sa-semgr-axiomatize}\rm
We would like to know whether
$\semgr^{\sa}$ admits a finite ``axiomatic'' characterization, too.
\end{open}

\subsubsection*{Acknowledgements}
{\small Most of this work was inspired by the research of our late colleague Fero Mat\'{u}\v{s}.
Tobias Boege was partially supported by the Academy of Finland grant number 323416 and by the Wallenberg Autonomous Systems and Software Program (WASP) funded by the Knut and Alice Wallenberg Foundation.
The authors discussed in person the topic of this paper at two meetings in 2022, one of them was the DAGSTUHL seminar 22301 ``{\em Algorithmic Aspects of Information Theory}''.
We also thank both reviewers of DAM journal for the valuable comments
which helped to improve the quality of presentation of our results.}

\bibliographystyle{tboege}
\bibliography{refs.bib}
\medskip

\nocitesoft{SCIP2018,Soplex2015,Normaliz,CInet,CaDiCaL,GANAK,TodaSAT,Mathematica}
\bibliographystylesoft{tboege}
\bibliographysoft{refs.bib}

\appendix

\section{Algorithm to compute the canonical basis}\label{app.canon-basis}

This is a pseudo-code of the algorithm we used to obtain the canonical generator ${\dv G}_{\calF}$
on basis of a Moore family $\calF$ of subsets of a finite set $X$ or its closure operation $\cl_{\calF}$:
\begin{enumerate}[itemsep=0.5em, parsep=0pt]
\item choose the maximal parameter $m$, $-1\leq m\leq|X|$, such that any set $Y\subseteq X$ with $|Y|\leq m$ is closed
(that is, $Y\in\calF$);
put ${\dv G}:=\emptyset\subseteq {\cal P}(X)\times {\cal P}(X)$,
\item unless $m=|X|$ raise $m:=m+1$; if $m=|X|$ then
put ${\dv G}_{\calF}:={\dv G}$ and exit,
\item test $Y\subseteq X$ with $|Y|=m$ whether they are pseudo-closed (Definition~\ref{def.pseudo-closed});\footnote{Note that pseudo-closed strict subsets of $Y$ are encoded in current version of ${\dv G}$.}\\
enlarge ${\dv G}:= {\dv G}\cup \{\, (Y\to\cl_{\cal F}(Y)\setminus Y)\,:\ Y\in \wp(\calF), \, |Y|=m\,\}$,
\item then compute  ${\dv G}^{\triangleright}$ and branch:
\begin{itemize}[leftmargin=2em, rightmargin=1em]
\item[(a)] if $|{\dv G}^{\triangleright}|>|\calF|$ then go to 2,
\item[(b)] if $|{\dv G}^{\triangleright}|=|\calF|$ then ${\dv G}_{\calF}:={\dv G}$
and exit.
\end{itemize}
\end{enumerate}
Note that in Step 1, when ${\cal F}=\ffam(N)$
for a suitable CI family $\ffam$, one typically
has $m=1$.  By~the construction of ${\dv G}$, we know that ${\dv G}\subseteq {\calF}^{\triangleleft}$
and ${\calF}\subseteq{\dv G}^{\triangleright}$ in every phase of the algorithm.
In~Step 4(b), we are sure that ${\dv G}$ is the implicational generator of ${\cal F}$ by the criterion
mentioned above Example~\ref{exa.impli-gen}.
\medskip

What is described above is just one of basic versions of the procedure to compute the canonical basis. There are surely various more refined and efficient versions of this procedure \cite[\S\,3.3.2]{GO16}; the paper \cite{DS11DAM} was particularly devoted to the comparison of computational complexities of such algorithms.

\section{Algorithm to test iterated self-adhesivity}\label{app.double-sa}
Consider a CI frame $\ffam$ closed under copying, marginalization, intersection, lifting,
tight replication, and ascetic extension.

The basis of the algorithm is a subroutine to compute the $\ffam^{\sa}$-closure of a given model $\hat{\altmod}\in\ffam(M)$, which is a modification of the procedure in Lemma~\ref{lem.self-adhesion-compute}.

\noindent Subroutine  $\ffam^{\sa}(\hat{\altmod})$ [the input is $\hat{\altmod}\in\ffam(M)$]
\begin{enumerate}[itemsep=0em]
\item put $\altmod:=\hat{\altmod}$
\item apply to $\altmod$ the following cycle governed by $L\subseteq M$, $|L|=2,\ldots, |M|-2$:
\begin{itemize}[leftmargin=2em, rightmargin=1em, itemsep=0em, parsep=0.3em]
\item[(a)] compute the self-adhesive closure $\ffam^{\sa}(\altmod|L)$ as described in Lemma~\ref{lem.self-adhesion-at-set},
\item[(b)] test whether $\ffam^{\sa}(\altmod|L)\setminus\altmod\neq\emptyset$\,:
    \begin{itemize}[leftmargin=2em, rightmargin=1em, parsep=0pt]
    \item if yes then put  $\altmod:=\ffam^{\sa}(\altmod|L)$, exit the cycle, and go to 2,
    \item if not then continue within the cycle with the next set $L$,
    \end{itemize}
\end{itemize}
\item put $\ffam^{\sa}(\hat{\altmod}):=\altmod$ and exit.
\end{enumerate}

Note that in Step 2, by Corollary~\ref{cor.lift-self-adhes-01} and
Lemma~\ref{lem.par-ext-n-1}, we need not test sets $L$ of cardinalities $0,1,|M|-1$, and $|M|$. As concerns updating of $\altmod$ in Step~2(b), by Definition~\ref{def.frame-self-adhesion}, $\ffam^{\sa}(\altmod|L)\in\ffam(M)$.

The main procedure is analogous and differs in small details only. The main difference is that
Lemma~\ref{lem.self-adhesion-at-set} is this time applied to $\ffam^{\sa}$ instead of $\ffam$.
Note that this modification is possible thanks to Lemma~\ref{lem.iter-copy}(iii).

\noindent Main procedure  $\ffam^{\sa\sa}(\hat{\mod})$ [the input is $\hat{\mod}\in\ffam^{\sa}(N)$, where $|N|=4$]
\begin{enumerate}
\item put $\mod:=\hat{\mod}$
\item apply to $\mod$ the following cycle governed by $E\subseteq N$, $|E|=2,3$:
\begin{itemize}[leftmargin=2em, rightmargin=1em]
\item[(a)] compute the self-adhesive closure $\ffam^{\sa\sa}(\mod|E)$ using Lemma~\ref{lem.self-adhesion-at-set}, where the subroutine is called to compute the $\ffam^{\sa}$-closures,
\item[(b)] test whether $\ffam^{\sa\sa}(\mod|E)\setminus\mod\neq\emptyset$\,:
    \begin{itemize}[leftmargin=2em, rightmargin=1em, parsep=0pt]
    \item if yes then put  $\mod:=\ffam^{\sa\sa}(\mod|E)$, exit the cycle, and go to 2,
    \item if not then continue within the cycle with the next set $E$,
    \end{itemize}
\end{itemize}
\item if $\mod\neq\hat{\mod}$ then
$\hat{\mod}\not\in\ffam^{\sa\sa}(N)$
and its closure is $\ffam^{\sa\sa}(\hat{\mod}):=\mod$; exit.
\end{enumerate}

Note that, in Step 2, by Lemma~\ref{lem.iter-copy}(i)(ii) and Corollary~\ref{cor.lift-self-adhes-01}, we need not test sets $E$ of cardinalities $0,1$, and $4$.
On the other hand, since we are not sure whether $\ffam^{\sa}$ is closed under tight replication (see Open problem~\ref{open.sa-tight-repli}) we have to test sets $E$ of cardinality~$3$.
So, the number of sets $E$ in line 2
is at most~10.
\smallskip

To check double self-adhesivity over $4$ variables one needs to apply the subroutine for $|M|=5,6$, in which case the number of sets $L$ in step 2 of the subroutine is at most 50. Also, to compute
$\ffam^{\sa}(\altmod|L)$  one applies the $\ffam$-closures over at most 10 variables.

\section{Table of our results in case of 4 variables} \label{app.table}

\begin{center}
\begin{tabular}{|c|cc|cc|cc|} \hline
\mbox{\rm family} & \mbox{\rm models} & \mbox{\rm in types}& \mbox{\rm irreducibles}& \mbox{\rm in types}&
 \mbox{\rm coatoms}& \mbox{\rm in types}\\ \hline\hline
$\semgr(N)$           & $26\,424$ & $1512$ & $181$ & $20$ & $37$ & $10$ \\ \hline
$\semgr^{\sa}(N)$     & $23\,190$ & $1352$ & $385$ & $29$ & $31$ & $9$  \\ \hline \hline
$\strum(N)$           & $22\,108$ & $1285$ & $37$  & $10$ & $37$ & $10$ \\ \hline
$\strum^{\sa}(N)$     & $20\,968$ & $1224$ & $85$  & $13$ & $31$ & $9$  \\ \hline \hline
$\gra(N)$             & $6482$    & $421$  & $408$ & $32$ & $14$ & $4$  \\ \hline
$\gra^\sa(N)$         & $6482$    & $421$  & $408$ & $32$ & $14$ & $4$  \\ \hline \hline
$\gra^\dual(N)$       & $6482$    & $421$  & $408$ & $32$ & $14$ & $4$  \\ \hline
$(\gra^\dual)^\sa(N)$ & $6182$    & $404$  & $428$ & $33$ & $14$ & $4$  \\ \hline \hline
$\cogr(N)$            & $2084$    & $157$  & $470$ & $30$ & $6$  & $1$  \\ \hline
$\cogr^\sa(N)$        & $2084$    & $157$  & $470$ & $30$ & $6$  & $1$  \\ \hline
\end{tabular}
\end{center}

\section{Irreducible self-adhesive structural models}\label{app.sa-strum-irred}

Here we give the list of permutational types of (meet) irreducible elements of the lattice $(\strum^{\sa}(N),\subseteq)$ in case $|N|=4$:\footnote{Recall that we apply the notational convention from Definition~\ref{def.global-CI}.}
\begin{itemize}
\item[\fbox{I.}] $\{\,\, [ijk,\ell|\emptyset],\, [ij,k|\emptyset],\,  [ij,k|\ell] \,\,\}$,
\item[\fbox{II.}] $\{\,\, [ijk,\ell|\emptyset],\, ij|k, ij|k\ell,\, ik|j, ik|j\ell,\, jk|i, jk|i\ell \,\,\}$,
\item[\fbox{III.}] $\{\,\, [ij,k|\ell],\,  [ij,\ell|k],\, ij|k, ij|\ell, ij|k\ell,\, ik|j,\, i\ell|j,\,
jk|i,\, j\ell|i,\, k\ell|i, k\ell|j, k\ell|ij \,\,\}$,
\item[\fbox{IV.}] $\{\,\, [ijk,\ell|\emptyset],\, ij|\emptyset, ij|\ell,\, ik|\emptyset, ik|\ell,\, jk|\emptyset, jk|i\ell \,\,\}$,
\item[\fbox{V.}] $\{\,\, [ij,k|\emptyset],\,  [ij,\ell|\emptyset],\, ij|\emptyset, ij|k, ij|\ell,\, ik|\ell,\, i\ell|k,\,
jk|\ell,\, j\ell|k,\, k\ell|\emptyset, k\ell|i, k\ell|j \,\,\}$,
\item[\fbox{VI.}] $\{\,\, [ij,k|\ell],\,  [ij,\ell|k],\, ij|\emptyset,\, ik|\emptyset,\, i\ell|\emptyset,\,
jk|\emptyset,\, j\ell|\emptyset,\, k\ell|ij \,\,\}$,
\item[\fbox{VII.}] $\{\,\, ij|\emptyset, ij|k\ell,\, ik|\emptyset, ik|j\ell,\, i\ell|\emptyset, i\ell|jk,\,
jk|\emptyset, jk|i\ell,\, j\ell|\emptyset, j\ell|ik,\, k\ell|\emptyset, k\ell|ij \,\,\}$,
\item[\fbox{VIII.}] $\{\,\, [ij,k|\ell],\, ij|\emptyset, ij|\ell, ij|k\ell,\, ik|\emptyset,\, i\ell|jk,\,
jk|\emptyset,\, j\ell|ik,\, k\ell|ij \,\,\}$,
\item[\fbox{IX.}] $\{\,\, [ij,k|\emptyset],\, ij|\emptyset, ij|k, ij|k\ell,\, ik|j\ell,\, i\ell|\emptyset,\,
jk|i\ell,\, j\ell|\emptyset,\, k\ell|\emptyset \,\,\}$,
\item[\fbox{X.}] $\{\,\, ij|\emptyset, ij|k, ij|\ell,\, ik|\ell,\, jk|\ell,\, k\ell|i, k\ell|j, k\ell|ij \,\,\}$,
\item[\fbox{XI.}] $\{\,\, ij|\emptyset, ij|k, ij|\ell,\, ik|\ell,\, j\ell|k,\, k\ell|i, k\ell|j, k\ell|ij \,\,\}$,
\item[\fbox{XII.}] $\{\,\, ij|\emptyset, ij|k, ij|\ell,\, ik|\ell,\, i\ell|k,\, jk|\ell, k\ell|j, k\ell|ij \,\,\}$,
\item[\fbox{XIII.}]  $\{\,\, ij|\emptyset, ij|k, ij|\ell,\, ik|\ell,\, i\ell|k,\, jk|\ell,\, j\ell|k,\, k\ell|ij \,\,\}$.
\end{itemize}
The types \fbox{I.}--\fbox{IX.} involve all
31 coatoms of the lattice. Their notation
follows the classification used in
\cite[Appendix B]{Stu21}. The actual
numbers of involved elementary CI statements are as follows:
\begin{itemize}[leftmargin=2em, rightmargin=1em]
\item \fbox{I.} consists of 20 CI statements,
\item \fbox{II.}--\fbox{V.} consists of 18 CI statements,
\item \fbox{VI.} consists of 14 CI statements,
\item \fbox{VII.}--\fbox{IX.} consists  of 12 CI statements, and
\item \fbox{X.}--\fbox{XIII.} consists of 8 CI statements.
\end{itemize}

\end{document}